\documentclass[11pt]{amsart}

\usepackage{
  amsmath,
  amsfonts,
  amssymb,
  amsthm,
  amscd,
  graphicx,
  comment,      
  etoolbox,     
  mathtools,    
  enumitem,
  mathdots,     
  stmaryrd,     
  booktabs,     
  stackrel,     
}
\usepackage[usenames,dvipsnames]{xcolor}

\usepackage[scaled=0.88]{beraserif} 
\usepackage[scaled=0.85]{berasans}  
\usepackage[scaled=0.84]{beramono}  
\usepackage[T1]{fontenc}
\usepackage{textcomp}               
\usepackage[sc]{mathpazo}           
\usepackage[T1,small,euler-digits,euler-hat-accent]{eulervm}  
\usepackage{mathrsfs}               

\usepackage[colorlinks=true, pdfstartview=FitV, linkcolor=blue, citecolor=blue, urlcolor=blue, breaklinks=true]{hyperref}

\usepackage[capitalize]{cleveref}   

\usepackage{tikz}
\usetikzlibrary{arrows.meta}
\usetikzlibrary{decorations.markings}

\newcommand\redcircle[1]{\filldraw[fill=white, draw=red] #1 circle (2pt)}
\newcommand\bluedot[1]{\filldraw[blue] #1 circle (2pt)}
\newcommand\greensquare[1]{\filldraw[green,xshift=-2pt,yshift=-2pt] #1 rectangle ++(4pt,4pt)}
\newcommand\opensquare[1]{\filldraw[color=violet,fill=white,xshift=-2pt,yshift=-2pt] #1 rectangle ++(4pt,4pt)}

\newcommand\squarelabel[1]{$\scriptstyle{#1}$}
\newcommand\dotlabel[1]{$\scriptstyle{#1}$}

\newcommand\circled[1]{\tikz[baseline=(char.base)]{
            \node[shape=circle,draw,inner sep=2pt] (char) {#1};}}

\newcommand\cbubble[2]{
  \begin{tikzpicture}[anchorbase]
    \draw[->] (0,0.3) arc (90:-270:0.3);
    \bluedot{(0.3,0)} node[anchor=west,color=black] {#1};
    \redcircle{(-0.3,0)} node[anchor=east,color=black] {#2};
  \end{tikzpicture}
}

\newcommand\ccbubble[2]{
  \begin{tikzpicture}[anchorbase]
    \draw[->] (0,0.3) arc (90:450:0.3);
    \bluedot{(-0.3,0)} node[anchor=east,color=black] {#1};
    \redcircle{(0.3,0)} node[anchor=west,color=black] {#2};
  \end{tikzpicture}
}

\tikzset{anchorbase/.style={>=To,baseline={([yshift=-0.5ex]current bounding box.center)}}}

\tikzset{->-/.style={decoration={
  markings,
  mark=at position #1 with {\arrow{>}}},postaction={decorate}}}
\tikzset{-<-/.style={decoration={
  markings,
  mark=at position #1 with {\arrow{<}}},postaction={decorate}}}

\crefname{eg}{Example}{Examples}
\crefname{lem}{Lemma}{Lemmas}
\crefname{theo}{Theorem}{Theorems}
\crefname{equation}{}{}
\crefname{enumi}{}{}

\newcommand\Z{\mathbb{Z}}

\newcommand\kk{\Bbbk}

\newcommand\cA{\mathcal{A}}

\newcommand\cH{\mathcal{H}}
\newcommand\Heis{\mathcal{H}\mathit{eis}}

\newcommand\bC{\mathbold{C}}
\newcommand\bsf{\mathbold{f}}

\newcommand\one{\mathbb{1}}

\newcommand{\rh}{\mathrm{Heis}}
\newcommand\tr{\mathrm{tr}}
\newcommand\op{\mathrm{op}}              

\newcommand\sQ{\mathsf{Q}}

\def\chk#1{#1^{\smash{\scalebox{.8}[1.4]{\rotatebox{90}{\textnormal{\guilsinglleft}}}}}} 

\DeclareMathOperator{\End}{End}
\DeclareMathOperator{\END}{END}
\DeclareMathOperator{\grdim}{grdim} 
\DeclareMathOperator{\Hom}{Hom}
\DeclareMathOperator{\HOM}{HOM}
\DeclareMathOperator{\id}{id}
\DeclareMathOperator{\Ind}{Ind}
\DeclareMathOperator{\Kar}{Kar}     
\DeclareMathOperator{\Res}{Res}

\newtheorem{theo}{Theorem}[section]
\newtheorem{prop}[theo]{Proposition}
\newtheorem{lem}[theo]{Lemma}
\newtheorem*{lem*}{Lemma}
\newtheorem{cor}[theo]{Corollary}

\theoremstyle{definition}
\newtheorem{defin}[theo]{Definition}

\allowdisplaybreaks

\setenumerate[1]{label=(\alph*)}  
\setenumerate[2]{label=(\roman*)}

\newtoggle{comments}
\newtoggle{details}
\newtoggle{detailsnote}

\toggletrue{detailsnote}   

\iftoggle{comments}{%
  \newcommand{\comments}[1]{
    \ \\
    {\color{red}
      \textbf{AS:} #1
    }
    \\
    }
}{%
  \newcommand{\comments}[1]{}
}

\iftoggle{details}{%
  \newcommand{\details}[1]{
      \ \\
      {\color{OliveGreen}
        \textbf{Details:} #1
      }
      \\
  }
}{%
  \newcommand{\details}[1]{}
}

\begin{document}
%

\title{Frobenius Heisenberg categorification}

\author{Alistair Savage}
\address{
  Department of Mathematics and Statistics \\
  University of Ottawa
}
\urladdr{\href{https://alistairsavage.ca}{alistairsavage.ca}, \textrm{\textit{ORCiD}:} \href{https://orcid.org/0000-0002-2859-0239}{orcid.org/0000-0002-2859-0239}}
\email{alistair.savage@uottawa.ca}

\thanks{This research was supported by Discovery Grant RGPIN-2017-03854 from the Natural Sciences and Engineering Research Council of Canada.}

\begin{abstract}
  We associate a graded monoidal supercategory $\Heis_{F,k}$ to every graded Frobenius superalgebra $F$ and integer $k$.  These categories, which categorify a broad range of lattice Heisenberg algebras, recover many previously defined Heisenberg categories as special cases.  In this way, the categories $\Heis_{F,k}$ serve as a unifying and generalizing framework for Heisenberg categorification.  Even in the case of previously defined Heisenberg categories, we obtain new, more efficient, presentations of these categories, based on an approach of Brundan.  When $k=0$, our construction yields new versions of the affine oriented Brauer category depending on a graded Frobenius superalgebra.
\end{abstract}


\maketitle
\thispagestyle{empty}


\section{Introduction}

In \cite{Kho14}, Khovanov developed a graphical calculus for the induction and restriction functors arising from the tower of algebras coming from the symmetric groups.  He showed that the Grothendieck ring of the resulting monoidal category contains an infinite-dimensional Heisenberg algebra and conjectured that the two are isomorphic.  Since Khovanov's original work, his construction has been generalized to $q$-deformations \cite{LS13,BSW2}, to categories depending on a graded Frobenius superalgebra \cite{CL12,RS17}, and to higher level \cite{MS17}.  In addition, a generalization of Khovanov's conjecture \cite[Conj.~4.5]{MS17} has recently been proved in \cite[Th.~1.1]{BSW1}.

In \cite{Bru17}, Brundan introduced a new approach to Heisenberg categorification, proving that the higher level Heisenberg categories of \cite{MS17}, which include Khovanov's original category, can be defined using a smaller set of relations, including an ``inversion relation''.  This approach also shows that the affine oriented Brauer category of \cite{BCNR17} can be viewed as the level zero Heisenberg category.

In the current paper, we associate a graded monoidal supercategory $\Heis_{F,k}$ to each graded Frobenius superalgebra $F$ and integer $k$.  The Heisenberg supercategories $\Heis_{F,k}$ can simultaneously be viewed as graded Frobenius superalgebra deformations of the categories of \cite{MS17}, and as extensions of the graded monoidal supercategories of \cite{RS17} to higher level.  When $k=0$, we also obtain new Frobenius deformations of affine oriented Brauer categories.  In this way, the categories introduced in the current paper unify and generalize these previous constructions.  Our approach is inspired by the inversion relation method of \cite{Bru17}.  As a consequence, even when we specialize to the setting of \cite{RS17}, which corresponds to the choice $k=-1$, we obtain two new presentations of the Heisenberg categories defined there.  Specializing further, we obtain higher level versions of, as well as new presentations of, the categories introduced in \cite{CL12}, which are related to the geometry of the Hilbert scheme.

Fix a commutative ground ring $\kk$.  Throughout the paper, the term \emph{graded} will mean $\Z$-graded.  Let $F$ be a nonnegatively graded Frobenius superalgebra with homogeneous basis $B$ (thus, we assume $F$ is free as a $\kk$-module), Nakayama automorphism $\psi$, even trace map $\tr$, and top degree $\Delta$.  By definition, the map
\[
  F \to \Hom_\kk(F,\kk),\quad f \mapsto \left( g \mapsto (-1)^{\bar f \bar g} \tr(gf) \right),
\]
is an isomorphism, and
\begin{equation} \label{eq:Nakayama-def}
  \tr(fg)
  = (-1)^{\bar f \bar g} \tr(g \psi(f))
  = (-1)^{\bar f} \tr(g \psi(f))
  = (-1)^{\bar g} \tr(g \psi(f))
  \quad f,g \in F,
\end{equation}
where $\bar f$ denotes the parity of a homogeneous element $f \in F$.  Here and throughout the paper, when an equation involves the parity of elements, it is understood that we extended it by linearity to non-homogeneous elements.  The degree of a homogeneous element $f \in F$ will be denoted by $|f|$.  We let $\{\chk{b} : b \in B\}$ denote the left dual basis, so that
\begin{equation} \label{eq:dual-basis-def}
  \tr (\chk{a} b) = \delta_{a,b},\quad a,b \in B.
\end{equation}
Fix $k \in \Z$.  We refer the reader to \cite{BE17} for a treatment of monoidal supercategories and the diagramatic conventions used below.

\begin{defin} \label{def:H}
  The supercategory $\Heis_{F,k}$ is the strict $\kk$-linear graded monoidal supercategory defined as follows. The objects are generated by $\sQ_+$ and $\sQ_-$, and we use juxaposition to denote tensor product. The morphisms of $\Heis_{F,k}$ are generated by
  \begin{gather*}
    x \colon \sQ_+ \to \sQ_+,\
    s \colon \sQ_+ \sQ_+ \to \sQ_+ \sQ_+,\
    c \colon \one \to \sQ_- \sQ_+,\
    d \colon \sQ_+ \sQ_- \to \one,\
    \beta_f \colon \sQ_+ \to \sQ_+,\ f \in F,
    \\
    |x| = \Delta,\ \bar x = 0,\
    |s| = 0,\ \bar s = 0,\
    |c| = 0,\ \bar c = 0,\
    |d| = 0,\ \bar d = 0,\
    |\beta_f| = |f|,\
    \bar{\beta}_f = \bar f,\ f \in F,
  \end{gather*}
  subject to certain relations.  Using the usual string calculus for strict monoidal (super)categories, we depict the generating morphisms by the diagrams
  \[
    x =
    \begin{tikzpicture}[anchorbase]
      \draw[->] (0,0) -- (0,0.6);
      \redcircle{(0,0.3)};
    \end{tikzpicture}
    \ ,\quad
    s =
    \begin{tikzpicture}[anchorbase]
      \draw [->](0,0) -- (0.6,0.6);
      \draw [->](0.6,0) -- (0,0.6);
    \end{tikzpicture}
    \ ,\quad
    c =
    \begin{tikzpicture}[anchorbase]
      \draw[->] (0,.2) -- (0,0) arc (180:360:.3) -- (.6,.2);
    \end{tikzpicture}
    \ ,\quad
    d =
    \begin{tikzpicture}[anchorbase]
      \draw[->] (0,-.2) -- (0,0) arc (180:0:.3) -- (.6,-.2);
    \end{tikzpicture}
    \ ,\quad
    \beta_f =
    \begin{tikzpicture}[anchorbase]
      \draw[->] (0,0) -- (0,0.6);
      \bluedot{(0,0.3)} node[anchor=west, color=black] {$f$};
    \end{tikzpicture}
    \ ,\ f \in F.
  \]
  We refer to the decoration representing $x$ as a \emph{dot} and the decorations representing $\beta_f$, $f \in F$, as \emph{tokens}.  The identity morphisms of $\sQ_+$ and $\sQ_-$ are denoted by $\uparrow$ and $\downarrow$, respectively.  For $n \ge 1$, we denote the $n$-th power $x^n$ of $x$ by labelling the dot with the exponent $n$:
  \[
    x^n =
    \begin{tikzpicture}[anchorbase]
      \draw[->] (0,0) -- (0,1);
      \redcircle{(0,0.5)} node[anchor=west, color=black] {$n$};
    \end{tikzpicture}
  \]
  We also define
  \begin{equation} \label{eq:t-def}
    t \colon \sQ_+ \sQ_- \to \sQ_- \sQ_+,\quad
    t =
    \begin{tikzpicture}[anchorbase]
      \draw [->](0,0) -- (0.6,0.6);
      \draw [<-](0.6,0) -- (0,0.6);
    \end{tikzpicture}
    \ :=\
    \begin{tikzpicture}[anchorbase,scale=0.6]
      \draw[->] (0.3,0) -- (-0.3,1);
      \draw[->] (-0.75,1) -- (-0.75,0.5) .. controls (-0.75,0.2) and (-0.5,0) .. (0,0.5) .. controls (0.5,1) and (0.75,0.8) .. (0.75,0.5) -- (0.75,0);
    \end{tikzpicture}
    \ .
  \end{equation}
  We impose three sets of relations:
  \begin{enumerate}[wide]
    \item \emph{Affine wreath product algebra relations}:  We have a homomorphism of graded superalgebras
      \begin{equation} \label{rel:token-homom}
        F \to \End \sQ_+,\quad f \mapsto \beta_f,
      \end{equation}
      so that, in particular,
      \begin{equation} \label{rel:token-colide-up}
        \begin{tikzpicture}[anchorbase]
          \draw[->] (0,0) -- (0,1);
          \bluedot{(0,0.35)} node[anchor=east,color=black] {$g$};
          \bluedot{(0,0.7)} node[anchor=east,color=black] {$f$};
        \end{tikzpicture}
        \ = \
        \begin{tikzpicture}[anchorbase]
          \draw[->] (0,0) -- (0,1);
          \bluedot{(0,0.5)} node[anchor=west,color=black] {$fg$};
        \end{tikzpicture}
        ,\quad f,g \in F.
      \end{equation}

      Furthermore, the following relations are satisfied for all $f \in F$:

      \noindent\begin{minipage}{0.33\linewidth}
        \begin{equation} \label{rel:braid-up}
          \begin{tikzpicture}[anchorbase]
            \draw[->] (0,0) -- (1,1);
            \draw[->] (1,0) -- (0,1);
            \draw[->] (0.5,0) .. controls (0,0.5) .. (0.5,1);
          \end{tikzpicture}
          \ =\
          \begin{tikzpicture}[anchorbase]
            \draw[->] (0,0) -- (1,1);
            \draw[->] (1,0) -- (0,1);
            \draw[->] (0.5,0) .. controls (1,0.5) .. (0.5,1);
          \end{tikzpicture}\ ,
        \end{equation}
      \end{minipage}%
      \begin{minipage}{0.33\linewidth}
        \begin{equation} \label{rel:doublecross-up}
          \begin{tikzpicture}[anchorbase]
            \draw[->] (0,0) .. controls (0.5,0.5) .. (0,1);
            \draw[->] (0.5,0) .. controls (0,0.5) .. (0.5,1);
          \end{tikzpicture}
          \ =\
          \begin{tikzpicture}[anchorbase]
            \draw[->] (0,0) --(0,1);
            \draw[->] (0.5,0) -- (0.5,1);
          \end{tikzpicture}\ ,
        \end{equation}
      \end{minipage}
      \begin{minipage}{0.33\linewidth}
        \begin{equation} \label{rel:dot-token-up-slide}
          \begin{tikzpicture}[anchorbase]
            \draw[->] (0,0) -- (0,1);
            \redcircle{(0,0.3)};
            \bluedot{(0,0.6)} node[anchor=east, color=black] {$f$};
          \end{tikzpicture}
          \ =\
          \begin{tikzpicture}[anchorbase]
            \draw[->] (0,0) -- (0,1);
            \redcircle{(0,0.6)};
            \bluedot{(0,0.3)} node[anchor=west, color=black] {$\psi(f)$};
          \end{tikzpicture}\ ,
        \end{equation}
      \end{minipage}\par\vspace{\belowdisplayskip}

      \noindent\begin{minipage}{0.5\linewidth}
        \begin{equation} \label{rel:tokenslide-up-right}
          \begin{tikzpicture}[anchorbase]
            \draw[->] (0,0) -- (1,1);
            \draw[->] (1,0) -- (0,1);
            \bluedot{(.25,.25)} node [anchor=south east, color=black] {$f$};
          \end{tikzpicture}
          \ =\
          \begin{tikzpicture}[anchorbase]
            \draw[->](0,0) -- (1,1);
            \draw[->](1,0) -- (0,1);
            \bluedot{(0.75,.75)} node [anchor=north west, color=black] {$f$};
          \end{tikzpicture}\ ,
        \end{equation}
      \end{minipage}%
      \begin{minipage}{0.5\linewidth}
        \begin{equation} \label{rel:dotslide1}
          \begin{tikzpicture}[anchorbase]
            \draw[->] (0,0) -- (1,1);
            \draw[->] (1,0) -- (0,1);
            \redcircle{(0.25,.75)};
          \end{tikzpicture}
          \ -\
          \begin{tikzpicture}[anchorbase]
            \draw[->] (0,0) -- (1,1);
            \draw[->] (1,0) -- (0,1);
            \redcircle{(.75,.25)};
          \end{tikzpicture}
          \ =\
          \begin{tikzpicture}[anchorbase]
            \draw[->] (0,0) -- (0,1);
            \draw[->] (0.5,0) -- (0.5,1);
            \bluedot{(0,0.3)} node[anchor=east, color=black] {$\chk{b}$};
            \bluedot{(0.5,0.6)} node[anchor=west, color=black] {$b$};
          \end{tikzpicture}\ .
        \end{equation}
      \end{minipage}\par\vspace{\belowdisplayskip}

      \noindent In \cref{rel:dotslide1} and throughout the paper we adopt the convention that, \emph{whenever an expression contains the symbols $b$ and $\chk{b}$ (or $a$ and $\chk{a}$, etc.), there is an implicit sum over $b \in B$ (or $a \in B$, etc.)}.  For example,
      \[
        \begin{tikzpicture}[anchorbase]
          \draw[->] (0,0) -- (0,1);
          \draw[->] (0.5,0) -- (0.5,1);
          \bluedot{(0,0.3)} node[anchor=east, color=black] {$\chk{b}$};
          \bluedot{(0.5,0.6)} node[anchor=west, color=black] {$b$};
        \end{tikzpicture}
        =
        \sum_{b \in B}
        \begin{tikzpicture}[anchorbase]
          \draw[->] (0,0) -- (0,1);
          \draw[->] (0.5,0) -- (0.5,1);
          \bluedot{(0,0.3)} node[anchor=east, color=black] {$\chk{b}$};
          \bluedot{(0.5,0.6)} node[anchor=west, color=black] {$b$};
        \end{tikzpicture}
        \qquad \text{and} \qquad
        b a \otimes \chk{a} = \sum_{a \in B} b a \otimes \chk{a}
        \qquad \text{by convention}.
      \]
      (In a few instances we will include the explicit sum where there is some possibility for confusion.)  It follows from the above relations that we also have the relations:

      \noindent\begin{minipage}{0.5\linewidth}
        \begin{equation} \label{rel:tokenslide-up-left}
          \begin{tikzpicture}[anchorbase]
            \draw[->] (0,0) -- (1,1);
            \draw[->] (1,0) -- (0,1);
            \bluedot{(.75,.25)} node [anchor=south west, color=black] {$f$};
          \end{tikzpicture}
          \ =\
          \begin{tikzpicture}[anchorbase]
            \draw[->] (0,0) -- (1,1);
            \draw[->] (1,0) -- (0,1);
            \bluedot{(0.25,.75)} node [anchor=north east, color=black] {$f$};
          \end{tikzpicture}\ ,
        \end{equation}
      \end{minipage}%
      \begin{minipage}{0.5\linewidth}
        \begin{equation} \label{rel:dotslide2}
          \begin{tikzpicture}[anchorbase]
            \draw[->] (0,0) -- (1,1);
            \draw[->] (1,0) -- (0,1);
            \redcircle{(0.25,.25)};
          \end{tikzpicture}
          \ -\
          \begin{tikzpicture}[anchorbase]
            \draw[->] (0,0) -- (1,1);
            \draw[->] (1,0) -- (0,1);
            \redcircle{(.75,.75)};
          \end{tikzpicture}
          \ =\
          \begin{tikzpicture}[anchorbase]
            \draw[->] (0,0) -- (0,1);
            \draw[->] (0.5,0) -- (0.5,1);
            \bluedot{(0,0.5)} node[anchor=east, color=black] {$b$};
            \bluedot{(0.5,0.5)} node[anchor=west, color=black] {$\chk{b}$};
          \end{tikzpicture}\ .
        \end{equation}
      \end{minipage}\par\vspace{\belowdisplayskip}

      \noindent Recall that when morphisms appear at the same height, as is the case for the two tokens on the right side of \cref{rel:dotslide2}, one obtains the same morphism by slightly increasing the height of the leftmost one.  See, for example, \cite[(1.3)]{BE17}.

    \item \emph{Right adjunction relations}:  We impose the following relations:

      \noindent\begin{minipage}{0.5\linewidth}
        \begin{equation} \label{rel:right-adjunction-up}
          \begin{tikzpicture}[anchorbase]
            \draw[->] (0,0) -- (0,0.6) arc(180:0:0.2) -- (0.4,0.4) arc(180:360:0.2) -- (0.8,1);
          \end{tikzpicture}
          \ =\
          \begin{tikzpicture}[anchorbase]
            \draw[->] (0,0) -- (0,1);
          \end{tikzpicture}\ ,
        \end{equation}
      \end{minipage}%
      \begin{minipage}{0.5\linewidth}
        \begin{equation} \label{rel:right-adjunction-down}
          \begin{tikzpicture}[anchorbase]
            \draw[->] (0,1) -- (0,0.4) arc(180:360:0.2) -- (0.4,0.6) arc(180:0:0.2) -- (0.8,0);
          \end{tikzpicture}
          \ =\
          \begin{tikzpicture}[anchorbase]
            \draw[<-] (0,0) -- (0,1);
          \end{tikzpicture}\ .
        \end{equation}
      \end{minipage}\par\vspace{\belowdisplayskip}

    \item \emph{Inversion relation}: The following matrix of morphisms is an isomorphism in the additive envelope of $\Heis_{F,k}$:
      \begin{gather} \label{eq:inversion-relation-pos-l}
        \left[
          \begin{tikzpicture}[anchorbase]
            \draw [->](0,0) -- (0.6,0.6);
            \draw [<-](0.6,0) -- (0,0.6);
          \end{tikzpicture}
          \quad
          \begin{tikzpicture}[anchorbase]
            \draw[->] (0,0) -- (0,0.7) arc (180:0:.3) -- (0.6,0);
            \redcircle{(0,0.6)} node[anchor=west,color=black] {$r$};
            \bluedot{(0,0.2)} node[anchor=west,color=black] {$\chk{b}$};
          \end{tikzpicture}\ ,\
          0 \le r \le k-1,\ b \in B
        \right]^T
        \ \colon \sQ_+ \sQ_- \to \sQ_- \sQ_+ \oplus \one ^{\oplus k \dim F} \quad \text{if } k \ge 0,
        \\ \label{eq:inversion-relation-neg-l}
        \left[
          \begin{tikzpicture}[anchorbase]
            \draw [->](0,0) -- (0.6,0.6);
            \draw [<-](0.6,0) -- (0,0.6);
          \end{tikzpicture}
          \quad
          \begin{tikzpicture}[anchorbase]
            \draw[->] (0,1) -- (0,0.3) arc (180:360:.3) -- (0.6,1);
            \redcircle{(0.6,0.7)} node[anchor=east,color=black] {$r$};
            \bluedot{(0.6,0.3)} node[anchor=east,color=black] {$\chk{b}$};
          \end{tikzpicture}\ ,\
          0 \le r \le -k-1,\ b \in B
        \right]
        \ \colon \sQ_+ \sQ_- \oplus \one^{\oplus (-k \dim F)} \to \sQ_- \sQ_+ \quad \text{if } k < 0.
      \end{gather}
      (The matrix \cref{eq:inversion-relation-pos-l} is of size $(1 + k \dim F) \times 1$, while the matrix \cref{eq:inversion-relation-neg-l} is of size $1 \times (1 + k \dim F)$.)  Note that these conditions are independent of the choice of basis $B$ of $F$.
  \end{enumerate}
\end{defin}

In the special case $k=0$, the inversion relation means that there is another generating morphism
\[
  t' =
  \begin{tikzpicture}[anchorbase]
    \draw [<-](0,0) -- (0.6,0.6);
    \draw [->](0.6,0) -- (0,0.6);
  \end{tikzpicture}
  \colon \sQ_- \sQ_+ \to \sQ_+ \sQ_-,
\]
that is inverse to $t$.  Thus we have
\[
  \begin{tikzpicture}[anchorbase]
    \draw[->] (0,0) .. controls (0.5,0.5) .. (0,1);
    \draw[<-] (0.5,0) .. controls (0,0.5) .. (0.5,1);
  \end{tikzpicture}
  \ =\
  \begin{tikzpicture}[anchorbase]
    \draw[->] (0,0) -- (0,1);
    \draw[<-] (0.5,0) -- (0.5,1);
  \end{tikzpicture}
  ,\qquad
  \begin{tikzpicture}[anchorbase]
    \draw[<-] (0,0) .. controls (0.5,0.5) .. (0,1);
    \draw[->] (0.5,0) .. controls (0,0.5) .. (0.5,1);
  \end{tikzpicture}
  \ =\
  \begin{tikzpicture}[anchorbase]
    \draw[<-] (0,0) -- (0,1);
    \draw[->] (0.5,0) -- (0.5,1);
  \end{tikzpicture}
  \ .
\]
If $F = \kk$ and we reflect diagrams in a vertical axis, we obtain precisely the affine oriented Brauer category of \cite{BCNR17}.  Thus, we can view $\Heis_{F,0}$ as a graded Frobenius superalgebra deformation of the (reversed) affine oriented Brauer category.

In the case $k \ne 0$, the inversion relation is substantially more intricate.  We analyze its consequences in \cref{sec:inversion}.  The main result of that analysis is contained in the next two theorems.  The proofs of these and other theorems stated in this introduction will be given in \cref{sec:proofs}.  In the special case that $F=\kk$, \cref{theo:alternate-presentation,theo:relations} are due to Brundan \cite[Th.~1.2, Th.~1.3]{Bru17}.

We adopt the following convention for computing determinants of matrices whose entries lie in a superalgebra.  For $1 \times 1$ matrices, we define $\det(a) = a$.  Then, for $n > 1$, we recursively define
\begin{equation} \label{eq:determinant-convention}
  \det (a_{i,j})_{i,j=1}^n
  = \sum_{s=1}^n (-1)^{s+1} a_{s,1} \det A_{s,1},
\end{equation}
where $A_{s,1}$ is the $(n-1) \times (n-1)$ matrix obtained from the matrix $(a_{i,j})_{i,j=1}^n$ by deleting the $s$-th row and first column.  In other words, we compute determinants by recursively expanding along the first column.

\begin{theo} \label{theo:alternate-presentation}
  There are unique even morphisms $c' \colon \one \to \sQ_+ \sQ_-$ and $d' \colon \sQ_- \sQ_+ \to \one$ in $\Heis_{F,k}$, drawn as
  \[
    c' =
    \begin{tikzpicture}[anchorbase]
      \draw[<-] (0,.2) -- (0,0) arc (180:360:.3) -- (.6,.2);
    \end{tikzpicture}
    \ ,\qquad
    d' =
    \begin{tikzpicture}[anchorbase]
      \draw[<-] (0,-.2) -- (0,0) arc (180:0:.3) -- (.6,-.2);
    \end{tikzpicture}
    \ ,
  \]
  with $|c'| = - k \Delta$ and $|d'| = k \Delta$, such that the following relations hold:
  \begin{equation} \label{theo-eq:doublecross-up-down}
    \begin{tikzpicture}[anchorbase]
      \draw[->] (0,0) .. controls (0.5,0.5) .. (0,1);
      \draw[<-] (0.5,0) .. controls (0,0.5) .. (0.5,1);
    \end{tikzpicture}
    \ =\
    \begin{tikzpicture}[anchorbase]
      \draw[->] (0,0) -- (0,1);
      \draw[<-] (0.5,0) -- (0.5,1);
    \end{tikzpicture}
    \ + \sum_{r,s \ge 0}
    \begin{tikzpicture}[anchorbase]
      \draw[->] (0,0) -- (0,0.7) arc (180:0:0.3) -- (0.6,0);
      \draw[<-] (0,2.1) -- (0,1.8) arc (180:360:0.3) -- (0.6,2.1);
      \redcircle{(0,0.6)} node[anchor=west,color=black] {$r$};
      \bluedot{(0,0.2)} node[anchor=west,color=black] {$\chk{b}$};
      \redcircle{(0.6,1.8)} node[anchor=west,color=black] {$s$};
      \bluedot{(0,1.8)} node[anchor=east,color=black] {$a$};
      \draw[->] (1,1.55) arc(90:450:0.3);
      \bluedot{(1.3,1.25)} node[anchor=west,color=black] {$\chk{a} b$};
      \redcircle{(0.7,1.25)} node[anchor=east,color=black] {\dotlabel{-r-s-2}};
    \end{tikzpicture}
    \left( =
    \begin{tikzpicture}[anchorbase]
      \draw[->] (0,0) -- (0,1);
      \draw[<-] (0.5,0) -- (0.5,1);
    \end{tikzpicture}
    \ + \delta_{k,1}
    \begin{tikzpicture}[anchorbase]
      \draw[<-] (0,1.1) -- (0,0.9) arc(180:360:0.3) -- (0.6,1.1);
      \draw[->] (0,-0.1) -- (0,0.1) arc(180:0:0.3) -- (0.6,-0.1);
      \bluedot{(0,0.1)} node[anchor=west,color=black] {$\chk{b}$};
      \bluedot{(0,0.9)} node[anchor=west,color=black] {$b$};
    \end{tikzpicture}
    \text{ if } k \le 1 \right),
  \end{equation}
  \begin{equation} \label{theo-eq:doublecross-down-up}
    \begin{tikzpicture}[anchorbase]
      \draw[<-] (0,0) .. controls (0.5,0.5) .. (0,1);
      \draw[->] (0.5,0) .. controls (0,0.5) .. (0.5,1);
    \end{tikzpicture}
    \ =\
    \begin{tikzpicture}[anchorbase]
      \draw[<-] (0,0) -- (0,1);
      \draw[->] (0.5,0) -- (0.5,1);
    \end{tikzpicture}
    \ + \sum_{r,s \ge 0} (-1)^{\bar a \bar b + \bar a + \bar b}
    \begin{tikzpicture}[anchorbase]
      \draw[->] (0,2) -- (0,1.3) arc (180:360:0.3) -- (0.6,2);
      \draw[<-] (0,-0.1) -- (0,0.2) arc (180:0:0.3) -- (0.6,-0.1);
      \redcircle{(0.6,1.7)} node[anchor=east,color=black] {$r$};
      \bluedot{(0.6,1.4)} node[anchor=west,color=black] {$\chk{b}$};
      \redcircle{(0.6,0.2)} node[anchor=west,color=black] {$s$};
      \bluedot{(0,0.2)} node[anchor=east,color=black] {$a$};
      \draw[->] (1,1.1) arc(90:-270:0.3);
      \bluedot{(1.3,0.8)} node[anchor=west,color=black] {$\chk{a}b$};
      \redcircle{(0.7,0.8)} node[anchor=east,color=black] {\dotlabel{-r-s-2}};
    \end{tikzpicture}
    \left( =
    \begin{tikzpicture}[anchorbase]
      \draw[<-] (0,0) -- (0,1);
      \draw[->] (0.5,0) -- (0.5,1);
    \end{tikzpicture}
    \ - \delta_{k,-1} (-1)^{\bar b}
    \begin{tikzpicture}[anchorbase]
      \draw[->] (0,1.1) -- (0,0.9) arc(180:360:0.3) -- (0.6,1.1);
      \draw[<-] (0,-0.1) -- (0,0.1) arc(180:0:0.3) -- (0.6,-0.1);
      \bluedot{(0,0.1)} node[anchor=east,color=black] {$b$};
      \bluedot{(0.6,0.9)} node[anchor=east,color=black] {$\chk{b}$};
    \end{tikzpicture}
    \text{ if } k \ge -1 \right),
  \end{equation}

  \noindent\begin{minipage}{0.5\linewidth}
    \begin{equation} \label{theo-eq:right-curl}
      \begin{tikzpicture}[anchorbase]
        \draw[->] (0,-0.75) .. controls (0,0.5) and (0.5,0.5) .. (0.5,0) .. controls (0.5,-0.5) and (0,-0.5) .. (0,0.75);
      \end{tikzpicture}
      = \delta_{k,0}\
      \begin{tikzpicture}[anchorbase]
        \draw[->] (0,-0.75) -- (0,0.75);
      \end{tikzpicture}
      \quad \text{if } k \ge 0,
    \end{equation}
  \end{minipage}%
  \begin{minipage}{0.5\linewidth}
    \begin{equation} \label{theo-eq:clockwise-circ}
      \begin{tikzpicture}[anchorbase]
        \draw[<-] (0,0.3) arc(90:450:0.3);
        \redcircle{(-0.3,0)} node[anchor=east,color=black] {$r$};
        \bluedot{(0.3,0)} node[anchor=west,color=black] {$f$};
      \end{tikzpicture}
      \ = -\delta_{r,k-1} \tr(f)
      \quad \text{if } 0 \le r < k,
    \end{equation}
  \end{minipage}\par\vspace{\belowdisplayskip}

  \noindent\begin{minipage}{0.5\linewidth}
    \begin{equation} \label{theo-eq:left-curl}
      \begin{tikzpicture}[anchorbase]
        \draw[->] (0,-0.75) .. controls (0,0.5) and (-0.5,0.5) .. (-0.5,0) .. controls (-0.5,-0.5) and (0,-0.5) .. (0,0.75);
      \end{tikzpicture}
      = \delta_{k,0}\
      \begin{tikzpicture}[anchorbase]
        \draw[->] (0,-0.75) -- (0,0.75);
      \end{tikzpicture}
      \quad \text{if } k \le 0,
    \end{equation}
  \end{minipage}%
  \begin{minipage}{0.5\linewidth}
    \begin{equation} \label{theo-eq:counterclockwise-circ}
      \begin{tikzpicture}[anchorbase]
        \draw[->] (0,0.3) arc(90:450:0.3);
        \redcircle{(0.3,0)} node[anchor=west,color=black] {$r$};
        \bluedot{(-0.3,0)} node[anchor=east,color=black] {$f$};
      \end{tikzpicture}
      \ = \delta_{r,-k-1} \tr(f)
      \quad \text{if } 0 \le r < -k.
    \end{equation}
  \end{minipage}\par\vspace{\belowdisplayskip}

  Moreover, $\Heis_{F,k}$ can be presented equivalently as the strict $\kk$-linear monoidal supercategory generated by the objects $\sQ_+$, $\sQ_-$, and morphisms $s,x,c,d,c',d'$, and $\beta_f$, $f \in F$, subject only to the relations \cref{rel:token-homom,rel:braid-up,rel:doublecross-up,rel:dot-token-up-slide,rel:tokenslide-up-right,rel:dotslide1,rel:right-adjunction-up,rel:right-adjunction-down,theo-eq:doublecross-up-down,theo-eq:doublecross-down-up,theo-eq:right-curl,theo-eq:clockwise-circ,theo-eq:left-curl,theo-eq:counterclockwise-circ}.
  In the above relations, in addition to the rightward crossing $t$ defined by \cref{eq:t-def}, we have used the left crossing $t' \colon \sQ_- \sQ_+ \to \sQ_+ \sQ_-$ defined by
  \begin{equation} \label{eq:t-def-alt}
    t' =
    \begin{tikzpicture}[anchorbase]
      \draw [<-](0,0) -- (0.6,0.6);
      \draw [->](0.6,0) -- (0,0.6);
    \end{tikzpicture}
    \ :=\
    \begin{tikzpicture}[anchorbase,scale=0.6]
      \draw[<-] (0.3,0) -- (-0.3,-1);
      \draw[<-] (-0.75,-1) -- (-0.75,-0.5) .. controls (-0.75,-0.2) and (-0.5,0) .. (0,-0.5) .. controls (0.5,-1) and (0.75,-0.8) .. (0.75,-0.5) -- (0.75,0);
    \end{tikzpicture}
    \ ,
  \end{equation}
  and the negatively dotted bubbles defined, for $f \in F$, by
  \begin{gather} \label{theo-eq:neg-ccbubble}
    \ccbubble{$f$}{\dotlabel{r-k-1}}
    = \sum_{b_1,\dotsc,b_{r-1} \in B} \det
    \left( \cbubble{$\chk{b}_{j-1}b_j$}{\dotlabel{i-j+k}} \right)_{i,j=1}^r,
    \quad \text{if } r \le k,
    \\ \label{theo-eq:neg-cbubble}
    \cbubble{$f$}{\dotlabel{r+k-1}}
    = (-1)^{r+1} \sum_{b_1,\dotsc,b_{r-1} \in B} \det
    \left( \ccbubble{$\chk{b}_{j-1}b_j$}{\dotlabel{i-j-k}} \right)_{i,j=1}^r,
    \quad \text{if } r \le -k,
  \end{gather}
  where we adopt the convention that $\chk{b}_0 = f$ and $b_r = 1$, and we interpret the determinants as $\tr(f)$ if $r=0$ and as $0$ if $r < 0$.  We have also used dots and tokens on downward strands, as defined by the first equalities in \cref{rel:token-rotation,rel:dot-rotation} below.
\end{theo}

\begin{theo} \label{theo:relations}
  Using the notation from \cref{theo:alternate-presentation}, the following relations are consequences of the defining relations.
  \begin{enumerate}[wide]
    \item \emph{Infinite grassmannian relations}: For $f,g \in F$, we have

      \noindent\begin{minipage}{0.5\linewidth}
        \begin{equation} \label{eq:inf-grass1}
          \begin{tikzpicture}[anchorbase]
            \draw[<-] (0,0.3) arc(90:450:0.3);
            \redcircle{(-0.3,0)} node[anchor=east,color=black] {$r$};
            \bluedot{(0.3,0)} node[anchor=west,color=black] {$f$};
          \end{tikzpicture}
          \ = -\delta_{r,k-1} \tr(f)
          \quad \text{if } r \le k-1,
       \end{equation}
      \end{minipage}%
      \begin{minipage}{0.5\linewidth}
        \begin{equation} \label{eq:inf-grass2}
          \begin{tikzpicture}[anchorbase]
            \draw[->] (0,0.3) arc(90:450:0.3);
            \redcircle{(0.3,0)} node[anchor=west,color=black] {$r$};
            \bluedot{(-0.3,0)} node[anchor=east,color=black] {$f$};
          \end{tikzpicture}
          \ = \delta_{r,-k-1} \tr(f)
          \quad \text{if } r \le -k-1,
        \end{equation}
      \end{minipage}\par\vspace{\belowdisplayskip}

      \begin{equation} \label{eq:inf-grass3}
        \sum_{\substack{r,s \in \Z \\ r + s = t-2}}\
        \begin{tikzpicture}[anchorbase]
          \draw[<-] (0,0.3) arc(90:450:0.3);
          \draw[->] (0,-0.5) arc(90:450:0.3);
          \redcircle{(-0.3,0)} node[anchor=east,color=black] {$r$};
          \bluedot{(0.3,0)} node[anchor=west,color=black] {$fb$};
          \redcircle{(0.3,-0.8)} node[anchor=west,color=black] {$s$};
          \bluedot{(-0.3,-0.8)} node[anchor=east,color=black] {$\chk{b}g$};
        \end{tikzpicture}
        \ = \sum_{\substack{r,s \ge 0 \\ r + s = t}}\
        \begin{tikzpicture}[anchorbase]
          \draw[<-] (0,0.3) arc(90:450:0.3);
          \draw[->] (0,-0.5) arc(90:450:0.3);
          \redcircle{(-0.3,0)} node[anchor=east,color=black] {\dotlabel{r+k-1}};
          \bluedot{(0.3,0)} node[anchor=west,color=black] {$fb$};
          \redcircle{(0.3,-0.8)} node[anchor=west,color=black] {\dotlabel{s-k-1}};
          \bluedot{(-0.3,-0.8)} node[anchor=east,color=black] {$\chk{b}g$};
        \end{tikzpicture}
        \ = - \delta_{t,0} \tr(fg).
      \end{equation}

    \item \emph{Left adjunction}:

      \noindent\begin{minipage}{0.5\linewidth}
        \begin{equation} \label{rel:zigzag-leftdown}
          \begin{tikzpicture}[anchorbase]
            \draw[<-] (0,0) -- (0,0.6) arc(180:0:0.2) -- (0.4,0.4) arc(180:360:0.2) -- (0.8,1);
          \end{tikzpicture}
          \ =\
          \begin{tikzpicture}[anchorbase]
            \draw[<-] (0,0) -- (0,1);
          \end{tikzpicture}\ ,
        \end{equation}
      \end{minipage}%
      \begin{minipage}{0.5\linewidth}
        \begin{equation} \label{rel:zigzag-leftup}
          \begin{tikzpicture}[anchorbase]
            \draw[<-] (0,1) -- (0,0.4) arc(180:360:0.2) -- (0.4,0.6) arc(180:0:0.2) -- (0.8,0);
          \end{tikzpicture}
          \ =\
          \begin{tikzpicture}[anchorbase]
            \draw[->] (0,0) -- (0,1);
          \end{tikzpicture}\ .
        \end{equation}
      \end{minipage}\par\vspace{\belowdisplayskip}

    \item \emph{Rotation relations}: For all $f \in F$,

      \noindent\begin{minipage}{0.4\linewidth}
        \begin{equation} \label{rel:token-rotation}
          \begin{tikzpicture}[anchorbase]
            \draw[<-] (0,0) -- (0,1);
            \bluedot{(0,0.5)} node[anchor=west,color=black] {$f$};
          \end{tikzpicture}
          \ :=\
          \begin{tikzpicture}[anchorbase]
            \draw[->] (0,1) -- (0,0.4) arc(180:360:0.2) -- (0.4,0.6) arc(180:0:0.2) -- (0.8,0);
            \bluedot{(0.4,0.5)} node[anchor=west,color=black] {$f$};
          \end{tikzpicture}
          \ =\
          \begin{tikzpicture}[anchorbase]
            \draw[<-] (0,0) -- (0,0.6) arc(180:0:0.2) -- (0.4,0.4) arc(180:360:0.2) -- (0.8,1);
            \bluedot{(0.4,0.5)};
            \draw (0.3,0.7) node[anchor=south,color=black] {\dotlabel{\psi^k(f)}};
          \end{tikzpicture}
          \ ,
        \end{equation}
      \end{minipage}%
      \begin{minipage}{0.6\linewidth}
        \begin{equation} \label{rel:dot-rotation}
          \begin{tikzpicture}[anchorbase]
            \draw[<-] (0,0) -- (0,1.4);
            \redcircle{(0,0.7)};
          \end{tikzpicture}
          \ :=\
          \begin{tikzpicture}[anchorbase]
            \draw[->] (0,1.2) -- (0,0.4) arc(180:360:0.2) -- (0.4,0.6) arc(180:0:0.2) -- (0.8,-0.2);
            \redcircle{(0.4,0.5)};
          \end{tikzpicture}
          \ =\
          \begin{tikzpicture}[anchorbase]
            \draw[<-] (0,-0.2) -- (0,0.6) arc(180:0:0.2) -- (0.4,0.4) arc(180:360:0.2) -- (0.8,1.2);
            \redcircle{(0.4,0.5)};
          \end{tikzpicture}
          \ - \
          \begin{tikzpicture}[anchorbase]
            \draw[<-] (0,-0.2) -- (0,1.2);
            \draw[<-] (0.7,1) arc(90:450:0.3);
            \bluedot{(0,0.1)} node[anchor=east,color=black] {$\chk{b}$};
            \redcircle{(0.4,0.7)} node[anchor=east,color=black] {$k$};
            \bluedot{(1,0.7)} node[anchor=west,color=black] {\dotlabel{\psi^{-1}(b)-b}};
          \end{tikzpicture}
          \ ,
        \end{equation}
      \end{minipage}\par\vspace{\belowdisplayskip}

      \begin{equation} \label{rel:crossing-rotation}
        \begin{tikzpicture}[anchorbase]
          \draw[<-] (0,0) -- (1,1);
          \draw[<-] (1,0) -- (0,1);
        \end{tikzpicture}
        \ :=\
        \begin{tikzpicture}[anchorbase,scale=0.5]
          \draw[->] (-1.5,1.5) .. controls (-1.5,0.5) and (-1,-1) .. (0,0) .. controls (1,1) and (1.5,-0.5) .. (1.5,-1.5);
          \draw[->] (-2,1.5) .. controls (-2,-2) and (1.5,-1.5) .. (0,0) .. controls (-1.5,1.5) and (2,2) .. (2,-1.5);
        \end{tikzpicture}
        \ =\
        \begin{tikzpicture}[anchorbase,scale=0.5]
          \draw[->] (1.5,1.5) .. controls (1.5,0.5) and (1,-1) .. (0,0) .. controls (-1,1) and (-1.5,-0.5) .. (-1.5,-1.5);
          \draw[->] (2,1.5) .. controls (2,-2) and (-1.5,-1.5) .. (0,0) .. controls (1.5,1.5) and (-2,2) .. (-2,-1.5);
        \end{tikzpicture}
        \ .
      \end{equation}

    \item \emph{Curl relations}: For all $r \ge 0$,

      \noindent\begin{minipage}{0.5\linewidth}
        \begin{equation} \label{rel:left-dotted-curl}
          \begin{tikzpicture}[anchorbase]
            \draw[->] (0,-0.75) .. controls (0,0.5) and (-0.5,0.5) .. (-0.5,0) .. controls (-0.5,-0.5) and (0,-0.5) .. (0,0.75);
            \redcircle{(-0.5,0)} node[anchor=east,color=black] {$r$};
          \end{tikzpicture}
          = \sum_{s \ge 0}
          \begin{tikzpicture}[anchorbase]
            \draw[->] (0,-0.75) -- (0,0.75);
            \draw[->] (-0.75,0.65) arc(90:450:0.3);
            \bluedot{(-0.45,0.35)} node[anchor=west,color=black] {$b$};
            \redcircle{(-1.05,0.35)} node[anchor=east,color=black] {\dotlabel{r-s-1}};
            \redcircle{(0,-0.05)} node[anchor=west,color=black] {$s$};
            \bluedot{(0,-0.4)} node[anchor=west,color=black] {$\chk{b}$};
          \end{tikzpicture}\ ,
        \end{equation}
      \end{minipage}%
      \begin{minipage}{0.5\linewidth}
        \begin{equation} \label{rel:right-dotted-curl}
          \begin{tikzpicture}[anchorbase]
            \draw[->] (0,-0.75) .. controls (0,0.5) and (0.5,0.5) .. (0.5,0) .. controls (0.5,-0.5) and (0,-0.5) .. (0,0.75);
            \redcircle{(0.5,0)} node[anchor=west,color=black] {$r$};
          \end{tikzpicture}
          = - \sum_{s \ge 0}\
          \begin{tikzpicture}[anchorbase]
            \draw[->] (0,-0.75) -- (0,0.75);
            \draw[<-] (0.95,-0.05) arc(90:450:0.3);
            \bluedot{(0.65,-0.35)} node[anchor=east,color=black] {$\chk{b}$};
            \redcircle{(1.25,-0.35)} node[anchor=west,color=black] {\dotlabel{r-s-1}};
            \redcircle{(0,0.05)} node[anchor=west,color=black] {$s$};
            \bluedot{(0,0.4)} node[anchor=west,color=black] {$b$};
          \end{tikzpicture}\ .
        \end{equation}
      \end{minipage}\par\vspace{\belowdisplayskip}

    \item \emph{Bubble slides}:  For all $f \in F$ and $r \ge 0$,
      \begin{equation} \label{rel:clockwise-bubble-slide}
        \begin{tikzpicture}[anchorbase]
          \draw[<-] (0,0.3) arc(90:450:0.3);
          \draw[->] (0.8,-1) -- (0.8,1);
          \redcircle{(0.3,0)} node[anchor=west,color=black] {$r$};
          \bluedot{(-0.3,0)} node[anchor=east,color=black] {$f$};
        \end{tikzpicture}
        \ =\
        \begin{tikzpicture}[anchorbase]
          \draw[<-] (0,0.3) arc(90:450:0.3);
          \draw[->] (-0.8,-1) -- (-0.8,1);
          \redcircle{(0.3,0)} node[anchor=west,color=black] {$r$};
          \bluedot{(-0.3,0)} node[anchor=east,color=black] {$f$};
        \end{tikzpicture}
        \ - \sum_{t \ge 0} \sum_{s=0}^t (-1)^{\bar a \bar b}\
        \begin{tikzpicture}[anchorbase]
          \draw[<-] (0.3,-0.2) arc(90:450:0.3);
          \draw[->] (-0.8,-1) -- (-0.8,1);
          \redcircle{(0.6,-0.5)} node[anchor=west,color=black] {\dotlabel{r-t-2}};
          \bluedot{(0,-0.5)} node[anchor=east,color=black] {$\chk{a}f$};
          \bluedot{(-0.8,0.6)} node[anchor=west,color=black] {$b \psi^{-s}(a) \psi^{-t}(\chk{b})$};
          \redcircle{(-0.8,0.1)} node[anchor=west,color=black] {$t$};
        \end{tikzpicture}
        \ ,
      \end{equation}
      \begin{equation} \label{rel:counterclockwise-bubble-slide}
        \begin{tikzpicture}[anchorbase]
          \draw[->] (0,0.3) arc(90:450:0.3);
          \draw[->] (-0.8,-1) -- (-0.8,1);
          \redcircle{(0.3,0)} node[anchor=west,color=black] {$r$};
          \bluedot{(-0.3,0)} node[anchor=east,color=black] {$f$};
        \end{tikzpicture}
        \ =\
        \begin{tikzpicture}[anchorbase]
          \draw[->] (0,0.3) arc(90:450:0.3);
          \draw[->] (0.8,-1) -- (0.8,1);
          \redcircle{(0.3,0)} node[anchor=west,color=black] {$r$};
          \bluedot{(-0.3,0)} node[anchor=east,color=black] {$f$};
        \end{tikzpicture}
        \ - \sum_{t \ge 0} \sum_{s=0}^t (-1)^{\bar a \bar b}\
        \begin{tikzpicture}[anchorbase]
          \draw[->] (-2.2,-0.2) arc(90:450:0.3);
          \draw[->] (-0.5,-1) -- (-0.5,1);
          \redcircle{(-1.9,-0.5)} node[anchor=west,color=black] {\dotlabel{r-t-2}};
          \bluedot{(-2.5,-0.5)} node[anchor=east,color=black] {$\chk{a}f$};
          \bluedot{(-0.5,0.6)} node[anchor=west,color=black] {$b \psi^{-s}(a) \psi^{-t}(\chk{b})$};
          \redcircle{(-0.5,0.1)} node[anchor=west,color=black] {$t$};
        \end{tikzpicture}
        \ .
      \end{equation}

    \item \emph{Alternating braid relation}:
      \begin{equation} \label{rel:braid-alternating}
        \begin{tikzpicture}[anchorbase]
          \draw[->] (0,0) -- (1,1);
          \draw[->] (1,0) -- (0,1);
          \draw[<-] (0.5,0) .. controls (0,0.5) .. (0.5,1);
        \end{tikzpicture}
        \ -\
        \begin{tikzpicture}[anchorbase]
          \draw[->] (0,0) -- (1,1);
          \draw[->] (1,0) -- (0,1);
          \draw[<-] (0.5,0) .. controls (1,0.5) .. (0.5,1);
        \end{tikzpicture}
        \ =\
        \begin{cases}
          \sum_{r,s,t \ge 0}
          \begin{tikzpicture}[anchorbase]
            \draw[<-] (0,2) -- (0,1.7) arc (180:360:0.3) -- (0.6,2);
            \draw[->] (0,-0.1) -- (0,0.5) arc (180:0:0.3) -- (0.6,-0.1);
            \draw[<-] (1,1.4) arc(90:-270:0.3);
            \draw[->] (2.3,-0.1) -- (2.3,2);
            \redcircle{(0.6,1.7)} node[anchor=east,color=black] {$r$};
            \bluedot{(0,1.7)} node[anchor=east,color=black] {$a$};
            \redcircle{(0.6,0.5)} node[anchor=west,color=black] {$s$};
            \bluedot{(0,0.5)} node[anchor=east,color=black] {$\chk{e}$};
            \bluedot{(1.3,1.1)} node[anchor=west,color=black] {$\chk{a} b$};
            \redcircle{(0.7,1.1)} node[anchor=east,color=black] {\dotlabel{-r-s-t-3}};
            \bluedot{(2.3,0.8)} node[anchor=west,color=black] {$e$};
            \redcircle{(2.3,0.5)} node[anchor=west,color=black] {$t$};
            \bluedot{(2.3,0.2)} node[anchor=west,color=black] {$\chk{b}$};
          \end{tikzpicture}
          & \text{if } k \ge 2,
          \\
          0 & \text{if } -1 \le k \le 1,
          \\
          \sum_{r,s,t \ge 0} (-1)^{\bar a \bar b + \bar a + \bar b + \bar b \bar e}\
          \begin{tikzpicture}[anchorbase]
            \draw[<-] (0,0) -- (0,0.2) arc(180:0:0.3) -- (0.6,0);
            \draw[->] (0,2) -- (0,1.5) arc(180:360:0.3) -- (0.6,2);
            \draw[->] (1.3,1.15) arc(90:-270:0.3);
            \draw[->] (-0.8,0) -- (-0.8,2);
            \bluedot{(0.6,1.8)} node[anchor=west,color=black] {\dotlabel{e \psi^{-r}(\chk{b})}};
            \redcircle{(0.6,1.5)} node[anchor=west,color=black] {$r$};
            \redcircle{(0.6,0.2)} node[anchor=west,color=black] {$s$};
            \bluedot{(0,0.2)} node[anchor=west,color=black] {$a$};
            \bluedot{(1.6,0.85)} node[anchor=west,color=black] {$\chk{a} b$};
            \redcircle{(1,0.85)} node[anchor=east,color=black] {\dotlabel{-r-s-t-3}};
            \bluedot{(-0.8,1.15)} node[anchor=east,color=black] {$\chk{e}$};
            \redcircle{(-0.8,1.6)} node[anchor=east,color=black] {$t$};
          \end{tikzpicture}
          & \text{if } k \le -2.
        \end{cases}
      \end{equation}
  \end{enumerate}
\end{theo}

The terminology \emph{infinite grassmannian relations} for \cref{eq:inf-grass1,eq:inf-grass2,eq:inf-grass3} comes from the analogy with the defining relations for the cohomology ring of the infinite grassmannian.  When $F=\kk$, these relations can also be interpreted at the relations between the elementary and complete symmetric functions.  (See, for example, \cite[(1.21)]{Bru17}.)  Thus, the results of the current paper suggest a theory of deformations of symmetric functions depending on a graded Frobenius superalgebra.

Taking $t=1$ in \cref{eq:inf-grass3} and using \cref{eq:inf-grass1,eq:inf-grass2,eq:f-in-basis,eq:f-in-dual-basis} gives
\begin{equation} \label{eq:central-bubble-reverse}
  \begin{tikzpicture}[anchorbase]
    \draw[->] (0,0.3) arc(90:450:0.3);
    \redcircle{(0.3,0)} node[anchor=west,color=black] {$-k$};
    \bluedot{(-0.3,0)} node[anchor=east,color=black] {$f$};
  \end{tikzpicture}
  \ =\
  \begin{tikzpicture}[anchorbase]
    \draw[<-] (0,0.3) arc(90:450:0.3);
    \redcircle{(-0.3,0)} node[anchor=east,color=black] {$k$};
    \bluedot{(0.3,0)} node[anchor=west,color=black] {$f$};
  \end{tikzpicture}
  \qquad \text{and} \qquad
  \begin{tikzpicture}[anchorbase]
    \draw[->] (0,0.3) arc(90:450:0.3);
    \redcircle{(-0.3,0)} node[anchor=east,color=black] {$-k$};
    \bluedot{(0.3,0)} node[anchor=west,color=black] {$f$};
  \end{tikzpicture}
  \ =\
  \begin{tikzpicture}[anchorbase]
    \draw[<-] (0,0.3) arc(90:450:0.3);
    \redcircle{(0.3,0)} node[anchor=west,color=black] {$k$};
    \bluedot{(-0.3,0)} node[anchor=east,color=black] {$f$};
  \end{tikzpicture}
  \qquad \text{for all } f \in F.
\end{equation}
\details{
  The first equation follows from the second by sliding the dot over the right cup/cap and the token over the left cup/cap.  Then the tokens carry labels of $\psi^{-k}(f)$, which is fine since we quantify over all $f \in F$ and $\psi$ is an automorphism.
}

It follows from the bubble slide relations \cref{rel:clockwise-bubble-slide,rel:counterclockwise-bubble-slide} that the bubbles
\begin{equation} \label{eq:central-bubbles}
  \circled{$f$} :=\
  \begin{tikzpicture}[anchorbase]
    \draw[->] (0,0.3) arc(90:450:0.3);
    \redcircle{(0.3,0)} node[anchor=west,color=black] {$-k$};
    \bluedot{(-0.3,0)} node[anchor=east,color=black] {$f$};
  \end{tikzpicture}
  \ =\
  \begin{tikzpicture}[anchorbase]
    \draw[<-] (0,0.3) arc(90:450:0.3);
    \redcircle{(-0.3,0)} node[anchor=east,color=black] {$k$};
    \bluedot{(0.3,0)} node[anchor=west,color=black] {$f$};
  \end{tikzpicture}
  \ ,\quad f \in F,
\end{equation}
which have the same degree and parity as $f$, are strictly central:
\[
  \begin{tikzpicture}[anchorbase]
    \draw[->] (-0.8,-0.5) -- (-0.8,0.5);
  \end{tikzpicture}
  \ \circled{$f$}
  \ = \circled{$f$}\
  \begin{tikzpicture}[anchorbase]
    \draw[->] (0.8,-0.5) -- (0.8,0.5);
  \end{tikzpicture}
  \qquad \text{and} \qquad
  \begin{tikzpicture}[anchorbase]
    \draw[<-] (-0.8,-0.5) -- (-0.8,0.5);
  \end{tikzpicture}
  \ \circled{$f$}
  \ = \circled{$f$}\
  \begin{tikzpicture}[anchorbase]
    \draw[<-] (0.8,-0.5) -- (0.8,0.5);
  \end{tikzpicture}
  \ .
\]
We can therefore impose additional relations on these bubbles.  If $R$ is a set of homogeneous relations involving these bubbles, we let $\Heis_{F,k}(R)$ denote the supercategory obtained from $\Heis_{F,k}$ by imposing the additional relations $R$.  For example, if
\[
  R = \left\{\circled{$f$} - \circled{$\psi(f)$} : f \in F \right\},
\]
then the sum on the right side of \cref{rel:dot-rotation} is zero, and hence the right and left mates of the dot are equal.  If we also have that $k$ is a multiple of the order of $\psi$, then we see from \cref{rel:token-rotation} that the right and left mates of tokens are equal, in which case the category $\Heis_{F,k}(R)$ is strictly pivotal.

It also follows from \cref{rel:clockwise-bubble-slide,rel:counterclockwise-bubble-slide,theo-eq:neg-ccbubble,theo-eq:neg-cbubble} that the bubbles
\[
  \begin{tikzpicture}[anchorbase]
    \draw[->] (0,0.3) arc(90:450:0.3);
    \redcircle{(0.3,0)} node[anchor=west,color=black] {$r$};
    \bluedot{(-0.3,0)} node[anchor=east,color=black] {$f$};
  \end{tikzpicture}
  \qquad \text{and} \qquad
  \begin{tikzpicture}[anchorbase]
    \draw[<-] (0,0.3) arc(90:450:0.3);
    \redcircle{(-0.3,0)} node[anchor=east,color=black] {$r$};
    \bluedot{(0.3,0)} node[anchor=west,color=black] {$f$};
  \end{tikzpicture}
  \ ,\quad r \in \Z,\ f \in F,\ |f| > 0,
\]
are central.  Indeed, consider the terms in the sum in \cref{rel:clockwise-bubble-slide}.  The token labelled $b \psi^{-s}(a) \psi^{-t}(\chk{b})$ is zero for degree reasons unless $|a|=0$.  But $|a| = 0$ implies that $|\chk{a}| = \Delta$, which in turn implies that the token labelled $\chk{a}f$ is zero whenever $f$ has positive degree.  A similar argument holds for \cref{rel:counterclockwise-bubble-slide}.  That negatively dotted bubbles are strictly central then follows from \cref{theo-eq:neg-ccbubble,theo-eq:neg-cbubble}.

\begin{theo} \label{theo:specializations}
  \begin{enumerate}
    \item \label{theo-item:F=k-specialization} The Heisenberg category $\tilde{\cH}^\lambda$ defined in \cite{MS17} is isomorphic to the additive envelope of $\Heis_{\kk,k}(R)$, where $k = - \sum_i \lambda_i$, and $R = \left\{ \circled{$1$} - \sum_i i\lambda_i \right\}$.  In particular, the Heisenberg category $\cH'$ defined by Khovanov in \cite{Kho14} is isomorphic to the additive envelope of $\Heis_{\kk,-1}(R)$, where $R = \left\{ \circled{$1$} \right\}$.

    \item \label{theo-item:xi=-1-specialization} The Heisenberg supercategory $\cH_F'$ defined in \cite{RS17} (in the case where the trace map of $F$ is even) is isomorphic to the additive envelope of the underlying category of the $\Pi$-envelope (see \cite[\S1.5]{BE17}) of $\Heis_{F^\op,-1}(R)$, where $R = \left\{ \circled{$f$} : f \in F \right\}$.
  \end{enumerate}
\end{theo}

We note also that when $F$ is the two-dimensional Clifford superalgebra, the supercategory $\Heis_{F,k}$ is studied in \cite{CK18}.  In particular, when $k=0$, this supercategory reduces to the degenerate affine oriented Brauer--Clifford supercategory of \cite{CK17}.

The definition of the supercategory $\Heis_{F,k}$ is inspired by the affine wreath product algebras studied in \cite{Sav17}.  One benefit of the inversion relation presentation of \cref{def:H} is that it makes it relatively straightforward to verify that $\Heis_{F,k}$ acts naturally on suitable categories.  In particular, we now describe how $\Heis_{F,k}$ acts naturally on modules over cyclotomic quotients of affine wreath product algebras.  For the remainder of this introduction, we suppose that the Nakayama auotmorphism $\psi$ has finite order $\theta$, and that $k<0$.  (Analogous results could be formulated in the case $k > 0$.)

Following \cite[\S3.1]{Sav17}, we define the \emph{affine wreath product algebra} $\cA_n(F)$ to be the graded superalgebra that is the free product of superalgebras
\[
  \kk[x_1,\dotsc,x_n] \star F^{\otimes n} \star \kk S_n,
\]
modulo the relations
\begin{align*}
  \bsf x_i &= x_i \psi_i(\bsf),& 1 \le i \le n,\ \bsf \in F^{\otimes n}, \\
  s_i x_j &= x_j s_i,& 1 \le i \le n-1,\ 1 \le j \le n,\ j \ne i,i+1, \\
  s_i x_i &= x_{i+1} s_i - b_i \chk{b}_{i+1},& 1 \le i \le n-1, \\
  \pi \bsf &= \prescript{\pi}{}{\bsf} \pi,& \pi \in S_n,\ \bsf \in F^{\otimes n},
\end{align*}
where $\psi_i = \id^{\otimes (i-1)} \otimes \psi \otimes \id^{\otimes (n-i)}$, $f_i = 1^{\otimes (i-1)} \otimes f \otimes 1^{\otimes (n-i)} \in F^{\otimes n}$ ($f \in F$), and $\prescript{\pi}{}{\bsf}$ denotes the natural action of $\pi$ on $\bsf$ by superpermutation of the factors.  The degree and parity on $\cA_n(F)$ are determined by the degree and parity of elements of $F^{\otimes n}$, together with
\[
  |x_i| = \Delta,\quad \bar x_i = 0,\quad
  |\pi| = 0,\quad \bar \pi = 0,\quad
  1 \le i \le n,\ \pi \in S_n,
\]
By convention, we set $\cA_0(F) = \kk$.

For $1 \le r \le \theta$, choose $n_r \ge 0$ and even elements of degree $r \Delta$
\[
  c^{(r,1)},\dotsc,c^{(r,n_r)} \in \left\{ f \in F : \psi(f) = f,\ gf = (-1)^{\bar f \bar g} f \psi^r(g) \text{ for all } g \in F \right\}
\]
such that $\sum_{r=1}^\theta r n_r = -k$.  Let $\bC = (c^{(1,1)}, \dotsc, c^{(1,n_1)},\dotsc, c^{(\theta,1)}, \dotsc, c^{(\theta,n_\theta)})$.  As in \cite[\S6.2]{Sav17}, for $n \ge 0$, we define the \emph{cyclotomic wreath product algebra} to be the quotient $\cA_n^\bC(F) = \cA_n(F)/J_\bC$, where $J_\bC$ is the two-sided ideal in $\cA_n(F)$ generated by the homogeneous element
\[
  \prod_{r=1}^\theta \prod_{j=1}^{n_r} \left( x_1^r - c^{(r,j)} \right).
\]
By convention, we set $\cA_0^\bC(F) = \kk$.  For $n \ge 0$, we have a natural inclusion $\cA_n^\bC(F) \hookrightarrow \cA_{n+1}^\bC(F)$.  (See \cite[\S6.5]{Sav17}.)

We claim that there is a strict $\kk$-linear supermonoidal functor
\[
  \Psi_\bC \colon \Heis_{F^\op,k} \to \END_\kk \left( \bigoplus_{n \ge 0} \cA_n^\bC(F)\text{-mod} \right)
\]
to the graded monoidal supercategory of endofunctors of the sum of the categories of left $\cA_n^\bC(F)$-modules.  The functor $\Psi_\bC$ sends $\sQ_+$ to the $\kk$-linear endofunctor taking an $\cA_n^\bC(F)$-module $M$ to the induced $\cA_{n+1}^\bC(F)$-module $\prescript{\bC}{}{\Ind}_n^{n+1} M := \cA_{n+1}^\bC(F) \otimes_{\cA_n^\bC(F)} M$ and sends $\sQ_-$ to the $\kk$-linear endofunctor taking an $\cA_n^\bC(F)$-module $M$ to the restricted $\cA_{n-1}^\bC(F)$-module $\prescript{\bC}{}{\Res}^n_{n-1} M$.  On the generating morphisms, $\Psi_\bC(x)$, $\Psi_\bC(s)$, $\Psi_\bC(c)$, $\Psi_\bC(d)$, and $\Psi_\bC(\beta_f)$, $f \in F^\op$, are the natural transformations defined on an $\cA_n^\bC(F)$-module as follows:
\begin{itemize}
  \item $\Psi_\bC(x)_M \colon {\prescript{\bC}{}{\Ind}_n^{n+1}} M \to {\prescript{\bC}{}{\Ind}_n^{n+1}} M$, $z \otimes m \mapsto z x_{n+1} \otimes m$;
  \item $\Psi_\bC(s)_M \colon {\prescript{\bC}{}{\Ind}_n^{n+2}} M \to {\prescript{\bC}{}{\Ind}_n^{n+2}}$, $z \otimes m \mapsto z s_{n+1} \otimes m$;
  \item $\Psi_\bC(c)_M \colon M \to {\prescript{\bC}{}{\Res}^{n+1}_n} {\prescript{\bC}{}{\Ind}_n^{n+1}} M$, $m \mapsto 1 \otimes m$;
  \item $\Psi_\bC(d)_M \colon {\prescript{\bC}{}{\Ind}_{n-1}^n} {\prescript{\bC}{}{\Res}^n_{n-1}} M \to M$, $z \otimes m \mapsto zm$;
  \item $\Psi_\bC(\beta_f)_M \colon {\prescript{\bC}{}{\Ind}_n^{n+1}} M \to {\prescript{\bC}{}{\Ind}_n^{n+1}} M$, $z \otimes m \mapsto (-1)^{\bar z \bar f} z f_{n+1} \otimes m$.
\end{itemize}

To prove that $\Psi_\bC$ is well defined, it suffices to verify that it preserves the defining relations from \cref{def:H}.  The affine wreath product algebra relations are straightforward to verify.
\details{
  When verifying \cref{rel:dotslide1}, we use the fact that the left dual basis to the basis $\{\chk{b} : b \in B\}$ of $F^\op$ is $\{(-1)^b : b \in B\}$.
}
The right adjunction relations follow from the fact that $\Psi_\bC(c)$ and $\Psi_\bC(d)$ are the unit and counit of the canonical adjunction making $\left( \prescript{\bC}{}{\Ind}_n^{n+1}, \prescript{\bC}{}{\Res}^{n+1}_n \right)$ into an adjoint pair.  Using \cref{eq:t-def}, we see that $\Psi_\bC(t)$ is the natural transformation coming from the bimodule homomorphism
\[
  \cA_n^\bC(F) \otimes_{\cA_{n-1}^\bC(F)} \cA_n^\bC(F) \to \cA_{n+1}^\bC(F),\quad z \otimes w \mapsto z s_n w.
\]
\details{
  We have
  \[
    \Psi_\bC(t)_M \colon
    z \otimes w \otimes m
    \stackrel{c}{\mapsto} z \otimes w \otimes m
    \stackrel{s}{\mapsto} z s_n \otimes w \otimes m
    \stackrel{d}{\mapsto} z s_n w \otimes m.
  \]
}
Then, the fact that $\Psi_\bC$ preserves the inversion relation follows from the fact that the bimodule isomorphism
\begin{gather*}
  \cA_n^\bC(F) \otimes_{\cA_{n-1}^\bC(F)} \cA_n^\bC(F) \oplus \bigoplus_{r=0}^{-k-1} \bigoplus_{b \in B} \cA_n^\bC(F) \to \cA_{n+1}^\bC(F),
  \\
  (z \otimes z',(w_{r,b})_{r,b}) \mapsto z s_n z' + \sum_{r=0}^{-k-1} \sum_{b \in B} x_{n+1}^r (1^{\otimes n} \otimes \chk{b}) w_{r,b},
\end{gather*}
is an isomorphism, which follows immediately from \cite[Prop.~6.17]{Sav17}.

The supercategory $\Heis_{F,k}$ categorifies a certain lattice Heisenberg algebra, as we now explain.  Again, we suppose that $k < 0$, although an analogous result could be stated in the case that $k > 0$.  Assume that $\kk$ is an algebraically closed field of characteristic zero, and let $\Kar \Heis_{F,k}$ denote the additive Karoubi envelope of the underlying category of the $\Pi$-envelope of $\Heis_{F,k}$ (see \cite[\S1.5]{BE17}).

Let
\[
  \Z_{q,\pi} =
  \begin{cases}
    \Z[q,q^{-1},\pi]/(\pi^2-1) & \text{if all simple left $F$-modules are of type $M$}, \\
    \Z[\frac{1}{2},q,q^{-1}] \cong \Z[\frac{1}{2},q,q^{-1},\pi]/(\pi-1) & \text{if $F$ has a simple left module of type $Q$},
  \end{cases}
\]
where $q$ and $\pi$ are formal parameters.  Then the split Grothendieck group $K_0(F)$ of the category of finitely-generated projective left $F$-modules is naturally a $\Z_{q,\pi}$-module, where $q$ and $\pi$ act by the grading shift and parity shift, respectively.  Similarly, the split Grothendieck ring $K_0(\Kar \Heis_{F,k})$ of $\Kar \Heis_{F,k}$ is naturally a $\Z_{q,\pi}$-algebra.

We equip $K_0(F)$ with a nondegenerate symmetric sesquilinear form
\[
  \langle -,- \rangle_k \colon K_0(F) \times K_0(F) \to \Z_{q,\pi},\quad
  \langle [M_1], [M_2] \rangle_k = (1+q+\dotsb+q^{-k-1}) \grdim \HOM_F(M_1,M_2),
\]
for finitely-generated projective $F$-modules $M_1$ and $M_2$.  (We extend by sesquilinearity.)  Let $\rh_{F,k}$ be the lattice Heiseberg $\Z_{q,\pi}$-algebra associated to the lattice $K_0(F)$ as in \cite[Def.~2.1]{LRS18}.

\begin{theo} \label{theo:categorification}
  Suppose $R$ is a set of bubble relations (that is, a set of homogeneous elements of the algebra generated by the bubbles of the form \cref{eq:central-bubbles}).  Then we have an injective homomorphism of algebras $\Phi \colon \rh_{F,k} \to K_0(\Kar \Heis_{F,k}(R))$.  Furthermore, if $\Delta > 0$ (that is, $F$ is not concentrated in degree zero) and, for each $f \in F$ of degree zero, the span of the elements of $R$ contains $\circled{$f$} - a_f$ for some $a_f \in \kk$, then the homomorphism $\Phi$ is also surjective.
\end{theo}

\Cref{theo:categorification} generalizes and unifies the categorification results of \cite{Kho14,CL12,HS16,RS17,MS17}.  We expect that the surjectivity statement in \cref{theo:categorification} also holds without the assumption that $\Delta > 0$ and without the assumption on $R$.  The assumption on $R$ could likely be dropped if one had a basis theorem for the morphism spaces of $\Heis_{F,k}$.  We are hopeful that such a basis theorem can be proved using the methods of the recent paper \cite{BSW1}, and we plan to look at this in future work.  The assumption $\Delta > 0$ is used to prove that one has found all the idempotents of the objects $\sQ_+^m \sQ_-^n$, $n,m \ge 0$.  It is for this reason that the categorifications considered in \cite{Kho14} and \cite{MS17} (where $\Delta=0$) were conjectural.  However, these conjectures have recently been proved in \cite[Th.~1.1]{BSW1}.  If, in addition to a basis theorem, one could compute the split Grothendieck group of the affine wreath product algebras, the methods of \cite{BSW1} should allow one to also remove the assumption $\Delta > 0$ in \cref{theo:categorification}.

Although, for simplicity, we assume in this paper that the trace map of the graded Frobenius superalgebra $F$ is even, one could just as well consider the case where it is odd.  In this case, the morphism $x$ (corresponding diagrammatically to the dot) would be odd and one of the left adjunction relations \cref{rel:zigzag-leftdown,rel:zigzag-leftup} would involve a sign.  In a different direction, one can define natural $q$-deformations of the affine wreath product algebras of \cite{Sav17}; this is done in \cite{RS19}.  This leads to a $q$-deformation of the supercategory $\Heis_{F,k}$ \cite{BS}.  In the case where $F=\kk$ and $k=0$, one recovers the affine oriented skein category of \cite[\S4]{Bru17b}.  In the case where $F=\kk$ and $k=-1$, one recovers the quantum Heisenberg category of \cite{BSW2} (see also \cite{LS13}).

We believe that the results of the current paper illustrate a broader phenomenon: that many of the categories appearing in invariant theory and categorification can be deformed to incorporate a graded Frobenius algebra.  Such deformations should unify many related constructions that are currently treated separately, as well as provide natural generalizations of current constructions and results.

\iftoggle{detailsnote}{
\subsection*{Hidden details} For the interested reader, the tex file of the \href{https://arxiv.org/abs/1802.01626}{arXiv version} of this paper includes hidden details of some straightforward computations and arguments that are omitted in the pdf file.  These details can be displayed by switching the \texttt{details} toggle to true in the tex file and recompiling.
}{}

\section{Consequences of the inversion relation \label{sec:inversion}}

In this section, we systematically analyse the consequences of the inversion relation.  We will see that this relation implies a surprising number of very natural relations in our category.

\subsection{Graded Frobenius superalgebras \label{subsec:Frobenius-alg}}

For all $f \in F$, we have

\noindent\begin{minipage}{0.5\linewidth}
  \begin{equation} \label{eq:f-in-basis}
    f = \tr (\chk{b} f) b.
  \end{equation}
\end{minipage}%
\begin{minipage}{0.5\linewidth}
  \begin{equation} \label{eq:f-in-dual-basis}
    f = \tr (f b) \chk{b}.
  \end{equation}
\end{minipage}\par\vspace{\belowdisplayskip}

\noindent It follows that, for all $f \in F$,
\begin{equation} \label{eq:f-Euler-commute}
  fb \otimes \chk{b}
  = b \otimes \chk{b}f.
\end{equation}
\details{
  We have
  \[
    fb \otimes \chk{b}
    \stackrel{\cref{eq:f-in-basis}}{=} \tr(\chk{a}fb)a \otimes \chk{b}
    = a \otimes \tr(\chk{a}fb) \chk{b}
    \stackrel{\cref{eq:f-in-dual-basis}}{=} a \otimes \chk{a}f.
  \]
}
This implies that tokens can ``teleport'' in $\Heis_{F,k}$ in the sense that, for $f \in F$, we have
\[
  \begin{tikzpicture}[anchorbase]
    \draw[->] (0,0) --(0,1);
    \draw[->] (0.5,0) -- (0.5,1);
    \bluedot{(0,0.6)} node[anchor=east,color=black] {$fb$};
    \bluedot{(0.5,0.4)} node[anchor=west,color=black] {$\chk{b}$};
  \end{tikzpicture}
  =
  \begin{tikzpicture}[anchorbase]
    \draw[->] (0,0) --(0,1);
    \draw[->] (0.5,0) -- (0.5,1);
    \bluedot{(0,0.6)} node[anchor=east,color=black] {$b$};
    \bluedot{(0.5,0.4)} node[anchor=west,color=black] {$\chk{b}f$};
  \end{tikzpicture}
  \ ,
\]
where the strands can occur anywhere in a diagram (i.e.\ they do not need to be adjacent).

Since the trace map is even, we have $\bar b = \overline{\chk{b}}$ for all $b \in B$.  We also have
\begin{equation} \label{eq:double-dual}
  \chk{(\chk{b})} = (-1)^{\bar b} \psi^{-1}(b).
\end{equation}
\details{
  For $b,c \in B$, we have
  \[
    \delta_{b,c}
    = \tr(\chk{b}c)
    = (-1)^{\bar b} \tr \left( \psi^{-1}(c) \chk{b} \right),
  \]
  which implies that $\chk{(\chk{b})} = (-1)^{\bar b} \psi^{-1}(b)$.
}

\subsection{Right mates}

We define the right mates
\begin{equation} \label{eq:right-mates}
  x' =
  \begin{tikzpicture}[anchorbase]
    \draw[<-] (0,0) -- (0,1);
    \redcircle{(0,0.5)};
  \end{tikzpicture}
  \ :=\
  \begin{tikzpicture}[anchorbase]
    \draw[->] (0,1) -- (0,0.4) arc(180:360:0.2) -- (0.4,0.6) arc(180:0:0.2) -- (0.8,0);
    \redcircle{(0.4,0.5)};
  \end{tikzpicture}
  \ ,\qquad
  s' =
  \begin{tikzpicture}[anchorbase]
    \draw [<-](0,0) -- (0.6,0.6);
    \draw [<-](0.6,0) -- (0,0.6);
  \end{tikzpicture}
  \ :=\
  \begin{tikzpicture}[anchorbase]
    \draw[<-] (0.3,0) -- (-0.3,1);
    \draw[->] (-0.75,1) -- (-0.75,0.5) .. controls (-0.75,0.2) and (-0.5,0) .. (0,0.5) .. controls (0.5,1) and (0.75,0.8) .. (0.75,0.5) -- (0.75,0);
  \end{tikzpicture}
  \ ,\qquad
  \beta_f' =
  \begin{tikzpicture}[anchorbase]
    \draw[->] (0,1) -- (0,0);
    \bluedot{(0,0.5)} node[anchor=east,color=black] {$f$};
  \end{tikzpicture}
  \ :=\
  \begin{tikzpicture}[anchorbase]
    \draw[->] (0,1) -- (0,0.4) arc(180:360:0.2) -- (0.4,0.6) arc(180:0:0.2) -- (0.8,0);
    \bluedot{(0.4,0.5)} node[anchor=east,color=black] {$f$};
  \end{tikzpicture}
  \ ,\ f \in F.
\end{equation}
Using \cref{rel:right-adjunction-up,rel:right-adjunction-down}, we immediately have that
\begin{equation} \label{eq:F-to-sQ_-}
  F \to \End \sQ_-,\quad f \mapsto \beta'_f,
\end{equation}
is an anti-homomorphism of graded superalgebras, that is, that
\[
  \begin{tikzpicture}[anchorbase]
    \draw[<-] (0,0) -- (0,1);
    \bluedot{(0,0.35)} node[anchor=east,color=black] {$g$};
    \bluedot{(0,0.7)} node[anchor=east,color=black] {$f$};
  \end{tikzpicture}
  \ = (-1)^{\bar f \bar g}\
  \begin{tikzpicture}[anchorbase]
    \draw[<-] (0,0) -- (0,1);
    \bluedot{(0,0.5)} node[anchor=west,color=black] {$gf$};
  \end{tikzpicture}
  ,\quad f,g \in F,
\]
and that, for all $f \in F$,

\noindent\begin{minipage}{0.5\linewidth}
  \begin{equation} \label{rel:dot-right-cupcap}
    \begin{tikzpicture}[anchorbase]
      \draw[->] (0,0) -- (0,-0.3) arc (180:360:.3) -- (0.6,0);
      \redcircle{(0,-0.3)};
    \end{tikzpicture}
    \ =\
    \begin{tikzpicture}[anchorbase]
      \draw[->] (0,0) -- (0,-0.3) arc (180:360:.3) -- (0.6,0);
      \redcircle{(0.6,-0.3)};
    \end{tikzpicture}
    \ ,\qquad
    \begin{tikzpicture}[anchorbase]
      \draw[->] (0,0) -- (0,0.3) arc (180:0:.3) -- (0.6,0);
      \redcircle{(0,0.3)};
    \end{tikzpicture}
    \ =\
    \begin{tikzpicture}[anchorbase]
      \draw[->] (0,0) -- (0,0.3) arc (180:0:.3) -- (0.6,0);
      \redcircle{(0.6,0.3)};
    \end{tikzpicture}
    \ ,
  \end{equation}
\end{minipage}%
\begin{minipage}{0.5\linewidth}
  \begin{equation} \label{eq:f-right-cupcap-slide}
    \begin{tikzpicture}[anchorbase]
      \draw[->] (0,0) -- (0,-0.3) arc (180:360:.3) -- (0.6,0);
      \bluedot{(0,-0.3)} node[anchor=west,color=black] {$f$};
    \end{tikzpicture}
    \ =\
    \begin{tikzpicture}[anchorbase]
      \draw[->] (0,0) -- (0,-0.3) arc (180:360:.3) -- (0.6,0);
      \bluedot{(0.6,-0.3)} node[anchor=east,color=black] {$f$};
    \end{tikzpicture}
    \ ,\qquad
    \begin{tikzpicture}[anchorbase]
      \draw[->] (0,0) -- (0,0.3) arc (180:0:.3) -- (0.6,0);
      \bluedot{(0,0.3)} node[anchor=west,color=black] {$f$};
    \end{tikzpicture}
    \ =\
    \begin{tikzpicture}[anchorbase]
      \draw[->] (0,0) -- (0,0.3) arc (180:0:.3) -- (0.6,0);
      \bluedot{(0.6,0.3)} node[anchor=east,color=black] {$f$};
    \end{tikzpicture}
    \ ,
  \end{equation}
\end{minipage}\par\vspace{\belowdisplayskip}

\begin{equation} \label{rel:pitchforks-right}
  \begin{tikzpicture}[anchorbase]
    \draw[<-] (0.8,0) .. controls (0.8,-0.7) and (0.3,-0.7) .. (0,0);
    \draw[<-] (0.4,0) -- (0,-0.7);
  \end{tikzpicture}
  \ =\
  \begin{tikzpicture}[anchorbase]
    \draw[->] (0,0) .. controls (0,-0.7) and (0.5,-0.7) .. (0.8,0);
    \draw[<-] (0.4,0) -- (0.8,-0.7);
  \end{tikzpicture}
  \ ,\qquad
  \begin{tikzpicture}[anchorbase]
    \draw[<-] (0.8,0) .. controls (0.8,-0.7) and (0.3,-0.7) .. (0,0);
    \draw[->] (0.4,0) -- (0,-0.7);
  \end{tikzpicture}
  \ =\
  \begin{tikzpicture}[anchorbase]
    \draw[->] (0,0) .. controls (0,-0.7) and (0.5,-0.7) .. (0.8,0);
    \draw[->] (0.4,0) -- (0.8,-0.7);
  \end{tikzpicture}
  \ ,\qquad
  \begin{tikzpicture}[anchorbase]
    \draw[<-] (0.8,0) .. controls (0.8,0.7) and (0.3,0.7) .. (0,0);
    \draw[->] (0.4,0) -- (0,0.7);
  \end{tikzpicture}
  \ =\
  \begin{tikzpicture}[anchorbase]
    \draw[->] (0,0) .. controls (0,0.7) and (0.5,0.7) .. (0.8,0);
    \draw[->] (0.4,0) -- (0.8,0.7);
  \end{tikzpicture}
  \ ,\qquad
  \begin{tikzpicture}[anchorbase]
    \draw[<-] (0.8,0) .. controls (0.8,0.7) and (0.3,0.7) .. (0,0);
    \draw[<-] (0.4,0) -- (0,0.7);
  \end{tikzpicture}
  \ =\
  \begin{tikzpicture}[anchorbase]
    \draw[->] (0,0) .. controls (0,0.7) and (0.5,0.7) .. (0.8,0);
    \draw[<-] (0.4,0) -- (0.8,0.7);
  \end{tikzpicture}
  \ .
\end{equation}
Furthermore, attaching right caps to the top and right cups to the bottom of the affine wreath product algebra relations gives the following relations for all $f \in F$:

\noindent\begin{minipage}{0.33\linewidth}
  \begin{equation} \label{rel:doublecross-down}
    \begin{tikzpicture}[anchorbase]
      \draw[<-] (0,0) .. controls (0.5,0.5) .. (0,1);
      \draw[<-] (0.5,0) .. controls (0,0.5) .. (0.5,1);
    \end{tikzpicture}
    \ =\
    \begin{tikzpicture}[anchorbase]
      \draw[<-] (0,0) --(0,1);
      \draw[<-] (0.5,0) -- (0.5,1);
    \end{tikzpicture}
    \ ,
  \end{equation}
\end{minipage}%
\begin{minipage}{0.33\linewidth}
  \begin{equation} \label{rel:braid-down}
    \begin{tikzpicture}[anchorbase]
      \draw[<-] (0,0) -- (1,1);
      \draw[<-] (1,0) -- (0,1);
      \draw[<-] (0.5,0) .. controls (0,0.5) .. (0.5,1);
    \end{tikzpicture}
    \ =\
    \begin{tikzpicture}[anchorbase]
      \draw[<-] (0,0) -- (1,1);
      \draw[<-] (1,0) -- (0,1);
      \draw[<-] (0.5,0) .. controls (1,0.5) .. (0.5,1);
    \end{tikzpicture}
    \ ,
  \end{equation}
\end{minipage}
\begin{minipage}{0.33\linewidth}
  \begin{equation} \label{rel:braid-down-up-up}
    \begin{tikzpicture}[anchorbase]
      \draw[->] (0,0) -- (1,1);
      \draw[<-] (1,0) -- (0,1);
      \draw[->] (0.5,0) .. controls (0,0.5) .. (0.5,1);
    \end{tikzpicture}
    \ =\
    \begin{tikzpicture}[anchorbase]
      \draw[->] (0,0) -- (1,1);
      \draw[<-] (1,0) -- (0,1);
      \draw[->] (0.5,0) .. controls (1,0.5) .. (0.5,1);
    \end{tikzpicture}
    \ .
  \end{equation}
\end{minipage}\par\vspace{\belowdisplayskip}

\begin{equation} \label{rel:dot-token-down-slide}
  \begin{tikzpicture}[anchorbase]
    \draw[<-] (0,0) -- (0,1);
    \redcircle{(0,0.7)};
    \bluedot{(0,0.4)} node[anchor=east, color=black] {$f$};
  \end{tikzpicture}
  \ =\
  \begin{tikzpicture}[anchorbase]
    \draw[<-] (0,0) -- (0,1);
    \redcircle{(0,0.4)};
    \bluedot{(0,0.7)} node[anchor=west, color=black] {$\psi(f)$};
  \end{tikzpicture}
  \ ,
\end{equation}

\noindent\begin{minipage}{0.5\linewidth}
  \begin{equation} \label{eq:tokenslide-rightcross1}
    \begin{tikzpicture}[anchorbase]
      \draw[->] (0,0) -- (1,1);
      \draw[<-] (1,0) -- (0,1);
      \bluedot{(.25,.25)} node [anchor=south east, color=black] {$f$};
    \end{tikzpicture}
    \ =\
    \begin{tikzpicture}[anchorbase]
      \draw[->](0,0) -- (1,1);
      \draw[<-](1,0) -- (0,1);
      \bluedot{(0.75,.75)} node [anchor=north west, color=black] {$f$};
    \end{tikzpicture}
    \ ,
  \end{equation}
\end{minipage}%
\begin{minipage}{0.5\linewidth}
  \begin{equation} \label{eq:tokenslide-rightcross2}
    \begin{tikzpicture}[anchorbase]
      \draw[->] (0,0) -- (1,1);
      \draw[<-] (1,0) -- (0,1);
      \bluedot{(.25,.75)} node [anchor=north east, color=black] {$f$};
    \end{tikzpicture}
    \ =\
    \begin{tikzpicture}[anchorbase]
      \draw[->](0,0) -- (1,1);
      \draw[<-](1,0) -- (0,1);
      \bluedot{(0.75,.25)} node [anchor=south west, color=black] {$f$};
    \end{tikzpicture}
    \ ,
  \end{equation}
\end{minipage}\par\vspace{\belowdisplayskip}

\noindent\begin{minipage}{0.5\linewidth}
  \begin{equation} \label{eq:tokenslide-downcross1}
    \begin{tikzpicture}[anchorbase]
      \draw[<-](0,0) -- (1,1);
      \draw[<-](1,0) -- (0,1);
      \bluedot{(0.75,.75)} node [anchor=north west, color=black] {$f$};
    \end{tikzpicture}
    \ =\
    \begin{tikzpicture}[anchorbase]
      \draw[<-] (0,0) -- (1,1);
      \draw[<-] (1,0) -- (0,1);
      \bluedot{(.25,.25)} node [anchor=south east, color=black] {$f$};
    \end{tikzpicture}
    \ ,
  \end{equation}
\end{minipage}%
\begin{minipage}{0.5\linewidth}
  \begin{equation} \label{eq:tokenslide-downcross2}
    \begin{tikzpicture}[anchorbase]
      \draw[<-] (0,0) -- (1,1);
      \draw[<-] (1,0) -- (0,1);
      \bluedot{(.25,.75)} node [anchor=north east, color=black] {$f$};
    \end{tikzpicture}
    \ =\
    \begin{tikzpicture}[anchorbase]
      \draw[<-](0,0) -- (1,1);
      \draw[<-](1,0) -- (0,1);
      \bluedot{(0.75,.25)} node [anchor=south west, color=black] {$f$};
    \end{tikzpicture}
    \ ,
  \end{equation}
\end{minipage}\par\vspace{\belowdisplayskip}

\noindent\begin{minipage}{0.5\linewidth}
  \begin{equation} \label{rel:dotslide-right1}
    \begin{tikzpicture}[anchorbase]
      \draw[->] (0,0) -- (1,1);
      \draw[<-] (1,0) -- (0,1);
      \redcircle{(0.75,0.75)};
    \end{tikzpicture}
    \ -\
    \begin{tikzpicture}[anchorbase]
      \draw[->] (0,0) -- (1,1);
      \draw[<-] (1,0) -- (0,1);
      \redcircle{(0.25,0.25)};
    \end{tikzpicture}
    \ =\
    (-1)^{\bar b}
    \begin{tikzpicture}[anchorbase]
      \draw[->] (0,1.1) -- (0,0.9) arc(180:360:0.3) -- (0.6,1.1);
      \draw[->] (0,-0.1) -- (0,0.1) arc(180:0:0.3) -- (0.6,-0.1);
      \bluedot{(0,0.1)} node[anchor=west,color=black] {$b$};
      \bluedot{(0.6,0.9)} node[anchor=west,color=black] {$\chk{b}$};
    \end{tikzpicture}\ ,
  \end{equation}
\end{minipage}%
\begin{minipage}{0.5\linewidth}
  \begin{equation} \label{rel:dotslide-right2}
    \begin{tikzpicture}[anchorbase]
      \draw[->] (0,0) -- (1,1);
      \draw[<-] (1,0) -- (0,1);
      \redcircle{(0.25,0.75)};
    \end{tikzpicture}
    \ -\
    \begin{tikzpicture}[anchorbase]
      \draw[->] (0,0) -- (1,1);
      \draw[<-] (1,0) -- (0,1);
      \redcircle{(0.75,0.25)};
    \end{tikzpicture}
    \ =\
    \begin{tikzpicture}[anchorbase]
      \draw[->] (0,1.1) -- (0,0.9) arc(180:360:0.3) -- (0.6,1.1);
      \draw[->] (0,-0.1) -- (0,0.1) arc(180:0:0.3) -- (0.6,-0.1);
      \bluedot{(0,0.1)} node[anchor=east,color=black] {$\chk{b}$};
      \bluedot{(0.6,0.9)} node[anchor=east,color=black] {$b$};
    \end{tikzpicture}\ ,
  \end{equation}
\end{minipage}\par\vspace{\belowdisplayskip}

\noindent\begin{minipage}{0.5\linewidth}
  \begin{equation} \label{rel:dotslide-down1}
    \begin{tikzpicture}[anchorbase]
      \draw[<-] (0,0) -- (1,1);
      \draw[<-] (1,0) -- (0,1);
      \redcircle{(0.75,0.25)};
    \end{tikzpicture}
    \ -\
    \begin{tikzpicture}[anchorbase]
      \draw[<-] (0,0) -- (1,1);
      \draw[<-] (1,0) -- (0,1);
      \redcircle{(0.25,0.75)};
    \end{tikzpicture}
    \ =\
    \begin{tikzpicture}[anchorbase]
      \draw[<-] (0,0) -- (0,1);
      \draw[<-] (0.5,0) -- (0.5,1);
      \bluedot{(0.5,0.4)} node[anchor=west, color=black] {$\chk{b}$};
      \bluedot{(0,0.7)} node[anchor=east, color=black] {$b$};
    \end{tikzpicture}\ ,
  \end{equation}
\end{minipage}%
\begin{minipage}{0.5\linewidth}
  \begin{equation} \label{rel:dotslide-down2}
    \begin{tikzpicture}[anchorbase]
      \draw[<-] (0,0) -- (1,1);
      \draw[<-] (1,0) -- (0,1);
      \redcircle{(0.75,0.75)};
    \end{tikzpicture}
    \ -\
    \begin{tikzpicture}[anchorbase]
      \draw[<-] (0,0) -- (1,1);
      \draw[<-] (1,0) -- (0,1);
      \redcircle{(0.25,0.25)};
    \end{tikzpicture}
    \ =\
    \begin{tikzpicture}[anchorbase]
      \draw[<-] (0,0) -- (0,1);
      \draw[<-] (0.5,0) -- (0.5,1);
      \bluedot{(0,0.4)} node[anchor=east, color=black] {$\chk{b}$};
      \bluedot{(0.5,0.7)} node[anchor=west, color=black] {$b$};
    \end{tikzpicture}\ .
  \end{equation}
\end{minipage}\par\vspace{\belowdisplayskip}

\begin{lem} \label{lem:flip-symmetry}
  There is an isomorphism of monoidal supercategories $\omega \colon \Heis_{F,k} \xrightarrow{\cong} \Heis^\op_{F^\op,-k}$ interchanging the objects $\sQ_+$ and $\sQ_-$ and defined on the generating morphisms by
  \[
    \omega(x) = x',\quad
    \omega(s) = -s',\quad
    \omega(c) = d,\quad
    \omega(d) = c,\quad
    \omega(\beta_f) = \beta'_f,\ f \in F.
  \]
\end{lem}

\begin{proof}
  The functor $\omega$ preserves the affine wreath product algebra relations by \cref{eq:F-to-sQ_-,rel:doublecross-down,rel:braid-down,rel:braid-down-up-up,rel:dot-token-down-slide,eq:tokenslide-downcross2,rel:dotslide-down2}.  The inversion relation is also preserved since we replace $k$ by $-k$.  By the right adjunction relations, we have $\omega(x') = x$ and $\omega(s')=-s$.  Thus $\omega^2 = \id$.  Hence $\omega$ is an isomorphism.
\end{proof}

Diagrammatically, $\omega$ reflects diagrams in the horizontal axis (multiplying by the appropriate sign when odd elements change height), and then multiplies by $(-1)^r$, where $r$ is the total number of crossings appearing in the diagram.  This description also holds for the other morphisms to be defined below, except that the sign becomes $(-1)^{r+r'}$, where $r$ is the total number of crossings and $r'$ is the total number of undecorated left cups and caps (see \cref{cor:omega-left-cupcap}).

\subsection{Left crossings, cups, and caps}

We define
\[
  t' =
  \begin{tikzpicture}[anchorbase]
    \draw [<-](0,0) -- (0.6,0.6);
    \draw [->](0.6,0) -- (0,0.6);
  \end{tikzpicture}
  \ \colon \sQ_- \sQ_+ \to \sQ_+ \sQ_-
\]
and
\begin{equation} \label{eq:decorated-left-cap-cup}
  \begin{tikzpicture}[>=To,baseline={([yshift=1ex]current bounding box.center)}]
    \draw[<-] (0,0.2) -- (0,0) arc (180:360:.3) -- (0.6,0.2);
    \greensquare{(0.3,-0.3)} node[anchor=north,color=black] {\squarelabel{(r,b)}};
  \end{tikzpicture}
  \ \colon \one \to \sQ_+ \sQ_-
  ,\qquad
  \begin{tikzpicture}[>=To,baseline={([yshift=-2ex]current bounding box.center)}]
    \draw[<-] (0,-0.2) -- (0,0) arc (180:0:.3) -- (0.6,-0.2);
    \greensquare{(0.3,0.3)} node[anchor=south,color=black] {\squarelabel{(r,b)}};
  \end{tikzpicture}
  \ \colon \sQ_- \sQ_+ \to \one.
\end{equation}
for $0 \le r < k$ or $0 \le r < -k$, respectively, by declaring that
\begin{equation} \label{eq:inversion-leftcup-def}
  \left[
    \begin{tikzpicture}[anchorbase]
      \draw [<-](0,0) -- (0.6,0.6);
      \draw [->](0.6,0) -- (0,0.6);
    \end{tikzpicture}
    \quad
    \begin{tikzpicture}[>=To,baseline={([yshift=1ex]current bounding box.center)}]
      \draw[<-] (0,0.2) -- (0,0) arc (180:360:.3) -- (0.6,0.2);
      \greensquare{(0.3,-0.3)} node[anchor=north,color=black] {\squarelabel{(r,b)}};
    \end{tikzpicture},\
    0 \le r \le k-1,\ b \in B
  \right]
  =
  \left(\left[
    \begin{tikzpicture}[anchorbase]
      \draw [->](0,0) -- (0.6,0.6);
      \draw [<-](0.6,0) -- (0,0.6);
    \end{tikzpicture}
    \quad
    \begin{tikzpicture}[anchorbase]
      \draw[->] (0,0) -- (0,0.7) arc (180:0:.3) -- (0.6,0);
      \redcircle{(0,0.6)} node[anchor=west,color=black] {$r$};
      \bluedot{(0,0.2)} node[anchor=west,color=black] {$\chk{b}$};
    \end{tikzpicture}\ ,\
    0 \le r \le k-1,\ b \in B
  \right]^T \right)^{-1}
  \quad \text{if } k \ge 0,
\end{equation}
or
\begin{equation} \label{eq:inversion-leftcap-def}
  \left[
    \begin{tikzpicture}[anchorbase]
      \draw [<-](0,0) -- (0.6,0.6);
      \draw [->](0.6,0) -- (0,0.6);
    \end{tikzpicture}
    \quad
    \begin{tikzpicture}[>=To,baseline={([yshift=-2ex]current bounding box.center)}]
      \draw[<-] (0,-0.2) -- (0,0) arc (180:0:.3) -- (0.6,-0.2);
      \greensquare{(0.3,0.3)} node[anchor=south,color=black] {\squarelabel{(r,b)}};
    \end{tikzpicture},\
    0 \le r \le -k-1,\ b \in B
  \right]^T
  =
  \left[
    \begin{tikzpicture}[anchorbase]
      \draw [->](0,0) -- (0.6,0.6);
      \draw [<-](0.6,0) -- (0,0.6);
    \end{tikzpicture}
    \quad
    \begin{tikzpicture}[anchorbase]
      \draw[->] (0,1) -- (0,0.3) arc (180:360:.3) -- (0.6,1);
      \redcircle{(0.6,0.7)} node[anchor=east,color=black] {$r$};
      \bluedot{(0.6,0.3)} node[anchor=east,color=black] {$\chk{b}$};
    \end{tikzpicture}\ ,\
    0 \le r \le -k-1,\ b \in B
  \right]^{-1}
  \quad \text{if } k < 0.
\end{equation}
More precisely, we add the left crossing $t'$ and the decorated cups and caps in \cref{eq:decorated-left-cap-cup} as new generators, and impose the relations corresponding to the statements that the matrices in \cref{eq:inversion-leftcup-def} or \cref{eq:inversion-leftcap-def} are two-sided inverses.

We extend the definition of the decorated left cups and caps by linearity in the second argument of the label.  In other words, for $f \in F$, we define
\[
  \begin{tikzpicture}[>=To,baseline={([yshift=1ex]current bounding box.center)}]
    \draw[<-] (0,0.2) -- (0,0) arc (180:360:.3) -- (0.6,0.2);
    \greensquare{(0.3,-0.3)} node[anchor=north,color=black] {\squarelabel{(r,f)}};
  \end{tikzpicture}
  \ = \tr(\chk{b}f)\
  \begin{tikzpicture}[>=To,baseline={([yshift=1ex]current bounding box.center)}]
    \draw[<-] (0,0.2) -- (0,0) arc (180:360:.3) -- (0.6,0.2);
    \greensquare{(0.3,-0.3)} node[anchor=north,color=black] {\squarelabel{(r,b)}};
  \end{tikzpicture}
  \ ,\quad \text{if } k > 0,
  \qquad
  \begin{tikzpicture}[>=To,baseline={([yshift=-2ex]current bounding box.center)}]
    \draw[<-] (0,-0.2) -- (0,0) arc (180:0:.3) -- (0.6,-0.2);
    \greensquare{(0.3,0.3)} node[anchor=south,color=black] {\squarelabel{(r,f)}};
  \end{tikzpicture}
  \ = \tr(\chk{b}f)\
  \begin{tikzpicture}[>=To,baseline={([yshift=-2ex]current bounding box.center)}]
    \draw[<-] (0,-0.2) -- (0,0) arc (180:0:.3) -- (0.6,-0.2);
    \greensquare{(0.3,0.3)} node[anchor=south,color=black] {\squarelabel{(r,b)}};
  \end{tikzpicture}
  \ ,\quad \text{if } k < 0.
\]
We then define

\noindent\begin{minipage}{0.5\linewidth}
  \begin{equation} \label{eq:left-cup-def}
    c' =
    \begin{tikzpicture}[anchorbase]
      \draw[<-] (0,0.2) -- (0,0) arc (180:360:.3) -- (0.6,0.2);
    \end{tikzpicture}
    :=
    \begin{cases}
      -\
      \begin{tikzpicture}[>=To,baseline={([yshift=1ex]current bounding box.center)}]
        \draw[<-] (0,0.2) -- (0,0) arc (180:360:.3) -- (0.6,0.2);
        \greensquare{(0.3,-0.3)} node[anchor=north,color=black] {\squarelabel{(k-1,1)}};
      \end{tikzpicture}
      & \text{if } k > 0, \\
      \begin{tikzpicture}[anchorbase]
        \draw[<-] (0,0) .. controls (0.3,-0.3) and (0.6,-0.3) .. (0.6,-0.6) arc(360:180:0.3) .. controls (0,-0.3) and (0.3,-0.3) .. (0.6,0);
        \redcircle{(0.6,-0.6)} node[anchor=west,color=black] {$-k$};
      \end{tikzpicture}
      & \text{if } k \le 0,
    \end{cases}
  \end{equation}
\end{minipage}%
\begin{minipage}{0.5\linewidth}
  \begin{equation} \label{eq:left-cap-def}
    d' =
    \begin{tikzpicture}[anchorbase]
      \draw[<-] (0,-0.2) -- (0,0) arc (180:0:.3) -- (0.6,-0.2);
    \end{tikzpicture}
    :=
    \begin{cases}
      \begin{tikzpicture}[anchorbase]
        \draw[<-] (0,0) .. controls (0.3,0.3) and (0.6,0.3) .. (0.6,0.6) arc(0:180:0.3) .. controls (0,0.3) and (0.3,0.3) .. (0.6,0);
        \redcircle{(0.6,0.6)} node[anchor=west,color=black] {$k$};
      \end{tikzpicture}
      & \text{if } k \ge 0, \\
      \begin{tikzpicture}[>=To,baseline={([yshift=-2ex]current bounding box.center)}]
        \draw[<-] (0,-0.2) -- (0,0) arc (180:0:.3) -- (0.6,-0.2);
        \greensquare{(0.3,0.3)} node[anchor=south,color=black] {\squarelabel{(-k-1,1)}};
      \end{tikzpicture}
      & \text{if } k < 0.
    \end{cases}
  \end{equation}
\end{minipage}\par\vspace{\belowdisplayskip}

\noindent Since the maps \cref{eq:decorated-left-cap-cup} have degree $-|\chk{b}| - r \Delta$ and parity $\bar b$, we have
\[
  |c'| = -k \Delta,\quad
  |d'| = k \Delta,\quad
  \bar c' = \bar d' = 0.
\]

It then follows that

\noindent \begin{minipage}{0.5\linewidth}
  \begin{equation} \label{rel:up-down-doublecross}
    \begin{tikzpicture}[anchorbase]
      \draw[->] (0,0) -- (0,1);
      \draw[<-] (0.5,0) -- (0.5,1);
    \end{tikzpicture}
    \ =\
    \begin{tikzpicture}[anchorbase]
      \draw[->] (0,0) .. controls (0.5,0.5) .. (0,1);
      \draw[<-] (0.5,0) .. controls (0,0.5) .. (0.5,1);
    \end{tikzpicture}
    \ + \sum_{r=0}^{k-1}\
    \begin{tikzpicture}[anchorbase]
      \draw[->] (0,0) -- (0,0.7) arc (180:0:0.3) -- (0.6,0);
      \draw[<-] (0,2) -- (0,1.8) arc (180:360:0.3) -- (0.6,2);
      \redcircle{(0,0.6)} node[anchor=west,color=black] {$r$};
      \bluedot{(0,0.2)} node[anchor=west,color=black] {$\chk{b}$};
      \greensquare{(0.3,1.5)} node[anchor=north west,color=black] {\squarelabel{(r,b)}};
    \end{tikzpicture}
    \ ,
  \end{equation}
\end{minipage}%
\noindent\begin{minipage}{0.5\linewidth}
  \begin{equation} \label{rel:down-up-doublecross}
    \begin{tikzpicture}[anchorbase]
      \draw[<-] (0,0) -- (0,1);
      \draw[->] (0.5,0) -- (0.5,1);
    \end{tikzpicture}
    \ =\
    \begin{tikzpicture}[anchorbase]
      \draw[<-] (0,0) .. controls (0.5,0.5) .. (0,1);
      \draw[->] (0.5,0) .. controls (0,0.5) .. (0.5,1);
    \end{tikzpicture}
    \ + \sum_{r=0}^{-k-1}\
    \begin{tikzpicture}[anchorbase]
      \draw[->] (0,2) -- (0,1.3) arc (180:360:0.3) -- (0.6,2);
      \draw[<-] (0,0) -- (0,0.2) arc (180:0:0.3) -- (0.6,0);
      \redcircle{(0.6,1.7)} node[anchor=east,color=black] {$r$};
      \bluedot{(0.6,1.3)} node[anchor=east,color=black] {$\chk{b}$};
      \greensquare{(0.3,0.5)} node[anchor=west,color=black] {\squarelabel{(r,b)}};
    \end{tikzpicture}
    \ ,
  \end{equation}
\end{minipage}\par\vspace{\belowdisplayskip}

\noindent where the right sides of \cref{rel:up-down-doublecross,rel:down-up-doublecross} are sums of mutually orthogonal idempotents.  When $k > 0$, we have

\noindent\begin{minipage}{0.5\linewidth}
  \begin{equation} \label{rel:leftcurl-l-positive}
    \begin{tikzpicture}[anchorbase]
      \draw[->] (0,0) .. controls (0.3,-0.3) and (0.6,-0.3) .. (0.6,-0.6) arc(360:180:0.3) .. controls (0,-0.3) and (0.3,-0.3) .. (0.6,0);
      \bluedot{(0,-0.6)} node[anchor=east,color=black] {$f$};
    \end{tikzpicture}
    \ =\
    \begin{tikzpicture}[anchorbase]
      \draw[->] (0,0) .. controls (0.3,-0.3) and (0.6,-0.3) .. (0.6,-0.6) arc(360:180:0.3) .. controls (0,-0.3) and (0.3,-0.3) .. (0.6,0);
      \bluedot{(0.6,-0.6)} node[anchor=west,color=black] {$f$};
    \end{tikzpicture}
    \ =\
    \begin{tikzpicture}[anchorbase]
      \draw[->] (0.6,-0.5) .. controls (0.1,0) and (0,-0.1) .. (0,0.3) -- (0,0.7) arc (180:0:.3) -- (0.6,0.3) .. controls (0.6,-0.1) and (0.5,0) .. (0,-0.5);
      \redcircle{(0,0.6)} node[anchor=west,color=black] {$r$};
      \bluedot{(0,0.3)} node[anchor=west,color=black] {$f$};
    \end{tikzpicture}
    \ =0,
  \end{equation}
\end{minipage}%
\begin{minipage}{0.5\linewidth}
  \begin{equation} \label{rel:cc-l-positive}
    \begin{tikzpicture}[anchorbase]
      \draw[<-] (0,0.3) arc(90:450:0.3);
      \redcircle{(-0.3,0)} node[anchor=east,color=black] {$r$};
      \bluedot{(0.3,0)} node[anchor=west,color=black] {$f$};
    \end{tikzpicture}
    = -\delta_{r,k-1} \tr(f),
  \end{equation}
\end{minipage}\par\vspace{\belowdisplayskip}

\noindent for all $0 \le r \le k-1$ and $f \in F$.
\details{
  To see \cref{rel:leftcurl-l-positive}, note that we can slide the token above the curl using \cref{eq:tokenslide-rightcross1,eq:tokenslide-rightcross2}.  Then the curl without a token is zero by the inversion relation.   The proof of \cref{rel:leftcurl-l-negative} below is analogous.

  To see \cref{rel:cc-l-positive}, we compute
  \begin{multline*}
    \begin{tikzpicture}[anchorbase]
      \draw[<-] (0,0.3) arc(90:450:0.3);
      \redcircle{(-0.3,0)} node[anchor=east,color=black] {$r$};
      \bluedot{(0.3,0)} node[anchor=west,color=black] {$f$};
    \end{tikzpicture}
    \stackrel{\cref{eq:left-cup-def}}{=} -\
    \begin{tikzpicture}[anchorbase]
      \draw[<-] (0,0.3) arc(90:450:0.3);
      \redcircle{(-0.3,0)} node[anchor=east,color=black] {$r$};
      \bluedot{(0.3,0)} node[anchor=west,color=black] {$f$};
      \greensquare{(0,-0.3)} node[anchor=north,color=black] {\dotlabel{(k-1,1)}};
    \end{tikzpicture}
    \stackrel[\cref{rel:dot-token-up-slide}]{\cref{eq:f-right-cupcap-slide}}{=} -\
    \begin{tikzpicture}[anchorbase]
      \draw[->] (0.6,0.2) -- (0.6,-0.2) arc(360:180:0.3) -- (0,0.2) arc(180:0:0.3);
      \redcircle{(0,0.2)} node[anchor=east,color=black] {$r$};
      \bluedot{(0,-0.2)} node[anchor=east,color=black] {$\psi^r(f)$};
      \greensquare{(0.3,-0.5)} node[anchor=north,color=black] {\dotlabel{(k-1,1)}};
    \end{tikzpicture}
    \stackrel[\cref{eq:f-in-dual-basis}]{\cref{eq:f-in-basis}}{=} - \tr \left( \psi^r(f) a \right) \tr(\chk{b})
    \begin{tikzpicture}[anchorbase]
      \draw[->] (0.6,0.2) -- (0.6,-0.2) arc(360:180:0.3) -- (0,0.2) arc(180:0:0.3);
      \redcircle{(0,0.2)} node[anchor=east,color=black] {$r$};
      \bluedot{(0,-0.2)} node[anchor=east,color=black] {$\chk{a}$};
      \greensquare{(0.3,-0.5)} node[anchor=north,color=black] {\dotlabel{(k-1,b)}};
    \end{tikzpicture}
    \\
    = - \tr \left( \psi^r(f) a \right) \tr(\chk{b}) \delta_{r,k-1} \delta_{a,b}
    = - \delta_{r,k-1} \tr \left( \psi^r(f) b \right) \tr(\chk{b})
    \\
    = - \delta_{r,k-1} \tr \left( \psi^r(f) \tr(\chk{b}) b \right)
    \stackrel{\cref{eq:f-in-basis}}{=} - \delta_{r,k-1} \tr \left( \psi^r(f) \right)
    = - \delta_{r,k-1} \tr(f).
  \end{multline*}
  The proof of \cref{rel:ccc-l-negative} below is analogous.
}
When $k < 0$, we have

\noindent\begin{minipage}{0.5\linewidth}
  \begin{equation} \label{rel:leftcurl-l-negative}
    \begin{tikzpicture}[anchorbase]
      \draw[->] (0,0) .. controls (0.3,0.3) and (0.6,0.3) .. (0.6,0.6) arc(0:180:0.3) .. controls (0,0.3) and (0.3,0.3) .. (0.6,0);
      \bluedot{(0,0.6)} node[anchor=east,color=black] {$f$};
    \end{tikzpicture}
    \ =\
    \begin{tikzpicture}[anchorbase]
      \draw[->] (0,0) .. controls (0.3,0.3) and (0.6,0.3) .. (0.6,0.6) arc(0:180:0.3) .. controls (0,0.3) and (0.3,0.3) .. (0.6,0);
      \bluedot{(0.6,0.6)} node[anchor=west,color=black] {$f$};
    \end{tikzpicture}
    \ =\
    \begin{tikzpicture}[anchorbase]
      \draw[->] (0.6,0.5) .. controls (0.1,0) and (0,0.1) .. (0,-0.3) -- (0,-0.7) arc (180:360:.3) -- (0.6,-0.3) .. controls (0.6,0.1) and (0.5,0) .. (0,0.5);
      \redcircle{(0.6,-0.3)} node[anchor=east,color=black] {$r$};
      \bluedot{(0.6,-0.6)} node[anchor=east,color=black] {$f$};
    \end{tikzpicture}
    \  = 0
  \end{equation}
\end{minipage}%
\begin{minipage}{0.5\linewidth}
  \begin{equation} \label{rel:ccc-l-negative}
    \begin{tikzpicture}[anchorbase]
      \draw[->] (0,0.3) arc(90:450:0.3);
      \redcircle{(0.3,0)} node[anchor=west,color=black] {$r$};
      \bluedot{(-0.3,0)} node[anchor=east,color=black] {$f$};
    \end{tikzpicture}
    = \delta_{r,-k-1} \tr(f),
  \end{equation}
\end{minipage}\par\vspace{\belowdisplayskip}

\noindent for all $0 \le r \le -k-1$ and $f \in F$.

\begin{cor} \label{cor:omega-left-cupcap}
  If $k > 0$, we have
  \begin{equation} \label{eq:omega-greenbox}
    \omega \left(\
    \begin{tikzpicture}[anchorbase]
      \draw[<-] (0,0.2) -- (0,0) arc (180:360:.3) -- (0.6,0.2);
      \greensquare{(0.3,-0.3)} node[anchor=north,color=black] {\squarelabel{(r,f)}};
    \end{tikzpicture}
    \right)
    = (-1)^{\bar f}\
    \begin{tikzpicture}[>=To,baseline={([yshift=-2ex]current bounding box.center)}]
      \draw[<-] (0,-0.2) -- (0,0) arc (180:0:.3) -- (0.6,-0.2);
      \greensquare{(0.3,0.3)} node[anchor=south,color=black] {\squarelabel{(r,f)}};
    \end{tikzpicture},
    \quad \text{for all } 0 \le r \le k-1,\ f \in F.
  \end{equation}
  In particular, $\omega(c') = - d'$ and $\omega(d') = - c'$.
\end{cor}

\begin{proof}
  When $f \in B$, \cref{eq:omega-greenbox} follows from \cref{eq:inversion-leftcup-def,eq:inversion-leftcap-def}, together with \cref{lem:flip-symmetry}.
  \details{
    The sign comes from the fact that we interchange the order in the inversion relation.  Note also that, for $a,b \in B$ and $0 \le r,s \le k$, we have
    \[
      \begin{tikzpicture}[anchorbase]
        \draw[->] (-0.3,0) arc(-180:180:0.3);
        \greensquare{(0,0.3)} node[anchor=south,color=black] {\squarelabel{(r,b)}};
        \bluedot{(0.25,-0.17)} node[anchor=west,color=black] {$\chk{a}$};
        \redcircle{(0.25,0.15)} node[anchor=west,color=black] {$s$};
      \end{tikzpicture}
      \ = \delta_{a,b} \delta_{r,s} =\
      \begin{tikzpicture}[anchorbase]
        \draw[->] (0.3,0) arc(0:360:0.3);
        \greensquare{(0,-0.3)} node[anchor=north,color=black] {\squarelabel{(r,b)}};
        \bluedot{(-0.25,-0.17)} node[anchor=east,color=black] {$\chk{a}$};
        \redcircle{(-0.25,0.15)} node[anchor=east,color=black] {$s$};
      \end{tikzpicture}
    \]
    Interchanging the height of the square and the token introduces a sign of $(-1)^{\bar b}$, which is exactly balanced by the sign in \cref{eq:omega-greenbox}.
  }
  It then follows for arbitrary $f \in F$ by linearity.  The final statement then follows from \cref{eq:left-cup-def,eq:left-cap-def} and \cref{lem:flip-symmetry}.
\end{proof}

Relations \cref{eq:leftcup-square-dual-dot,eq:leftcap-square-dual-dot} in the following lemma will be generalized in \cref{lem:square-left-cap-cup}.

\begin{lem}
  The following relations hold for all $f \in F$:

  \noindent\begin{minipage}{0.5\linewidth}
    \begin{equation} \label{eq:tokenslide-leftcross1}
      \begin{tikzpicture}[anchorbase]
        \draw[<-] (0,0) -- (1,1);
        \draw[->] (1,0) -- (0,1);
        \bluedot{(.25,.75)} node [anchor=north east, color=black] {$f$};
      \end{tikzpicture}
      \ =\
      \begin{tikzpicture}[anchorbase]
        \draw[<-](0,0) -- (1,1);
        \draw[->](1,0) -- (0,1);
        \bluedot{(0.75,.25)} node [anchor=south west, color=black] {$f$};
      \end{tikzpicture}
      \ ,
    \end{equation}
  \end{minipage}%
  \begin{minipage}{0.5\linewidth}
    \begin{equation} \label{eq:tokenslide-leftcross2}
      \begin{tikzpicture}[anchorbase]
        \draw[<-] (0,0) -- (1,1);
        \draw[->] (1,0) -- (0,1);
        \bluedot{(.25,.25)} node [anchor=south east, color=black] {$f$};
      \end{tikzpicture}
      \ =\
      \begin{tikzpicture}[anchorbase]
        \draw[<-](0,0) -- (1,1);
        \draw[->](1,0) -- (0,1);
        \bluedot{(0.75,.75)} node [anchor=north west, color=black] {$f$};
      \end{tikzpicture}
      \ ,
    \end{equation}
  \end{minipage}\par\vspace{\belowdisplayskip}

  \noindent\begin{minipage}{0.5\linewidth}
    \begin{equation} \label{eq:leftcup-square-dual-dot}
      \begin{tikzpicture}[anchorbase]
        \draw[<-] (0,0.2) -- (0,0) arc (180:360:.3) -- (0.6,0.2);
        \greensquare{(0.3,-0.3)} node[anchor=north,color=black] {\squarelabel{(k-1,f)}};
      \end{tikzpicture}
      \ = - \
      \begin{tikzpicture}[anchorbase]
        \draw[<-] (0,0) -- (0,-0.3) arc (180:360:.3) -- (0.6,0);
        \bluedot{(0,-0.3)} node[anchor=west,color=black] {$f$};
      \end{tikzpicture}
      ,\quad k > 0,
    \end{equation}
  \end{minipage}%
  \begin{minipage}{0.5\linewidth}
    \begin{equation} \label{eq:leftcap-square-dual-dot}
      \begin{tikzpicture}[anchorbase]
        \draw[<-] (0,-0.2) -- (0,0) arc (180:0:.3) -- (0.6,-0.2);
        \greensquare{(0.3,0.3)} node[anchor=south,color=black] {\squarelabel{(-k-1,f)}};
      \end{tikzpicture}
      \ = (-1)^{\bar f}\
      \begin{tikzpicture}[anchorbase]
        \draw[<-] (0,0) -- (0,0.3) arc (180:0:.3) -- (0.6,0);
        \bluedot{(0,0.3)} node[anchor=west,color=black] {$f$};
      \end{tikzpicture}
      ,\quad k < 0.
    \end{equation}
  \end{minipage}\par\vspace{\belowdisplayskip}
\end{lem}

\begin{proof}
  To prove \cref{eq:tokenslide-leftcross1}, compose \cref{eq:tokenslide-rightcross1} on the top and bottom with the left crossing $t'$, and then use \cref{rel:up-down-doublecross,rel:down-up-doublecross,rel:leftcurl-l-positive,rel:leftcurl-l-negative}.  To prove \cref{eq:leftcup-square-dual-dot} for $f \in B$, compose \cref{rel:up-down-doublecross} on the bottom with
  \[
    \begin{tikzpicture}[anchorbase]
      \draw[<-] (0,0) -- (0,-0.3) arc (180:360:.3) -- (0.6,0);
      \bluedot{(0,-0.3)} node[anchor=west,color=black] {$f$};
    \end{tikzpicture}
  \]
  and use \cref{rel:leftcurl-l-positive,rel:cc-l-positive}.  The result for general $f \in F$ follows by linearity.  The relations \cref{eq:tokenslide-leftcross2,eq:leftcap-square-dual-dot} then follow from \cref{eq:tokenslide-leftcross1,eq:leftcup-square-dual-dot} by applying the automorphism $\omega$ and using \cref{cor:omega-left-cupcap}.
\end{proof}

\begin{lem}
  The following relations hold for all $f \in F$:

  \noindent\begin{minipage}{0.5\linewidth}
    \begin{equation} \label{rel:f-left-cap}
      \begin{tikzpicture}[anchorbase]
        \draw[<-] (0,0) -- (0,0.3) arc (180:0:.3) -- (0.6,0);
        \bluedot{(0.0,0.3)} node[anchor=east,color=black] {$f$};
      \end{tikzpicture}
      \ = \
      \begin{tikzpicture}[anchorbase]
        \draw[<-] (0,0) -- (0,0.3) arc (180:0:.3) -- (0.6,0);
        \bluedot{(0.6,0.3)} node[anchor=west,color=black] {$\psi^k(f)$};
      \end{tikzpicture}\ ,
    \end{equation}
  \end{minipage}%
  \begin{minipage}{0.5\linewidth}
    \begin{equation} \label{rel:f-left-cup}
      \begin{tikzpicture}[anchorbase]
        \draw[<-] (0,0) -- (0,-0.3) arc (180:360:.3) -- (0.6,0);
        \bluedot{(0.6,-0.3)} node[anchor=west,color=black] {$f$};
      \end{tikzpicture}
      \ =\
      \begin{tikzpicture}[anchorbase]
        \draw[<-] (0,0) -- (0,-0.3) arc (180:360:.3) -- (0.6,0);
        \bluedot{(0.0,-0.3)} node[anchor=east,color=black] {$\psi^k(f)$};
      \end{tikzpicture}\ .
    \end{equation}
  \end{minipage}\par\vspace{\belowdisplayskip}
\end{lem}

\begin{proof}
  It suffices to prove \cref{rel:f-left-cap}, since \cref{rel:f-left-cup} then follows from \cref{lem:flip-symmetry,cor:omega-left-cupcap}, together with the fact that the Nakayama automorphism of $F^\op$ is $\psi^{-1}$.

  If $k \ge 0$, \cref{rel:f-left-cap} follows immediately from the definition \cref{eq:left-cap-def} of the left cap $d'$, \cref{rel:dot-token-down-slide,eq:tokenslide-leftcross1,eq:tokenslide-leftcross2}.  When $k < 0$, we have
  \begin{multline*}
    \begin{tikzpicture}[anchorbase]
      \draw[<-] (0,0) -- (0,0.3) arc (180:0:.3) -- (0.6,0);
      \bluedot{(0.6,0.3)} node[anchor=west,color=black] {$\psi^k(f)$};
    \end{tikzpicture}
    \stackrel{\cref{rel:down-up-doublecross}}{=}
    \begin{tikzpicture}[anchorbase]
      \draw[<-] (0,0) .. controls (0,0.3) and (0.5,0.2) .. (0.5,0.5) .. controls (0.5,0.8) and (0,0.7) .. (0,1);
      \draw (0.5,0) .. controls (0.5,0.3) and (0,0.2) .. (0,0.5) .. controls (0,0.8) and (0.5,0.7) .. (0.5,1) arc(0:180:0.25);
      \bluedot{(0.5,1)} node[anchor=west,color=black] {$\psi^k(f)$};
    \end{tikzpicture}
    \ + \sum_{r=0}^{-k-1}
    \begin{tikzpicture}[anchorbase]
      \draw[->] (0,1.3) arc (180:360:0.3) -- (0.6,2) arc(0:180:0.3) -- (0,1.3);
      \draw[<-] (0,0) -- (0,0.2) arc (180:0:0.3) -- (0.6,0);
      \bluedot{(0.6,2)} node[anchor=west,color=black] {$\psi^k(f)$};
      \redcircle{(0.6,1.7)} node[anchor=east,color=black] {$r$};
      \bluedot{(0.6,1.4)} node[anchor=west,color=black] {$\chk{b}$};
      \greensquare{(0.3,0.5)} node[anchor=west,color=black] {\squarelabel{(r,b)}};
    \end{tikzpicture}
    \ \stackrel[\substack{\cref{rel:dot-token-up-slide} \\ \cref{eq:f-right-cupcap-slide}}]{\substack{\cref{rel:leftcurl-l-negative} \\ \cref{rel:ccc-l-negative}}}{=}
    \tr \left( \psi^{-1}(f) \chk{b} \right)\
    \begin{tikzpicture}[anchorbase]
      \draw[<-] (0,0) -- (0,0.2) arc (180:0:0.3) -- (0.6,0);
      \greensquare{(0.3,0.5)} node[anchor=west,color=black] {\squarelabel{(-k-1,b)}};
    \end{tikzpicture}
    \\
    \stackrel{\cref{eq:Nakayama-def}}{=} (-1)^{\bar f}\
    \begin{tikzpicture}[anchorbase]
      \draw[<-] (0,0) -- (0,0.2) arc (180:0:0.3) -- (0.6,0);
      \greensquare{(0.3,0.5)} node[anchor=west,color=black] {\squarelabel{(-k-1,f)}};
    \end{tikzpicture}
    \stackrel{\cref{eq:leftcap-square-dual-dot}}{=}
    \begin{tikzpicture}[anchorbase]
      \draw[<-] (0,0) -- (0,0.3) arc (180:0:.3) -- (0.6,0);
      \bluedot{(0.0,0.3)} node[anchor=east,color=black] {$f$};
    \end{tikzpicture}
    \ . \qedhere
  \end{multline*}
\end{proof}

\begin{lem} \label{lem:dotslide-left}
  The following relations hold:

  \noindent\begin{minipage}{0.5\linewidth}
    \begin{equation} \label{eq:dotslide-left1}
      \begin{tikzpicture}[anchorbase]
        \draw[<-] (0,0) -- (1,1);
        \draw[->] (1,0) -- (0,1);
        \redcircle{(0.25,0.25)};
      \end{tikzpicture}
      \ -\
      \begin{tikzpicture}[anchorbase]
        \draw[<-] (0,0) -- (1,1);
        \draw[->] (1,0) -- (0,1);
        \redcircle{(0.75,0.75)};
      \end{tikzpicture}
      \ =\
      \begin{tikzpicture}[anchorbase]
        \draw[<-] (0,1.1) -- (0,0.9) arc(180:360:0.3) -- (0.6,1.1);
        \draw[<-] (0,-0.1) -- (0,0.1) arc(180:0:0.3) -- (0.6,-0.1);
        \bluedot{(0.6,0.1)} node[anchor=west,color=black] {$\chk{b}$};
        \bluedot{(0,0.9)} node[anchor=west,color=black] {$b$};
      \end{tikzpicture}\ ,
    \end{equation}
  \end{minipage}%
  \begin{minipage}{0.5\linewidth}
    \begin{equation} \label{eq:dotslide-left2}
      \begin{tikzpicture}[anchorbase]
        \draw[<-] (0,0) -- (1,1);
        \draw[->] (1,0) -- (0,1);
        \redcircle{(0.75,0.25)};
      \end{tikzpicture}
      \ -\
      \begin{tikzpicture}[anchorbase]
        \draw[<-] (0,0) -- (1,1);
        \draw[->] (1,0) -- (0,1);
        \redcircle{(0.25,0.75)};
      \end{tikzpicture}
      \ = (-1)^{\bar b}
      \begin{tikzpicture}[anchorbase]
        \draw[<-] (0,1.1) -- (0,0.9) arc(180:360:0.3) -- (0.6,1.1);
        \draw[<-] (0,-0.1) -- (0,0.1) arc(180:0:0.3) -- (0.6,-0.1);
        \bluedot{(0,0.1)} node[anchor=east,color=black] {$b$};
        \bluedot{(0.6,0.9)} node[anchor=west,color=black] {$\chk{b}$};
      \end{tikzpicture}\ .
    \end{equation}
  \end{minipage}\par\vspace{\belowdisplayskip}
\end{lem}

\begin{proof}
  It suffices to prove \cref{eq:dotslide-left2}, since \cref{eq:dotslide-left1} then follows from \cref{lem:flip-symmetry,cor:omega-left-cupcap}.

  Composing \cref{rel:dotslide-right1} on the top and bottom with $t'$ we have
  \begin{gather*}
    \begin{tikzpicture}[anchorbase]
      \draw[<-] (0,0) .. controls (0,0.3) and (0.6,0.3) .. (0.6,0.6) .. controls (0.6,0.9) and (0,0.9) .. (0,1.2) .. controls (0,1.5) and (0.6,1.5) .. (0.6,1.8);
      \draw[->] (0.6,0) .. controls (0.6,0.3) and (0,0.3) .. (0,0.6) .. controls (0,0.9) and (0.6,0.9) .. (0.6,1.2) .. controls (0.6,1.5) and (0,1.5) .. (0,1.8);
      \redcircle{(0.6,1.2)};
    \end{tikzpicture}
    \ -\
    \begin{tikzpicture}[anchorbase]
      \draw[<-] (0,0) .. controls (0,0.3) and (0.6,0.3) .. (0.6,0.6) .. controls (0.6,0.9) and (0,0.9) .. (0,1.2) .. controls (0,1.5) and (0.6,1.5) .. (0.6,1.8);
      \draw[->] (0.6,0) .. controls (0.6,0.3) and (0,0.3) .. (0,0.6) .. controls (0,0.9) and (0.6,0.9) .. (0.6,1.2) .. controls (0.6,1.5) and (0,1.5) .. (0,1.8);
      \redcircle{(0,0.6)};
    \end{tikzpicture}
    = (-1)^{\bar b}\
    \begin{tikzpicture}[anchorbase]
      \draw[<-] (0,0) .. controls (0.25,0.25) and (0.5,0.25) .. (0.5,0.5) arc(0:180:0.25) .. controls (0,0.25) and (0.25,0.25) .. (0.5,0);
      \draw[<-] (0,1.8) .. controls (0.25,1.55) and (0.5,1.55) .. (0.5,1.3) arc(360:180:0.25) .. controls (0,1.55) and (0.25,1.55) .. (0.5,1.8);
      \bluedot{(0.5,1.3)} node[anchor=west,color=black] {$\chk{b}$};
      \bluedot{(0,0.5)} node[anchor=east,color=black] {$b$};
    \end{tikzpicture}
    \\
    \stackrel[\substack{\cref{eq:left-cup-def}, \cref{eq:left-cap-def} \\ \cref{rel:dot-right-cupcap}, \cref{eq:tokenslide-leftcross1}}]{\substack{\cref{rel:up-down-doublecross}, \cref{rel:down-up-doublecross} \\ \cref{rel:leftcurl-l-positive}, \cref{rel:leftcurl-l-negative}}}{\implies}
    \begin{tikzpicture}[anchorbase]
      \draw[<-] (0,0) -- (1,1);
      \draw[->] (1,0) -- (0,1);
      \redcircle{(0.75,0.25)};
    \end{tikzpicture}
    \ - \sum_{r=0}^{-k-1}\
    \begin{tikzpicture}[anchorbase]
      \draw[<-] (0,0) -- (0,0.2) arc (180:0:0.3) -- (0.6,0);
      \draw[<-] (0,1.7) .. controls (0.3,1.4) and (0.6,1.4) .. (0.6,1.1) arc(360:180:0.3) .. controls (0,1.4) and (0.3,1.4) .. (0.6,1.7);
      \greensquare{(0.3,0.5)} node[anchor=west,color=black] {\squarelabel{(r,b)}};
      \bluedot{(0.57,0.95)} node[anchor=west,color=black] {$\chk{b}$};
      \redcircle{(0.57,1.25)} node[anchor=west,color=black] {$r+1$};
    \end{tikzpicture}
    \ -\
    \begin{tikzpicture}[anchorbase]
      \draw[<-] (0,0) -- (1,1);
      \draw[->] (1,0) -- (0,1);
      \redcircle{(0.25,0.75)};
    \end{tikzpicture}
    \ + \sum_{r=0}^{k-1}
    \begin{tikzpicture}[anchorbase]
      \draw[<-] (0,1.7) -- (0,1.5) arc (180:360:0.3) -- (0.6,1.7);
      \draw[<-] (0,0) .. controls (0.3,0.3) and (0.6,0.3) .. (0.6,0.6) arc(0:180:0.3) .. controls (0,0.3) and (0.3,0.3) .. (0.6,0);
      \greensquare{(0.3,1.2)} node[anchor=west,color=black] {\squarelabel{(r,b)}};
      \bluedot{(0.03,0.75)} node[anchor=east,color=black] {$\chk{b}$};
      \redcircle{(0.03,0.45)};
      \redcircle{(0.6,0.6)} node[anchor=west,color=black] {$r$};
    \end{tikzpicture}
    \ = \delta_{k,0} (-1)^{\bar b}
    \begin{tikzpicture}[anchorbase]
      \draw[<-] (0,1.1) -- (0,0.9) arc(180:360:0.3) -- (0.6,1.1);
      \draw[<-] (0,-0.1) -- (0,0.1) arc(180:0:0.3) -- (0.6,-0.1);
      \bluedot{(0.6,0.1)} node[anchor=west,color=black] {$b$};
      \bluedot{(0,0.9)} node[anchor=west,color=black] {$\chk{b}$};
    \end{tikzpicture}\ .
  \end{gather*}
  Now, when $k < 0$, we have
  \[
    \sum_{r=0}^{-k-1}\
    \begin{tikzpicture}[anchorbase]
      \draw[<-] (0,0) -- (0,0.2) arc (180:0:0.3) -- (0.6,0);
      \draw[<-] (0,1.7) .. controls (0.3,1.4) and (0.6,1.4) .. (0.6,1.1) arc(360:180:0.3) .. controls (0,1.4) and (0.3,1.4) .. (0.6,1.7);
      \greensquare{(0.3,0.5)} node[anchor=west,color=black] {\squarelabel{(r,b)}};
      \bluedot{(0.57,0.95)} node[anchor=west,color=black] {$\chk{b}$};
      \redcircle{(0.57,1.25)} node[anchor=west,color=black] {$r+1$};
    \end{tikzpicture}
    \ \stackrel{\cref{rel:leftcurl-l-negative}}{=}\
    \begin{tikzpicture}[anchorbase]
      \draw[<-] (0,0) -- (0,0.2) arc (180:0:0.3) -- (0.6,0);
      \draw[<-] (0,1.7) .. controls (0.3,1.4) and (0.6,1.4) .. (0.6,1.1) arc(360:180:0.3) .. controls (0,1.4) and (0.3,1.4) .. (0.6,1.7);
      \greensquare{(0.3,0.5)} node[anchor=west,color=black] {\squarelabel{(-k-1,b)}};
      \bluedot{(0.57,0.95)} node[anchor=west,color=black] {$\chk{b}$};
      \redcircle{(0.57,1.25)} node[anchor=west,color=black] {$-k$};
    \end{tikzpicture}
    \ \stackrel[\substack{\cref{eq:left-cup-def} \\ \cref{eq:leftcap-square-dual-dot}}]{\substack{\cref{eq:f-right-cupcap-slide} \\ \cref{eq:tokenslide-leftcross2}}}{=} (-1)^{\bar b}\
    \begin{tikzpicture}[anchorbase]
      \draw[<-] (0,1.1) -- (0,0.9) arc(180:360:0.3) -- (0.6,1.1);
      \draw[<-] (0,-0.1) -- (0,0.1) arc(180:0:0.3) -- (0.6,-0.1);
      \bluedot{(0,0.1)} node[anchor=west,color=black] {$b$};
      \bluedot{(0.6,0.9)} node[anchor=west,color=black] {$\chk{b}$};
    \end{tikzpicture}\ .
  \]
  Similarly, when $k > 0$, we have
  \begin{multline*}
    \sum_{r=0}^{k-1}
    \begin{tikzpicture}[anchorbase]
      \draw[<-] (0,1.7) -- (0,1.5) arc (180:360:0.3) -- (0.6,1.7);
      \draw[<-] (0,0) .. controls (0.3,0.3) and (0.6,0.3) .. (0.6,0.6) arc(0:180:0.3) .. controls (0,0.3) and (0.3,0.3) .. (0.6,0);
      \greensquare{(0.3,1.2)} node[anchor=west,color=black] {\squarelabel{(r,b)}};
      \bluedot{(0.03,0.75)} node[anchor=east,color=black] {$\chk{b}$};
      \redcircle{(0.03,0.45)};
      \redcircle{(0.6,0.6)} node[anchor=west,color=black] {$r$};
    \end{tikzpicture}
    \ \stackrel[\cref{rel:dot-right-cupcap}]{\substack{\cref{rel:dot-token-up-slide} \\ \cref{eq:tokenslide-leftcross1}}}{=} \sum_{r=0}^{k-1}
    \begin{tikzpicture}[anchorbase]
      \draw[<-] (0,1.7) -- (0,1.5) arc (180:360:0.3) -- (0.6,1.7);
      \draw[<-] (0,0) .. controls (0.3,0.3) and (0.6,0.3) .. (0.6,0.6) arc(0:180:0.3) .. controls (0,0.3) and (0.3,0.3) .. (0.6,0);
      \greensquare{(0.3,1.2)} node[anchor=west,color=black] {\squarelabel{(r,b)}};
      \bluedot{(0.45,0.14)} node[anchor=west,color=black] {$\psi(\chk{b})$};
      \redcircle{(0,0.6)} node[anchor=east,color=black] {\dotlabel{r+1}};
    \end{tikzpicture}
    \ \stackrel{\cref{rel:leftcurl-l-positive}}{=}
    \begin{tikzpicture}[anchorbase]
      \draw[<-] (0,1.7) -- (0,1.5) arc (180:360:0.3) -- (0.6,1.7);
      \draw[<-] (0,0) .. controls (0.3,0.3) and (0.6,0.3) .. (0.6,0.6) arc(0:180:0.3) .. controls (0,0.3) and (0.3,0.3) .. (0.6,0);
      \greensquare{(0.3,1.2)} node[anchor=west,color=black] {\squarelabel{(k-1,b)}};
      \bluedot{(0.45,0.14)} node[anchor=west,color=black] {$\psi(\chk{b})$};
      \redcircle{(0,0.6)} node[anchor=east,color=black] {$k$};
    \end{tikzpicture}
    \\
    \stackrel[\cref{eq:leftcup-square-dual-dot}]{\cref{eq:left-cap-def}}{=} -
    \begin{tikzpicture}[anchorbase]
      \draw[<-] (0,1.1) -- (0,0.9) arc(180:360:0.3) -- (0.6,1.1);
      \draw[<-] (0,-0.1) -- (0,0.1) arc(180:0:0.3) -- (0.6,-0.1);
      \bluedot{(0.6,0.1)} node[anchor=west,color=black] {$\psi(\chk{b})$};
      \bluedot{(0,0.9)} node[anchor=west,color=black] {$b$};
    \end{tikzpicture}
    \ = - (-1)^{\bar b}\
    \begin{tikzpicture}[anchorbase]
      \draw[<-] (0,1.1) -- (0,0.9) arc(180:360:0.3) -- (0.6,1.1);
      \draw[<-] (0,-0.1) -- (0,0.1) arc(180:0:0.3) -- (0.6,-0.1);
      \bluedot{(0.6,0.1)} node[anchor=west,color=black] {$b$};
      \bluedot{(0,0.9)} node[anchor=west,color=black] {$\chk{b}$};
    \end{tikzpicture}
    \ ,
  \end{multline*}
  where in the last equality we used the fact that $\{(-1)^{\bar b} b : b \in B\}$ is the basis left dual to the basis $\{\psi(\chk{b}) : b \in B\}$.
  \details{
    For $a,b \in B$, we have
    \[
      \delta_{a,b}
      = \tr(\chk{b}a)
      = (-1)^{\bar a} \left( a \psi(\chk{b}) \right)
    \]
  }
  Relation \cref{eq:dotslide-left2} follows.
\end{proof}

\subsection{Proofs of relations}

We define the following \emph{negatively dotted bubbles} for $r < 0$ and $f \in F$:

\begin{equation} \label{eq:neg-clockwise-bubble}
  \begin{tikzpicture}[anchorbase]
    \draw[<-] (0,0.3) arc(90:450:0.3);
    \redcircle{(-0.3,0)} node[anchor=east,color=black] {$r$};
    \bluedot{(0.3,0)} node[anchor=west,color=black] {$f$};
  \end{tikzpicture}
  =
  \begin{cases}
    (-1)^{\bar f}\
    \begin{tikzpicture}[anchorbase]
      \draw[->] (0,-0.3) arc(-90:270:0.3);
      \redcircle{(0.3,0)} node[anchor=west,color=black] {$-k$};
      \greensquare{(0,0.3)} node[anchor=south,color=black] {\squarelabel{(-r-1,f)}};
    \end{tikzpicture}
    & \text{if } r > k-1, \\
    - \tr(f) & \text{if } r = k-1, \\
    0 & \text{if } r < k-1.
  \end{cases}
\end{equation}

\begin{equation} \label{eq:neg-cc-bubble}
  \begin{tikzpicture}[anchorbase]
    \draw[->] (0,0.3) arc(90:450:0.3);
    \redcircle{(0.3,0)} node[anchor=west,color=black] {$r$};
    \bluedot{(-0.3,0)} node[anchor=east,color=black] {$f$};
  \end{tikzpicture}
  =
  \begin{cases}
    -\
    \begin{tikzpicture}[anchorbase]
      \draw[<-] (0,0.3) arc(90:450:0.3);
      \redcircle{(0.3,0)} node[anchor=west,color=black] {$k$};
      \greensquare{(0,-0.3)} node[anchor=north,color=black] {\squarelabel{(-r-1,\psi^{-r}(f))}};
    \end{tikzpicture}
    & \text{if } r > -k-1, \\
    \tr(f) & \text{if } r = -k-1, \\
    0 & \text{if } r < -k-1.
  \end{cases}
\end{equation}
Note that these definitions are compatible with the action of $\omega$.

\begin{prop} \label{prop:inf-grass}
  The infinite grassmannian relations \cref{eq:inf-grass1,eq:inf-grass2,eq:inf-grass3} hold.
\end{prop}

\begin{proof}
  The relations \cref{eq:inf-grass1,eq:inf-grass2} and the first equality in \cref{eq:inf-grass3} follow immediately from \cref{rel:cc-l-positive,rel:ccc-l-negative,eq:neg-clockwise-bubble,eq:neg-cc-bubble}.

  It remains to prove the second equality in \cref{eq:inf-grass3}.   Using \cref{eq:f-Euler-commute}, it suffices to prove it in the case where $g=1$.  First consider the case $k \ge 0$.  When $t=0$, the middle term of \cref{eq:inf-grass3} becomes, using \cref{eq:inf-grass1,eq:inf-grass2}
  \[
    - \tr(fb) \tr(\chk{b})
    = - \tr \left( \tr(fb) \chk{b} \right)
    \stackrel{\cref{eq:f-in-dual-basis}}{=} - \tr(f).
  \]

  When $t>0$, we have
  \details{
    In the second line below (the expression after the second equality):
    \begin{itemize}
      \item the first sum corresponds to the terms from the first line with $s=-k-1$,
      \item the second (double) sum corresponds to the terms from the first line with $-k \le s \le -1$,
      \item the third (double) sum corrsponds to the terms from the first line with $s \ge 0$.
    \end{itemize}
    For the second (double) sum, we also changed to sum over the basis $\{\psi^{-u}(\chk{b}) : b \in B\}$, with left dual basis $\{(-1)^{\bar b} \psi^{-u-1}(b) : b \in B\}$.  So for the $-k \le s \le -1$ terms from the first line, we have
    \[
      \sum_{u=0}^{k-1} (-1)^{\bar b}\
      \begin{tikzpicture}[anchorbase]
        \draw[<-] (0,-0.5) arc(90:450:0.3);
        \draw[->] (0,0.3) arc(90:450:0.3);
        \redcircle{(-0.3,-0.8)} node[anchor=east,color=black] {\dotlabel{u+t-1}};
        \bluedot{(0.26,-0.65)} node[anchor=west,color=black] {$b$};
        \bluedot{(0.26,-0.95)} node[anchor=west,color=black] {$f$};
        \redcircle{(0.3,0)} node[anchor=west,color=black] {\dotlabel{-u-1}};
        \bluedot{(-0.3,0)} node[anchor=east,color=black] {$\chk{b}$};
      \end{tikzpicture}
      =
      \sum_{u=0}^{k-1}\
      \begin{tikzpicture}[anchorbase]
        \draw[<-] (0,-0.5) arc(90:450:0.3);
        \draw[->] (0,0.3) arc(90:450:0.3);
        \redcircle{(-0.3,-0.8)} node[anchor=east,color=black] {\dotlabel{u+t-1}};
        \bluedot{(0.26,-0.65)} node[anchor=west,color=black] {\dotlabel{\psi^{-u}(\chk{b})}};
        \bluedot{(0.26,-0.95)} node[anchor=west,color=black] {\dotlabel{f}};
        \redcircle{(0.3,0)} node[anchor=west,color=black] {\dotlabel{-u-1}};
        \bluedot{(-0.3,0)} node[anchor=east,color=black] {$\psi^{-u-1}(b)$};
      \end{tikzpicture}
      \stackrel{\mathclap{\cref{eq:neg-cc-bubble}}}{=} - \sum_{u=0}^{k-1}\
      \begin{tikzpicture}[anchorbase]
        \draw[<-] (0,0.5) arc(90:450:0.3);
        \draw[<-] (0,-0.5) arc(90:450:0.3);
        \redcircle{(-0.3,-0.8)} node[anchor=east,color=black] {\dotlabel{u+t-1}};
        \bluedot{(0.26,-0.65)} node[anchor=west,color=black] {\dotlabel{\psi^{-u}(\chk{b})}};
        \bluedot{(0.26,-0.95)} node[anchor=west,color=black] {\dotlabel{f}};
        \redcircle{(0.3,0.2)} node[anchor=west,color=black] {$k$};
        \greensquare{(0,-0.1)} node[anchor=north east,color=black] {\squarelabel{(u,b)}};
      \end{tikzpicture}
    \]
  }
  \begin{align*}
    \sum_{\substack{r,s \in \Z \\ r + s = t-2}} &\
    \begin{tikzpicture}[anchorbase]
      \draw[<-] (0,0.3) arc(90:450:0.3);
      \draw[->] (0,-0.5) arc(90:450:0.3);
      \redcircle{(-0.3,0)} node[anchor=east,color=black] {$r$};
      \bluedot{(0.3,0)} node[anchor=west,color=black] {$fb$};
      \redcircle{(0.3,-0.8)} node[anchor=west,color=black] {$s$};
      \bluedot{(-0.3,-0.8)} node[anchor=east,color=black] {$\chk{b}$};
    \end{tikzpicture}
    =
    \sum_{\substack{r,s \in \Z \\ r + s = t-2}} (-1)^{\bar b}\
    \begin{tikzpicture}[anchorbase]
      \draw[<-] (0,-0.5) arc(90:450:0.3);
      \draw[->] (0,0.3) arc(90:450:0.3);
      \redcircle{(-0.3,-0.8)} node[anchor=east,color=black] {$r$};
      \bluedot{(0.26,-0.65)} node[anchor=west,color=black] {$b$};
      \bluedot{(0.26,-0.95)} node[anchor=west,color=black] {$f$};
      \redcircle{(0.3,0)} node[anchor=west,color=black] {$s$};
      \bluedot{(-0.3,0)} node[anchor=east,color=black] {$\chk{b}$};
    \end{tikzpicture}
    \\
    &\stackrel{\mathclap{\cref{eq:neg-cc-bubble}}}{=}\
    (-1)^{\bar b} \tr(\chk{b})\
    \begin{tikzpicture}[anchorbase]
      \draw[<-] (0,0.3) arc(90:450:0.3);
      \redcircle{(-0.3,0)} node[anchor=north east,color=black] {\dotlabel{t+k-1}};
      \bluedot{(0.26,0.15)} node[anchor=west,color=black] {$b$};
      \bluedot{(0.26,-0.15)} node[anchor=west,color=black] {$f$};
    \end{tikzpicture}
    \ - \sum_{u=0}^{k-1}\
    \begin{tikzpicture}[anchorbase]
      \draw[<-] (0,0.5) arc(90:450:0.3);
      \draw[<-] (0,-0.5) arc(90:450:0.3);
      \redcircle{(-0.3,-0.8)} node[anchor=east,color=black] {\dotlabel{u+t-1}};
      \bluedot{(0.26,-0.65)} node[anchor=west,color=black] {\dotlabel{\psi^{-u}(\chk{b})}};
      \bluedot{(0.26,-0.95)} node[anchor=west,color=black] {\dotlabel{f}};
      \redcircle{(0.3,0.2)} node[anchor=west,color=black] {$k$};
      \greensquare{(0,-0.1)} node[anchor=north east,color=black] {\squarelabel{(u,b)}};
    \end{tikzpicture}
    + \sum_{\substack{r \ge -1,\, s \ge 0 \\ r + s = t-2}} (-1)^{\bar b}\
    \begin{tikzpicture}[anchorbase]
      \draw[<-] (0,-0.5) arc(90:450:0.3);
      \draw[->] (0,0.3) arc(90:450:0.3);
      \redcircle{(-0.3,-0.8)} node[anchor=east,color=black] {$r$};
      \bluedot{(0.26,-0.65)} node[anchor=west,color=black] {$b$};
      \bluedot{(0.26,-0.95)} node[anchor=west,color=black] {$f$};
      \redcircle{(0.3,0)} node[anchor=west,color=black] {$s$};
      \bluedot{(-0.3,0)} node[anchor=east,color=black] {$\chk{b}$};
    \end{tikzpicture}
    \\
    \ &\stackrel[\mathclap{\cref{eq:f-in-basis}}]{\mathclap{\substack{\cref{rel:dot-token-up-slide} \\ \cref{eq:f-right-cupcap-slide}}}}{=}
    \begin{tikzpicture}[anchorbase]
      \draw[<-] (0,0.3) arc(90:450:0.3);
      \redcircle{(-0.3,0)} node[anchor=north east,color=black] {\dotlabel{t+k-1}};
      \bluedot{(0.3,0)} node[anchor=west,color=black] {$f$};
    \end{tikzpicture}
    \ - \sum_{u=0}^{k-1} \
    \begin{tikzpicture}[anchorbase]
      \draw[<-] (0,0.5) arc(90:450:0.3);
      \draw[<-] (0.3,-0.8) arc(0:180:0.3) -- (-0.3,-1.6) arc(180:360:0.3) -- (0.3,-0.8);
      \redcircle{(-0.3,-0.8)} node[anchor=east,color=black] {$u$};
      \bluedot{(-0.3,-1.2)} node[anchor=east,color=black] {$\chk{b}$};
      \redcircle{(-0.3,-1.6)} node[anchor=east,color=black] {$t-1$};
      \bluedot{(0.3,-1.4)} node[anchor=west,color=black] {$f$};
      \redcircle{(0.3,0.2)} node[anchor=west,color=black] {$k$};
      \greensquare{(0,-0.1)} node[anchor=north east,color=black] {\squarelabel{(u,b)}};
    \end{tikzpicture}
    \ + \sum_{\substack{r \ge -1,\, s \ge 0 \\ r + s = t-2}} (-1)^{\bar b}\
    \begin{tikzpicture}[anchorbase]
      \draw[<-] (0,-0.5) arc(90:450:0.3);
      \draw[->] (0,0.3) arc(90:450:0.3);
      \redcircle{(-0.3,-0.8)} node[anchor=east,color=black] {$r$};
      \bluedot{(0.26,-0.65)} node[anchor=west,color=black] {$b$};
      \bluedot{(0.26,-0.95)} node[anchor=west,color=black] {$f$};
      \redcircle{(0.3,0)} node[anchor=west,color=black] {$s$};
      \bluedot{(-0.3,0)} node[anchor=east,color=black] {$\chk{b}$};
    \end{tikzpicture}
    \\
    &\stackrel{\mathclap{\cref{rel:up-down-doublecross}}}{=}
    \begin{tikzpicture}[anchorbase]
      \draw[-<-=.67] (-0.25,0) .. controls (-0.25,-0.3) and (0.25,-0.2) .. (0.25,-0.5) .. controls (0.25,-0.8) and (-0.25,-0.7) .. (-0.25,-1) arc (180:360:0.25) .. controls (0.25,-0.7) and (-0.25,-0.8) .. (-0.25,-0.5) .. controls (-0.25,-0.2) and (0.25,-0.3) .. (0.25,0) arc (0:180:0.25);
      \redcircle{(0.25,0)} node[anchor=west,color=black] {$k$};
      \redcircle{(-0.25,-1)} node[anchor=east,color=black] {\dotlabel{t-1}};
      \bluedot{(0.25,-1)} node[anchor=west,color=black] {$f$};
    \end{tikzpicture}
    \ + \sum_{\substack{r,s \ge 0 \\ r + s = t-2}} (-1)^{\bar b}\
    \begin{tikzpicture}[anchorbase]
      \draw[<-] (0,-0.5) arc(90:450:0.3);
      \draw[->] (0,0.3) arc(90:450:0.3);
      \redcircle{(-0.3,-0.8)} node[anchor=east,color=black] {$r$};
      \bluedot{(0.26,-0.65)} node[anchor=west,color=black] {$b$};
      \bluedot{(0.26,-0.95)} node[anchor=west,color=black] {$f$};
      \redcircle{(0.3,0)} node[anchor=west,color=black] {$s$};
      \bluedot{(-0.3,0)} node[anchor=east,color=black] {$\chk{b}$};
    \end{tikzpicture}
    \ + (-1)^{\bar b}\
    \begin{tikzpicture}[anchorbase]
      \draw[<-] (0,-0.5) arc(90:450:0.3);
      \draw[->] (0,0.3) arc(90:450:0.3);
      \redcircle{(-0.3,-0.8)} node[anchor=east,color=black] {$-1$};
      \bluedot{(0.26,-0.65)} node[anchor=west,color=black] {$b$};
      \bluedot{(0.26,-0.95)} node[anchor=west,color=black] {$f$};
      \redcircle{(0.3,0)} node[anchor=west,color=black] {\dotlabel{t-1}};
      \bluedot{(-0.3,0)} node[anchor=east,color=black] {$\chk{b}$};
    \end{tikzpicture}
    \\
    &\stackrel{\mathclap{\substack{\cref{eq:left-cap-def} \\ \cref{eq:inf-grass1}}}}{=}
    \begin{tikzpicture}[anchorbase]
      \draw[->] (-0.25,0) .. controls (-0.25,-0.3) and (0.25,-0.2) .. (0.25,-0.5) arc (360:180:0.25) .. controls (-0.25,-0.2) and (0.25,-0.3) .. (0.25,0) arc (0:180:0.25);
      \redcircle{(-0.25,-0.5)} node[anchor=east,color=black] {\dotlabel{t-1}};
      \bluedot{(0.26,-0.5)} node[anchor=west,color=black] {$f$};
    \end{tikzpicture}
    \ + \sum_{\substack{r,s \ge 0 \\ r + s = t-2}} (-1)^{\bar b}\
    \begin{tikzpicture}[anchorbase]
      \draw[<-] (0,-0.5) arc(90:450:0.3);
      \draw[->] (0,0.3) arc(90:450:0.3);
      \redcircle{(-0.3,-0.8)} node[anchor=east,color=black] {$r$};
      \bluedot{(0.26,-0.65)} node[anchor=west,color=black] {$b$};
      \bluedot{(0.26,-0.95)} node[anchor=west,color=black] {$f$};
      \redcircle{(0.3,0)} node[anchor=west,color=black] {$s$};
      \bluedot{(-0.3,0)} node[anchor=east,color=black] {$\chk{b}$};
    \end{tikzpicture}
    \ - \delta_{k,0} \tr(fb)\
    \begin{tikzpicture}[anchorbase]
      \draw[->] (0,0.3) arc(90:450:0.3);
      \redcircle{(0.3,0)} node[anchor=west,color=black] {\dotlabel{t-1}};
      \bluedot{(-0.3,0)} node[anchor=east,color=black] {$\chk{b}$};
    \end{tikzpicture}
    \\
    &\stackrel{\mathclap{\cref{eq:f-in-dual-basis}}}{=}
    \begin{tikzpicture}[anchorbase]
      \draw[->] (-0.25,0) .. controls (-0.25,-0.3) and (0.25,-0.2) .. (0.25,-0.5) arc (360:180:0.25) .. controls (-0.25,-0.2) and (0.25,-0.3) .. (0.25,0) arc (0:180:0.25);
      \redcircle{(-0.25,-0.5)} node[anchor=east,color=black] {\dotlabel{t-1}};
      \bluedot{(0.26,-0.5)} node[anchor=west,color=black] {$f$};
    \end{tikzpicture}
    \ + \sum_{\substack{r,s \ge 0 \\ r + s = t-2}} (-1)^{\bar b}\
    \begin{tikzpicture}[anchorbase]
      \draw[<-] (0,-0.5) arc(90:450:0.3);
      \draw[->] (0,0.3) arc(90:450:0.3);
      \redcircle{(-0.3,-0.8)} node[anchor=east,color=black] {$r$};
      \bluedot{(0.26,-0.65)} node[anchor=west,color=black] {$b$};
      \bluedot{(0.26,-0.95)} node[anchor=west,color=black] {$f$};
      \redcircle{(0.3,0)} node[anchor=west,color=black] {$s$};
      \bluedot{(-0.3,0)} node[anchor=east,color=black] {$\chk{b}$};
    \end{tikzpicture}
    \ - \delta_{k,0}\
    \begin{tikzpicture}[anchorbase]
      \draw[->] (0,0.3) arc(90:450:0.3);
      \redcircle{(0.3,0)} node[anchor=west,color=black] {\dotlabel{t-1}};
      \bluedot{(-0.3,0)} node[anchor=east,color=black] {$f$};
    \end{tikzpicture}
    \\
    &\stackrel[\mathclap{\cref{eq:f-right-cupcap-slide}}]{\mathclap{\cref{rel:dotslide-right1}}}{=}\ \
    \begin{tikzpicture}[anchorbase]
      \draw[->] (-0.25,0) .. controls (-0.25,-0.3) and (0.25,-0.2) .. (0.25,-0.5) arc (360:180:0.25) .. controls (-0.25,-0.2) and (0.25,-0.3) .. (0.25,0) arc (0:180:0.25);
      \redcircle{(0.25,0)} node[anchor=west,color=black] {\dotlabel{t-1}};
      \bluedot{(0.26,-0.5)} node[anchor=west,color=black] {$f$};
    \end{tikzpicture}
    \ - \delta_{k,0}\
    \begin{tikzpicture}[anchorbase]
      \draw[->] (0,0.3) arc(90:450:0.3);
      \redcircle{(0.3,0)} node[anchor=west,color=black] {\dotlabel{t-1}};
      \bluedot{(-0.3,0)} node[anchor=east,color=black] {$f$};
    \end{tikzpicture}
    \stackrel[\cref{eq:left-cup-def}]{\cref{rel:leftcurl-l-positive}}{=} \delta_{k,0} \left(
    \begin{tikzpicture}[anchorbase]
      \draw[->] (-0.25,0) .. controls (-0.25,-0.3) and (0.25,-0.2) .. (0.25,-0.5) .. controls (0.25,-0.8) and (-0.25,-0.7) .. (-0.25,-1) arc (180:360:0.25) .. controls (0.25,-0.7) and (-0.25,-0.8) .. (-0.25,-0.5) .. controls (-0.25,-0.2) and (0.25,-0.3) .. (0.25,0) arc (0:180:0.25);
      \redcircle{(0.25,0)} node[anchor=west,color=black] {\dotlabel{t-1}};
      \bluedot{(0.25,-0.5)} node[anchor=west,color=black] {$f$};
    \end{tikzpicture}
    \ -\
    \begin{tikzpicture}[anchorbase]
      \draw[->] (0,0.3) arc(90:450:0.3);
      \redcircle{(0.3,0)} node[anchor=west,color=black] {\dotlabel{t-1}};
      \bluedot{(-0.3,0)} node[anchor=east,color=black] {$f$};
    \end{tikzpicture}
    \right)
    \stackrel{\substack{\cref{eq:tokenslide-rightcross2} \\ \cref{rel:down-up-doublecross}}}{=} 0.
  \end{align*}

  Now suppose $k<0$.  If we let $\cdot$ denote multiplication in $F^\op$, we have
  \begin{align*}
    \omega &\left(
    \sum_{\substack{r,s \ge 0 \\ r + s = t}} \
    \begin{tikzpicture}[anchorbase]
      \draw[<-] (0,0.3) arc(90:450:0.3);
      \draw[->] (0,-0.5) arc(90:450:0.3);
      \redcircle{(-0.3,0)} node[anchor=east,color=black] {\dotlabel{r+k-1}};
      \bluedot{(0.3,0)} node[anchor=west,color=black] {$fb$};
      \redcircle{(0.3,-0.8)} node[anchor=west,color=black] {\dotlabel{s-k-1}};
      \bluedot{(-0.3,-0.8)} node[anchor=east,color=black] {$\chk{b}g$};
    \end{tikzpicture}
    \right)
    = \sum_{\substack{r,s \ge 0 \\ r + s = t}} (-1)^{(\bar f + \bar b)(\bar b + \bar g)}\
    \begin{tikzpicture}[anchorbase]
      \draw[<-] (0,0.3) arc(90:450:0.3);
      \draw[->] (0,-0.5) arc(90:450:0.3);
      \bluedot{(-0.3,0)} node[anchor=east,color=black] {$\chk{b}g$};
      \redcircle{(0.3,0)} node[anchor=west,color=black] {\dotlabel{s-k-1}};
      \bluedot{(0.3,-0.8)} node[anchor=west,color=black] {$fb$};
      \redcircle{(-0.3,-0.8)} node[anchor=east,color=black] {\dotlabel{r+k-1}};
    \end{tikzpicture}
    \\
    &\qquad \qquad \stackrel[\mathclap{\cref{rel:dot-right-cupcap}}]{\mathclap{\substack{\cref{rel:f-left-cap} \\ \cref{rel:f-left-cup}}}}{=}\ \sum_{\substack{r,s \ge 0 \\ r + s = t}} (-1)^{(\bar f + \bar b)(\bar b + \bar g)} \
    \begin{tikzpicture}[anchorbase]
      \draw[<-] (0,0.3) arc(90:450:0.3);
      \draw[->] (0,-0.5) arc(90:450:0.3);
      \redcircle{(-0.3,0)} node[anchor=east,color=black] {\dotlabel{s-k-1}};
      \bluedot{(0.3,0)} node[anchor=west,color=black] {$\psi^{-k}(\chk{b}) \psi^{-k}(g)$};
      \redcircle{(0.3,-0.8)} node[anchor=west,color=black] {\dotlabel{r+k-1}};
      \bluedot{(-0.3,-0.8)} node[anchor=east,color=black] {$\psi^{-k}(f) \psi^{-k}(b)$};
    \end{tikzpicture}
    \\
    &\qquad \qquad = \sum_{\substack{r,s \ge 0 \\ r + s = t}} (-1)^{\bar f \bar g + \bar b}\
    \begin{tikzpicture}[anchorbase]
      \draw[<-] (0,0.3) arc(90:450:0.3);
      \draw[->] (0,-0.5) arc(90:450:0.3);
      \redcircle{(-0.3,0)} node[anchor=east,color=black] {\dotlabel{s-k-1}};
      \bluedot{(0.3,0)} node[anchor=west,color=black] {$ \psi^{-k}(g) \cdot \psi^{-k}(\chk{b})$};
      \redcircle{(0.3,-0.8)} node[anchor=west,color=black] {\dotlabel{r+k-1}};
      \bluedot{(-0.3,-0.8)} node[anchor=east,color=black] {$\psi^{-k}(b) \cdot \psi^{-k}(f)$};
    \end{tikzpicture}
    \\
    &\qquad \qquad = - \delta_{t,0} (-1)^{\bar f \bar g}  \tr \left( \psi^{-k}(g) \cdot \psi^{-k}(f) \right)
    \\
    &\qquad \qquad = - \delta_{t,0} \tr(fg)
    = \omega \big( - \delta_{t,0} \tr(fg) \big),
  \end{align*}
  where, in the fourth equality, we use \cref{eq:inf-grass3} for $-k \ge 0$ case, together with the fact that $\{(-1)^{\bar b} \psi^{-k}(b) : b \in B\}$ is the basis for $F^\op$ left dual to the basis $\{\psi^{-k}(\chk{b}) : b \in B\}$.
  \details{
    For $a,b \in B$, we have
    \begin{multline*}
      (-1)^{\bar b} \tr \left( \psi^{-k}(b) \cdot \psi^{-k}(\chk{a}) \right)
      = (-1)^{\bar b + \bar a \bar b} \tr \left( \psi^{-k}(\chk{a}) \psi^{-k}(b) \right)
      \\
      = (-1)^{\bar b + \bar a \bar b} \tr ( \chk{a} b )
      = (-1)^{\bar b + \bar a \bar b} \delta_{a,b}
      = \delta_{a,b}.
    \end{multline*}
  }
  Since $\omega$ is an involution, \cref{eq:inf-grass3} follows.
\end{proof}

\begin{lem}
  The following relations hold:

  \noindent\begin{minipage}{0.45\linewidth}
    \begin{equation} \label{rel:dot-left-cap-slide}
      \begin{tikzpicture}[anchorbase]
        \draw[<-] (0,0) -- (0,0.3) arc (180:0:.3) -- (0.6,0);
        \redcircle{(0.6,0.3)};
      \end{tikzpicture}
      \ =\
      \begin{tikzpicture}[anchorbase]
        \draw[<-] (0,0) -- (0,0.3) arc (180:0:.3) -- (0.6,0);
        \redcircle{(0.0,0.3)};
      \end{tikzpicture}
      \ +
      \begin{tikzpicture}[anchorbase]
        \draw[<-] (0,-0.3) -- (0,0) arc (180:0:.3) -- (0.6,-0.3);
        \draw[->] (0.6,1) arc (90:-270:.3);
        \bluedot{(0.6,0)} node[anchor=west,color=black] {$\chk{b}$};
        \bluedot{(0.3,0.7)} node[anchor=east,color=black] {\dotlabel{\psi^{-1}(b) - b}};
        \redcircle{(0.9,0.7)} node[anchor=west,color=black] {$k$};
      \end{tikzpicture}\ ,
    \end{equation}
  \end{minipage}%
  \begin{minipage}{0.55\linewidth}
    \begin{equation} \label{rel:dot-left-cup-slide}
      \begin{tikzpicture}[anchorbase]
        \draw[<-] (0,0) -- (0,-0.3) arc (180:360:.3) -- (0.6,0);
        \redcircle{(0.6,-0.3)};
      \end{tikzpicture}
      \ =\
      \begin{tikzpicture}[anchorbase]
        \draw[<-] (0,0) -- (0,-0.3) arc (180:360:.3) -- (0.6,0);
        \redcircle{(0.0,-0.3)};
      \end{tikzpicture}
      - (-1)^{\bar b}
      \begin{tikzpicture}[anchorbase]
        \draw[<-] (0,0.3) -- (0,0) arc (180:360:.3) -- (0.6,0.3);
        \draw[->] (0.6,-0.5) arc (90:450:.3);
        \bluedot{(0.6,0)} node[anchor=west,color=black] {$\chk{b}$};
        \bluedot{(0.3,-0.8)} node[anchor=east,color=black] {\dotlabel{\psi^{-1}(b) - b}};
        \redcircle{(0.9,-0.8)} node[anchor=west,color=black] {$-k$};
      \end{tikzpicture}\ .
    \end{equation}
  \end{minipage}\par\vspace{\belowdisplayskip}

  \noindent In particular, if $\psi = \id$, then dots slide over left cups and caps.
\end{lem}

\begin{proof}
  It suffices to prove \cref{rel:dot-left-cap-slide}, since then \cref{rel:dot-left-cup-slide} follows by applying $\omega$.

  First suppose $k \ge 0$.  Then we have
  \[
    \begin{tikzpicture}[anchorbase]
      \draw[<-] (0,0) -- (0,0.3) arc (180:0:.3) -- (0.6,0);
      \redcircle{(0.6,0.3)};
    \end{tikzpicture}
    \ \stackrel{\cref{eq:left-cap-def}}{=}\
    \begin{tikzpicture}[anchorbase]
      \draw[<-] (-0.2,-0.2) .. controls (0.3,0.3) and (0.6,0.3) .. (0.6,0.6) arc(0:180:0.3) .. controls (0,0.3) and (0.3,0.3) .. (0.8,-0.2);
      \redcircle{(0.6,0.6)} node[anchor=west,color=black] {$k$};
      \redcircle{(0.55,0.03)};
    \end{tikzpicture}
    \stackrel[\cref{rel:dot-right-cupcap}]{\cref{eq:dotslide-left2}}{=}
    \begin{tikzpicture}[anchorbase]
      \draw[<-] (-0.2,-0.2) .. controls (0.3,0.3) and (0.6,0.3) .. (0.6,0.6) arc(0:180:0.3) .. controls (0,0.3) and (0.3,0.3) .. (0.8,-0.2);
      \redcircle{(0.6,0.6)} node[anchor=west,color=black] {$k+1$};
    \end{tikzpicture}
    \ + (-1)^{\bar b}
    \begin{tikzpicture}[anchorbase]
      \draw[<-] (0,-0.3) -- (0,0) arc (180:0:.3) -- (0.6,-0.3);
      \draw[->] (0.6,1) arc (90:-270:.3);
      \bluedot{(0,0)} node[anchor=east,color=black] {$b$};
      \bluedot{(0.86,0.55)} node[anchor=west,color=black] {$\chk{b}$};
      \redcircle{(0.86,0.85)} node[anchor=south west,color=black] {$k$};
    \end{tikzpicture}
    =
    \begin{tikzpicture}[anchorbase]
      \draw[<-] (-0.2,-0.2) .. controls (0.3,0.3) and (0.6,0.3) .. (0.6,0.6) arc(0:180:0.3) .. controls (0,0.3) and (0.3,0.3) .. (0.8,-0.2);
      \redcircle{(0.6,0.6)} node[anchor=west,color=black] {$k+1$};
    \end{tikzpicture}
    \ +
    \begin{tikzpicture}[anchorbase]
      \draw[<-] (0,-0.3) -- (0,0) arc (180:0:.3) -- (0.6,-0.3);
      \draw[->] (0.6,1) arc (90:-270:.3);
      \bluedot{(0.6,0)} node[anchor=west,color=black] {$\chk{b}$};
      \bluedot{(0.3,0.7)} node[anchor=east,color=black] {$\psi^{-1}(b)$};
      \redcircle{(0.9,0.7)} node[anchor=west,color=black] {$k$};
    \end{tikzpicture}
    \ .
  \]
  In the final equality above, we move the tokens labeled $b$ and $\chk{b}$ over the left cup and cap using \cref{rel:f-left-cap,rel:f-left-cup}.  Then we used the fact that $\{\psi^k(\chk{b}) : b \in B\}$ is the basis left dual to $\{\psi^k(b) : b \in B\}$ to remove the $\psi^k$.  Finally, we used \cref{eq:double-dual}.
  We also have
  \[
    \begin{tikzpicture}[anchorbase]
      \draw[<-] (0,0) -- (0,0.3) arc (180:0:.3) -- (0.6,0);
      \redcircle{(0.0,0.3)};
    \end{tikzpicture}
    \ \stackrel{\cref{eq:left-cap-def}}{=}\
    \begin{tikzpicture}[anchorbase]
      \draw[<-] (-0.2,-0.2) .. controls (0.3,0.3) and (0.6,0.3) .. (0.6,0.6) arc(0:180:0.3) .. controls (0,0.3) and (0.3,0.3) .. (0.8,-0.2);
      \redcircle{(0.6,0.6)} node[anchor=west,color=black] {$k$};
      \redcircle{(0.05,0.04)};
    \end{tikzpicture}
    \stackrel{\cref{eq:dotslide-left1}}{=}
    \begin{tikzpicture}[anchorbase]
      \draw[<-] (-0.2,-0.2) .. controls (0.3,0.3) and (0.6,0.3) .. (0.6,0.6) arc(0:180:0.3) .. controls (0,0.3) and (0.3,0.3) .. (0.8,-0.2);
      \redcircle{(0.6,0.6)} node[anchor=west,color=black] {$k+1$};
    \end{tikzpicture}
    \ +
    \begin{tikzpicture}[anchorbase]
      \draw[<-] (0,-0.3) -- (0,0) arc (180:0:.3) -- (0.6,-0.3);
      \draw[->] (0.6,1) arc (90:-270:.3);
      \bluedot{(0.6,0)} node[anchor=west,color=black] {$\chk{b}$};
      \bluedot{(0.3,0.7)} node[anchor=east,color=black] {$b$};
      \redcircle{(0.9,0.7)} node[anchor=west,color=black] {$k$};
    \end{tikzpicture}
  \]
  Relation \eqref{rel:dot-left-cap-slide} follows.

  Now suppose $k < 0$.  Composing both sides of \cref{rel:dot-left-cap-slide} on the bottom with the invertible map \cref{eq:inversion-relation-neg-l}, we see that it suffices to prove
  \begin{gather} \label{eq:dot-leftslide-proof1}
    \begin{tikzpicture}[anchorbase]
      \draw[->] (0,0) .. controls (0.3,0.3) and (0.6,0.3) .. (0.6,0.6) arc(0:180:0.3) .. controls (0,0.3) and (0.3,0.3) .. (0.6,0);
      \redcircle{(0.6,0.6)};
    \end{tikzpicture}
    =
    \begin{tikzpicture}[anchorbase]
      \draw[->] (0,0) .. controls (0.3,0.3) and (0.6,0.3) .. (0.6,0.6) arc(0:180:0.3) .. controls (0,0.3) and (0.3,0.3) .. (0.6,0);
      \redcircle{(0,0.6)};
    \end{tikzpicture}
    +
    \begin{tikzpicture}[anchorbase]
      \draw[<-] (0.6,-0.6) .. controls (0.3,-0.3) and (0,-0.3) .. (0,0) arc (180:0:.3) .. controls (0.6,-0.3) and (0.3,-0.3) .. (0,-0.6);
      \draw[->] (0.6,1) arc (90:-270:.3);
      \bluedot{(0.6,0)} node[anchor=west,color=black] {$\chk{b}$};
      \bluedot{(0.3,0.7)} node[anchor=east,color=black] {\dotlabel{\psi^{-1}(b)-b}};
      \redcircle{(0.9,0.7)} node[anchor=west,color=black] {$k$};
    \end{tikzpicture}
    \qquad \text{and}
    \\ \label{eq:dot-leftslide-proof2}
    \begin{tikzpicture}[anchorbase]
      \draw[->-=.15] (0.6,2) arc(0:180:0.3) -- (0,1.3) arc (180:360:0.3) -- (0.6,2);
      \redcircle{(0.6,1.9)} node[anchor=west,color=black] {$r+1$};
      \bluedot{(0.6,1.4)} node[anchor=east,color=black] {$\chk{a}$};
    \end{tikzpicture}
    =
    \begin{tikzpicture}[anchorbase]
      \draw[->-=.15] (0.6,2) arc(0:180:0.3) -- (0,1.3) arc (180:360:0.3) -- (0.6,2);
      \redcircle{(0.6,1.9)} node[anchor=west,color=black] {$r$};
      \bluedot{(0.6,1.4)} node[anchor=east,color=black] {$\chk{a}$};
      \redcircle{(0,2)};
    \end{tikzpicture}
    \ +
    \begin{tikzpicture}[anchorbase]
      \draw[<-] (0,0.3) arc(90:450:0.3);
      \draw[->] (-0.3,-0.8) -- (-0.3,-1.6) arc (180:360:0.3) -- (0.3,-0.8) arc (0:180:0.3);
      \redcircle{(0.3,0)} node[anchor=west,color=black] {$k$};
      \bluedot{(-0.3,0)} node[anchor=east,color=black] {\dotlabel{\psi^{-1}(b)-b}};
      \bluedot{(0.3,-0.8)} node[anchor=west,color=black] {$\chk{b}$};
      \redcircle{(0.3,-1.2)} node[anchor=west,color=black] {$r$};
      \bluedot{(0.3,-1.6)} node[anchor=west,color=black] {$\chk{a}$};
    \end{tikzpicture}
    \quad \text{for all } 0 \le r < -k,\ a \in B.
  \end{gather}
  All the terms in the sum in \cref{eq:dot-leftslide-proof1} are zero by \cref{rel:leftcurl-l-negative}.  Then we compute
  \begin{gather*}
    \begin{tikzpicture}[anchorbase]
      \draw[->] (0,0) .. controls (0.3,0.3) and (0.6,0.3) .. (0.6,0.6) arc(0:180:0.3) .. controls (0,0.3) and (0.3,0.3) .. (0.6,0);
      \redcircle{(0.6,0.6)};
    \end{tikzpicture}
    \ \stackrel{\cref{rel:dotslide-right1}}{=}\
    \begin{tikzpicture}[anchorbase]
      \draw[->] (0,0) .. controls (0.3,0.3) and (0.6,0.3) .. (0.6,0.6) arc(0:180:0.3) .. controls (0,0.3) and (0.3,0.3) .. (0.6,0);
      \redcircle{(0.15,0.15)};
    \end{tikzpicture}
    + (-1)^{\bar b}
    \begin{tikzpicture}[anchorbase]
      \draw[->] (0,-0.3) -- (0,0) arc (180:0:.3) -- (0.6,-0.3);
      \draw[<-] (0.6,1) arc (90:-270:.3);
      \bluedot{(0,0)} node[anchor=west,color=black] {$b$};
      \bluedot{(0.9,0.7)} node[anchor=west,color=black] {$\chk{b}$};
    \end{tikzpicture}
    \stackrel[\cref{eq:f-in-basis}]{\substack{\cref{rel:leftcurl-l-negative} \\ \cref{rel:ccc-l-negative}}}{=} \delta_{k,-1}\
    \begin{tikzpicture}[anchorbase]
      \draw[->] (0,-0.3) -- (0,0) arc (180:0:.3) -- (0.6,-0.3);
    \end{tikzpicture}
    \qquad \text{and}
    \\
    \begin{tikzpicture}[anchorbase]
      \draw[->] (0,0) .. controls (0.3,0.3) and (0.6,0.3) .. (0.6,0.6) arc(0:180:0.3) .. controls (0,0.3) and (0.3,0.3) .. (0.6,0);
      \redcircle{(0,0.6)};
    \end{tikzpicture}
    \ \stackrel{\cref{rel:dotslide-right2}}{=}\
    \begin{tikzpicture}[anchorbase]
      \draw[->] (0,0) .. controls (0.3,0.3) and (0.6,0.3) .. (0.6,0.6) arc(0:180:0.3) .. controls (0,0.3) and (0.3,0.3) .. (0.6,0);
      \redcircle{(0.45,0.15)};
    \end{tikzpicture}
    + \
    \begin{tikzpicture}[anchorbase]
      \draw[->] (0,-0.3) -- (0,0) arc (180:0:.3) -- (0.6,-0.3);
      \draw[<-] (0.6,1) arc (90:-270:.3);
      \bluedot{(0,0)} node[anchor=west,color=black] {$\chk{b}$};
      \bluedot{(0.9,0.7)} node[anchor=west,color=black] {$b$};
    \end{tikzpicture}
    \stackrel[\cref{eq:f-in-dual-basis}]{\substack{\cref{rel:leftcurl-l-negative} \\ \cref{rel:ccc-l-negative}}}{=} \delta_{k,-1}\
    \begin{tikzpicture}[anchorbase]
      \draw[->] (0,-0.3) -- (0,0) arc (180:0:.3) -- (0.6,-0.3);
    \end{tikzpicture}\ .
  \end{gather*}
  Thus \cref{eq:dot-leftslide-proof1} holds.

  It remains to prove \cref{eq:dot-leftslide-proof2}.  If $0 \le r \le -k-3$, then all terms in \cref{eq:dot-leftslide-proof2} are zero by \cref{rel:ccc-l-negative}.  Now suppose $r = -k-2$.  Then all the terms in the sum on the right side of \cref{eq:dot-leftslide-proof2} are zero by \cref{rel:ccc-l-negative}.  We also have
  \[
    \begin{tikzpicture}[anchorbase]
      \draw[->-=.15] (0.6,2) arc(0:180:0.3) -- (0,1.3) arc (180:360:0.3) -- (0.6,2);
      \redcircle{(0.6,1.9)} node[anchor=west,color=black] {$r+1$};
      \bluedot{(0.6,1.4)} node[anchor=east,color=black] {$\chk{a}$};
    \end{tikzpicture}
    \ \stackrel{\cref{eq:f-right-cupcap-slide}}{=}
    \begin{tikzpicture}[anchorbase]
      \draw[->] (0,0.3) arc(90:450:0.3);
      \redcircle{(0.3,0)} node[anchor=west,color=black] {$-k-1$};
      \bluedot{(-0.3,0)} node[anchor=east,color=black] {$\chk{a}$};
    \end{tikzpicture}
    \stackrel{\cref{rel:ccc-l-negative}}{=} \tr(\chk{a})
    = \tr(\psi(\chk{a}))
    \stackrel{\cref{rel:ccc-l-negative}}{=}
    \begin{tikzpicture}[anchorbase]
      \draw[->] (0,0.3) arc(90:450:0.3);
      \redcircle{(0.3,0)} node[anchor=west,color=black] {$-k-1$};
      \bluedot{(-0.3,0)} node[anchor=east,color=black] {$\psi(\chk{a})$};
    \end{tikzpicture}
    \stackrel[\cref{eq:f-right-cupcap-slide}]{\substack{\cref{rel:dot-right-cupcap} \\ \cref{rel:dot-token-up-slide}}}{=}
    \begin{tikzpicture}[anchorbase]
      \draw[->-=.15] (0.6,2) arc(0:180:0.3) -- (0,1.3) arc (180:360:0.3) -- (0.6,2);
      \redcircle{(0.6,1.9)} node[anchor=west,color=black] {$r$};
      \bluedot{(0.6,1.4)} node[anchor=east,color=black] {$\chk{a}$};
      \redcircle{(0,2)};
    \end{tikzpicture}
    \ ,
  \]
  so \cref{eq:dot-leftslide-proof2} holds.  Finally, consider the case $r=-k-1$.  Then we have
  \begin{multline*}
    \begin{tikzpicture}[anchorbase]
      \draw[<-] (0,0.3) arc(90:450:0.3);
      \draw[->] (-0.3,-0.8) -- (-0.3,-1.6) arc (180:360:0.3) -- (0.3,-0.8) arc (0:180:0.3);
      \redcircle{(0.3,0)} node[anchor=west,color=black] {$k$};
      \bluedot{(-0.3,0)} node[anchor=east,color=black] {\dotlabel{\psi^{-1}(b)-b}};
      \bluedot{(0.3,-0.8)} node[anchor=west,color=black] {$\chk{b}$};
      \redcircle{(0.3,-1.2)} node[anchor=west,color=black] {$r$};
      \bluedot{(0.3,-1.6)} node[anchor=west,color=black] {$\chk{a}$};
    \end{tikzpicture}
    \stackrel[\substack{\cref{rel:ccc-l-negative} \\ \cref{eq:f-in-basis}}]{\substack{\cref{rel:dot-token-up-slide} \\ \cref{eq:f-right-cupcap-slide}}}{=}
    \begin{tikzpicture}[anchorbase]
      \draw[<-] (0,0.3) arc(90:450:0.3);
      \redcircle{(0.3,0)} node[anchor=west,color=black] {$k$};
      \bluedot{(-0.3,0)} node[anchor=east,color=black] {\dotlabel{\psi^{-r-1}(\chk{a})-\psi^{-r}(\chk{a})}};
    \end{tikzpicture}
    \\
    \stackrel{\cref{eq:central-bubble-reverse}}{=}
    \begin{tikzpicture}[anchorbase]
      \draw[->] (0,0.3) arc(90:450:0.3);
      \redcircle{(-0.3,0)} node[anchor=east,color=black] {$-k$};
      \bluedot{(0.3,0)} node[anchor=west,color=black] {\dotlabel{\psi^{-r-1}(\chk{a})-\psi^{-r}(\chk{a})}};
    \end{tikzpicture}
    \ \stackrel[\cref{rel:dot-right-cupcap}]{\cref{rel:dot-token-up-slide}}{=}
    \begin{tikzpicture}[anchorbase]
      \draw[->-=.15] (0.6,2) arc(0:180:0.3) -- (0,1.3) arc (180:360:0.3) -- (0.6,2);
      \redcircle{(0.6,1.9)} node[anchor=west,color=black] {$r+1$};
      \bluedot{(0.6,1.4)} node[anchor=east,color=black] {$\chk{a}$};
    \end{tikzpicture}
    -
    \begin{tikzpicture}[anchorbase]
      \draw[->-=.15] (0.6,2) arc(0:180:0.3) -- (0,1.3) arc (180:360:0.3) -- (0.6,2);
      \redcircle{(0.6,1.9)} node[anchor=west,color=black] {$r$};
      \bluedot{(0.6,1.4)} node[anchor=east,color=black] {$\chk{a}$};
      \redcircle{(0,2)};
    \end{tikzpicture}
  \end{multline*}
  \details{
    Using \cref{rel:dot-token-up-slide,eq:f-right-cupcap-slide,rel:ccc-l-negative}, the lower bubble in the first sum is equal to $\tr(\psi^r(\chk{b})\chk{a}) = \tr(\chk{b}\psi^{-r}(\chk{a}))$.  Then we use \cref{eq:f-in-basis}.
  }
  and so \cref{eq:dot-leftslide-proof2} holds.
\end{proof}

\begin{lem} \label{lem:bubble-determinant-identities}
  Recall our convention for computing determinants from \cref{eq:determinant-convention}.  For all $f \in F$ and $r > 0$, we have
  \begin{gather} \label{eq:ccbubble-in-terms-of-cbubble}
    \ccbubble{$f$}{\dotlabel{r-k-1}}
    = \sum_{b_1,\dotsc,b_{r-1} \in B} \det
    \left( \cbubble{$\chk{b}_{j-1}b_j$}{\dotlabel{i-j+k}} \right)_{i,j=1}^r,
    \\ \label{eq:cbubble-in-terms-of-ccbubble}
    \cbubble{$f$}{\dotlabel{r+k-1}}
    = (-1)^{r+1} \sum_{b_1,\dotsc,b_{r-1} \in B} \det
    \left( \ccbubble{$\chk{b}_{j-1}b_j$}{\dotlabel{i-j-k}} \right)_{i,j=1}^r,
  \end{gather}
  where we adopt the convention that $\chk{b}_0 = f$ and $b_r = 1$.
\end{lem}

\begin{proof}
  We prove \cref{eq:ccbubble-in-terms-of-cbubble} and the leave the proof of \cref{eq:cbubble-in-terms-of-ccbubble}, which is similar, to the reader.  The case $r=1$ follows immediately from \cref{eq:central-bubble-reverse}.  Assume that $r>1$ and that the result holds for $r-1$.  For $b_1,b_2,\dotsc,b_{r-1} \in B$, define the matrix
  \[
    A = \left( \cbubble{$\chk{b}_{j-1}b_j$}{\dotlabel{i-j+k}} \right)_{i,j=1}^r.
  \]
  (We leave the dependence on $b_1,b_2,\dotsc,b_{r-1} \in B$ implicit to simplify the notation.)  We have
  \[
    \det A
    = \sum_{t=1}^r (-1)^{t+1} \cbubble{$fb_1$}{\dotlabel{t+k-1}} \det A_{t,1}.
  \]
  If we consider $A_{t,1}$ as a block matrix with upper-left block of size $(t-1) \times (t-1)$ and lower-right block of size $(r-t) \times (r-t)$, we see that it is block lower triangular.  By \cref{eq:inf-grass1}, the upper-left block is lower triangular with diagonal entries
  \[
    -\tr(\chk{b}_1 b_2),\ -\tr(\chk{b}_2 b_3),\ \dotsc\ ,\ -\tr(\chk{b}_{t-1} b_t).
  \]
  On the other hand, the lower-right block is the matrix
  \[
    \left(
      \cbubble{$\chk{b}_{t+j-1} b_{t+j}$}{\dotlabel{i - j + k}}
    \right)_{i,j=1}^{r-t}.
  \]
  Thus, using the induction hypothesis and \cref{eq:dual-basis-def}, we have
  \begin{equation} \label{eq:det-expansion}
    \sum_{b_1,\dotsc,b_{r-1} \in B} \det A
    = \sum_{t=1}^r \
    \begin{tikzpicture}[anchorbase]
      \draw[->] (0,0.3) arc (90:-270:0.3);
      \bluedot{(0.3,0)} node[anchor=west,color=black] {$fb_1$};
      \redcircle{(-0.3,0)} node[anchor=east] {\dotlabel{t+k-1}};
      \draw[->] (0,-0.5) arc (90:450:0.3);
      \bluedot{(-0.3,-0.8)} node[anchor=east,color=black] {$\chk{b}_1$};
      \redcircle{(0.3,-0.8)} node[anchor=west,color=black] {\dotlabel{r-t-k-1}};
    \end{tikzpicture}
    \stackrel[\cref{eq:f-in-dual-basis}]{\substack{\cref{eq:inf-grass1} \\ \cref{eq:inf-grass3}}}{=}
    \ccbubble{$f$}{\dotlabel{r-k-1}}\ . \qedhere
  \end{equation}
  \details{
    We prove \cref{eq:cbubble-in-terms-of-ccbubble}.  The case $r=1$ follows immediately from \cref{eq:central-bubble-reverse}.  Assume that $r>1$ and that the result holds for $r-1$.  For $b_1,b_2,\dotsc,b_{r-1} \in B$, define the matrix
    \[
      A = \left( \ccbubble{$\chk{b}_{j-1}b_j$}{\dotlabel{i-j-k}} \right)_{i,j=1}^r.
    \]
    (We leave the dependence on $b_1,b_2,\dotsc,b_{r-1} \in B$ implicit to simplify the notation.)  We have
    \[
      \det A
      = \sum_{t=1}^r (-1)^{t+1} \ccbubble{$fb_1$}{\dotlabel{t-k-1}} \det A_{t,1}.
    \]
    If we consider $A_{t,1}$ as a block matrix with upper-left block of size $(t-1) \times (t-1)$ and lower-right block of size $(r-t) \times (r-t)$, we see that it is block lower triangular.  By \cref{eq:inf-grass2}, the upper-left block is lower triangular with diagonal entries
    \[
      \tr(\chk{b}_1 b_2),\ \tr(\chk{b}_2 b_3),\ \dotsc\ ,\ \tr(\chk{b}_{t-1} b_t).
    \]
    On the other hand, the lower-right block is the matrix
    \[
      \left(
        \ccbubble{$\chk{b}_{t+j-1} b_{t+j}$}{\dotlabel{i - j - k}}
      \right)_{i,j=1}^{r-t}.
    \]
    Thus, using the induction hypothesis and \cref{eq:dual-basis-def}, we have
    \[
      \sum_{b_1,\dotsc,b_{r-1} \in B} \det A
      = (-1)^{r+1} \sum_{t=1}^r\
      \begin{tikzpicture}[anchorbase]
        \draw[<-] (0,0.3) arc (90:-270:0.3);
        \bluedot{(-0.3,0)} node[anchor=east,color=black] {$fb_1$};
        \redcircle{(0.3,0)} node[anchor=west,color=black] {\dotlabel{t-k-1}};
        \draw[<-] (0,-0.5) arc (90:450:0.3);
        \bluedot{(0.3,-0.8)} node[anchor=west,color=black] {$\chk{b}_1$};
        \redcircle{(-0.3,-0.8)} node[anchor=east,color=black] {\dotlabel{r-t+k-1}};
      \end{tikzpicture}
      \stackrel[\cref{eq:f-in-dual-basis}]{\substack{\cref{eq:inf-grass2} \\ \cref{eq:inf-grass3}}}{=}
      (-1)^{r+1}\ \ccbubble{$f$}{\dotlabel{r-k-1}}\ . \qedhere
    \]
  }
\end{proof}

\begin{lem} \label{lem:square-left-cap-cup}
  For $f \in F$, we have
  \begin{equation} \label{eq:leftcap-expression}
    \begin{tikzpicture}[>=To,baseline={([yshift=1ex]current bounding box.center)}]
      \draw[<-] (0,0.2) -- (0,0) arc (180:360:.3) -- (0.6,0.2);
      \greensquare{(0.3,-0.3)} node[anchor=north,color=black] {\squarelabel{(r,f)}};
    \end{tikzpicture}
    = -\sum_{s \ge 0} \
    \begin{tikzpicture}[anchorbase]
      \draw[<-] (0,0.3) -- (0,0) arc (180:360:.3) -- (0.6,0.3);
      \draw[->] (0.6,-0.4) arc (90:450:.3);
      \redcircle{(0.6,0)} node[anchor=east,color=black] {$s$};
      \bluedot{(0,0)} node[anchor=east,color=black] {$a$};
      \bluedot{(0.9,-0.7)} node[anchor=west,color=black] {$\chk{a}f$};
      \redcircle{(0.3,-0.7)} node[anchor=east,color=black] {\dotlabel{-r-s-2}};
    \end{tikzpicture}
    \quad, \quad 0 \le r < k,
  \end{equation}
  and
  \begin{equation} \label{eq:leftcup-expression}
    \begin{tikzpicture}[>=To,baseline={([yshift=-2ex]current bounding box.center)}]
      \draw[<-] (0,-0.2) -- (0,0) arc (180:0:.3) -- (0.6,-0.2);
      \greensquare{(0.3,0.3)} node[anchor=south,color=black] {\squarelabel{(r,f)}};
    \end{tikzpicture}
    = -\sum_{s \ge 0} (-1)^{\bar a \bar f + \bar a + \bar f}\
    \begin{tikzpicture}[anchorbase]
      \draw[<-] (0,-0.3) -- (0,0) arc (180:0:.3) -- (0.6,-0.3);
      \draw[->] (0.6,1) arc (90:-270:.3);
      \redcircle{(0.6,0)} node[anchor=west,color=black] {$s$};
      \bluedot{(0,0)} node[anchor=east,color=black] {$a$};
      \bluedot{(0.9,0.7)} node[anchor=west,color=black] {$\chk{a}f$};
      \redcircle{(0.3,0.7)} node[anchor=east,color=black] {\dotlabel{-r-s-2}};
    \end{tikzpicture}
    \quad, \quad 0 \le r < -k.
  \end{equation}
\end{lem}

\begin{proof}
  By linearity, it suffices to prove the result for $f = b \in B$.  Suppose $0 \le r < k$.  Using \cref{rel:f-left-cap,rel:f-left-cup,rel:dot-right-cupcap}, together with the fact that the sum over $a \in B$ is independent of the basis, we have that the right side of \cref{eq:leftcap-expression} is equal to
  \begin{multline} \label{eq:leftcup-proof-big-sum}
    - \sum_{s \ge 0}\
    \begin{tikzpicture}[anchorbase]
      \draw[<-] (0,0.6) -- (0,0) arc (180:360:.3) -- (0.6,0.6);
      \draw[->] (0.6,-0.4) arc (90:450:.3);
      \redcircle{(0.6,0.3)} node[anchor=east,color=black] {$s$};
      \bluedot{(0.6,0)} node[anchor=west,color=black] {$a$};
      \bluedot{(0.3,-0.7)} node[anchor=east,color=black] {$\chk{a}\psi^{-k}(b)$};
      \redcircle{(0.9,-0.7)} node[anchor=west,color=black] {\dotlabel{-r-s-2}};
    \end{tikzpicture}
    \\
    \stackrel[\cref{rel:f-left-cup}]{\substack{\cref{rel:up-down-doublecross} \\ \cref{rel:dot-right-cupcap}}}{=}
    -\sum_{s \ge 0}\
    \begin{tikzpicture}[anchorbase]
      \draw[<-] (0,0) .. controls (0,-0.3) and (0.5,-0.2) .. (0.5,-0.5) .. controls (0.5,-0.8) and (0,-0.7) .. (0,-1) -- (0,-1.2) arc (180:360:0.25) -- (0.5,-1) .. controls (0.5,-0.7) and (0,-0.8) .. (0,-0.5) .. controls (0,-0.2) and (0.5,-0.3) .. (0.5,0);
      \draw[->] (0.5,-1.6) arc (90:450:.3);
      \redcircle{(0.47,-0.9)} node[anchor=west,color=black] {$s$};
      \bluedot{(0.47,-1.3)} node[anchor=west,color=black] {$a$};
      \bluedot{(0.2,-1.9)} node[anchor=east,color=black] {$\chk{a}\psi^{-k}(b)$};
      \redcircle{(0.8,-1.9)} node[anchor=west,color=black] {\dotlabel{-r-s-2}};
    \end{tikzpicture}
    \ - \sum_{s \ge 0} \sum_{j=0}^{k-1}\
    \begin{tikzpicture}[anchorbase]
      \draw[<-] (0,0.2) -- (0,0) arc (180:360:0.3) -- (0.6,0.2);
      \draw[<-] (0.3,-0.5) arc (90:450:0.3);
      \draw[->] (0.6,-1.2) arc (90:450:0.3);
      \greensquare{(0.3,-0.3)} node[anchor=west,color=black] {\squarelabel{(j,e)}};
      \redcircle{(0,-0.8)} node[anchor=east,color=black] {\dotlabel{j+s}};
      \bluedot{(0.6,-0.8)} node[anchor=west,color=black] {$\psi^{-k}(\chk{e})a$};
      \bluedot{(0.3,-1.5)} node[anchor=east,color=black] {$\chk{a}\psi^{-k}(b)$};
      \redcircle{(0.9,-1.5)} node[anchor=west,color=black] {\dotlabel{-r-s-2}};
    \end{tikzpicture}\ .
  \end{multline}
  By \cref{eq:inf-grass1,eq:inf-grass2},
  \begin{equation} \label{eq:leftcup-proof-bubbles}
    \begin{tikzpicture}[anchorbase]
      \draw[<-] (0.3,-0.5) arc (90:450:0.3);
      \draw[->] (0.6,-1.2) arc (90:450:0.3);
      \redcircle{(0,-0.8)} node[anchor=east,color=black] {\dotlabel{j+s}};
      \bluedot{(0.6,-0.8)} node[anchor=west,color=black] {$\psi^{-k}(\chk{e})a$};
      \bluedot{(0.3,-1.5)} node[anchor=east,color=black] {$\chk{a}\psi^{-k}(b)$};
      \redcircle{(0.9,-1.5)} node[anchor=west,color=black] {\dotlabel{-r-s-2}};
    \end{tikzpicture}
  \end{equation}
  is equal to zero unless $k - 1 - j \le s \le k-1 - r$.  In particular, \cref{eq:leftcup-proof-bubbles} is zero unless $j \ge r$.  If $j \ge r$, then
  \begin{multline*}
    \sum_{s \ge 0}\
    \begin{tikzpicture}[anchorbase]
      \draw[<-] (0.3,-0.5) arc (90:450:0.3);
      \draw[->] (0.6,-1.2) arc (90:450:0.3);
      \redcircle{(0,-0.8)} node[anchor=east,color=black] {\dotlabel{j+s}};
      \bluedot{(0.6,-0.8)} node[anchor=west,color=black] {$\psi^{-k}(\chk{e})a$};
      \bluedot{(0.3,-1.5)} node[anchor=east,color=black] {$\chk{a}\psi^{-k}(b)$};
      \redcircle{(0.9,-1.5)} node[anchor=west,color=black] {\dotlabel{-r-s-2}};
    \end{tikzpicture}
    \stackrel{\cref{eq:inf-grass2}}{=} \sum_{\substack{u,s \in \Z \\ u + s = j-r-2}}\
    \begin{tikzpicture}[anchorbase]
      \draw[<-] (0.3,-0.5) arc (90:450:0.3);
      \draw[->] (0.6,-1.2) arc (90:450:0.3);
      \redcircle{(0,-0.8)} node[anchor=east,color=black] {$u$};
      \bluedot{(0.6,-0.8)} node[anchor=west,color=black] {$\psi^{-k}(\chk{e})a$};
      \bluedot{(0.3,-1.5)} node[anchor=east,color=black] {$\chk{a}\psi^{-k}(b)$};
      \redcircle{(0.9,-1.5)} node[anchor=west,color=black] {$s$};
    \end{tikzpicture}
    \\
    \stackrel{\cref{eq:inf-grass3}}{=} - \delta_{j,r} \tr \left( \psi^{-k}(\chk{e})\psi^{-k}(b) \right)
    = - \delta_{j,r} \delta_{b,e}.
  \end{multline*}

  Now consider the sum
  \begin{equation} \label{eq:leftcup-proof-term-vanish}
    \sum_{s \ge 0}\
    \begin{tikzpicture}[anchorbase]
      \draw[<-] (0,0) .. controls (0,-0.3) and (0.5,-0.2) .. (0.5,-0.5) .. controls (0.5,-0.8) and (0,-0.7) .. (0,-1) -- (0,-1.2) arc (180:360:0.25) -- (0.5,-1) .. controls (0.5,-0.7) and (0,-0.8) .. (0,-0.5) .. controls (0,-0.2) and (0.5,-0.3) .. (0.5,0);
      \draw[->] (0.5,-1.6) arc (90:450:.3);
      \redcircle{(0.47,-0.9)} node[anchor=west,color=black] {$s$};
      \bluedot{(0.47,-1.3)} node[anchor=west,color=black] {$a$};
      \bluedot{(0.2,-1.9)} node[anchor=east,color=black] {$\chk{a}f$};
      \redcircle{(0.8,-1.9)} node[anchor=west,color=black] {\dotlabel{-r-s-2}};
    \end{tikzpicture}\ .
  \end{equation}
  The terms with $s > k-1-r$ are zero by \cref{eq:inf-grass2}.  On the other hand, for $0 \le s \le k - 1 - r$, we have
  \[
    \begin{tikzpicture}[anchorbase]
      \draw[<-] (0,0) .. controls (0,-0.3) and (0.5,-0.2) .. (0.5,-0.5) .. controls (0.5,-0.8) and (0,-0.7) .. (0,-1) -- (0,-1.2) arc (180:360:0.25) -- (0.5,-1) .. controls (0.5,-0.7) and (0,-0.8) .. (0,-0.5) .. controls (0,-0.2) and (0.5,-0.3) .. (0.5,0);
      \redcircle{(0.47,-0.9)} node[anchor=west,color=black] {$s$};
      \bluedot{(0.47,-1.3)} node[anchor=west,color=black] {$a$};
    \end{tikzpicture}
    \ \stackrel{\cref{rel:dotslide-right2}}{=}\
    \begin{tikzpicture}[anchorbase]
      \draw[<-] (0,0) .. controls (0,-0.3) and (0.5,-0.2) .. (0.5,-0.5) -- (0.5,-0.7) .. controls (0.5,-1) and (0,-0.9) .. (0,-1.2) arc (180:360:0.25) .. controls (0.5,-0.9) and (0,-1) .. (0,-0.7) -- (0,-0.5) .. controls (0,-0.2) and (0.5,-0.3) .. (0.5,0);
      \redcircle{(0,-0.5)} node[anchor=south east,color=black] {$s$};
      \bluedot{(0.47,-1.3)} node[anchor=west,color=black] {$a$};
    \end{tikzpicture}
    \ - \sum_{t=0}^{s-1}
    \begin{tikzpicture}[anchorbase]
      \draw[<-] (0,1.8) .. controls (0.25,1.55) and (0.5,1.55) .. (0.5,1.3) arc(360:180:0.25) .. controls (0,1.55) and (0.25,1.55) .. (0.5,1.8);
      \draw[<-] (0.25,0.8) arc(90:180:0.25) -- (0,.25) arc(180:360:0.25) -- (0.5,0.55) arc(0:90:0.25);
      \redcircle{(0,1.3)} node[anchor=east,color=black] {$t$};
      \redcircle{(0.5,0.55)} node[anchor=west,color=black] {\dotlabel{s-t-1}};
      \bluedot{(0.5,0.25)} node[anchor=west,color=black] {$a$};
      \bluedot{(0.5,1.3)} node[anchor=west,color=black] {$b$};
      \bluedot{(0,0.55)} node[anchor=east,color=black] {$\chk{b}$};
    \end{tikzpicture}
    \stackrel[\cref{eq:inf-grass1}]{\cref{rel:leftcurl-l-positive}}{=} 0
  \]
  Thus, the sum \cref{eq:leftcup-proof-term-vanish} is equal to zero.  Combined with the above, this proves \cref{eq:leftcap-expression}.

  The relation \cref{eq:leftcup-expression} follows from \cref{eq:leftcap-expression} by applying $\omega$.
  \details{
    By \cref{eq:omega-greenbox}, applying $\omega$ to the left side of \cref{eq:leftcap-expression} yields the inverted diagram (i.e.\ flipped in the horizontal axis) times $(-1)^{\bar f}$.  Applying $\omega$ to the right side of \cref{eq:leftcap-expression} yields the inverted diagram times $(-1)^{\bar a (\bar f + \bar a)} = (-1)^{\bar a \bar f + \bar a}$.
  }
\end{proof}

Using \cref{lem:square-left-cap-cup,eq:inf-grass1,eq:inf-grass2}, relations \cref{rel:up-down-doublecross,rel:down-up-doublecross} become
\begin{equation} \label{rel:up-down-doublecross-new}

  \text{ if } k \ge -1 \right).
\end{equation}

\begin{lem}
  The curl relations \cref{rel:left-dotted-curl,rel:right-dotted-curl} hold.
\end{lem}

\begin{proof}
  We first prove that

  \noindent\begin{minipage}{0.5\linewidth}
    \begin{equation} \label{eq:vertical-left-curl}
      %
\ .
    \end{equation}
  \end{minipage}\par\vspace{\belowdisplayskip}

  \noindent It suffices to prove \cref{eq:vertical-left-curl}, since \cref{eq:vertical-right-curl} then follows by applying $\omega$.  When $k < 0$, \cref{eq:vertical-left-curl} holds by \cref{rel:leftcurl-l-negative}, so we assume $k \ge 0$.  Then we have
  \begin{align*}
    %
\ .
  \end{align*}
  The final expression is equal to the right side of \cref{eq:vertical-left-curl} by \cref{eq:inf-grass2,eq:f-in-basis}.
  \details{
    Although the token and dot are on the opposite sides of the bubble above, when compared to \cref{eq:inf-grass2}, we can use \cref{rel:f-left-cap,rel:dot-right-cupcap} and the fact that the trace map is invariant under $\psi$.
  }
  This completes the proof of \cref{eq:vertical-left-curl}.

  Now we have
  \[
    %
\ .
  \]
  Placing an upward strand on the right, joining the bottom of the two rightmost strands with a right cup, and using \cref{eq:t-def,rel:right-adjunction-up} gives \cref{rel:left-dotted-curl}.  Relation \cref{rel:right-dotted-curl} is obtained similarly, using \cref{eq:vertical-right-curl}.
\end{proof}

It follows from \cref{eq:vertical-left-curl,eq:vertical-right-curl,eq:f-in-dual-basis,eq:inf-grass1,eq:inf-grass2} that

\noindent\begin{minipage}{0.5\linewidth}
  \begin{equation} \label{eq:right-twist-cap-ell-neg}
    %
    \ ,\quad \text{if } k \ge 0.
  \end{equation}
\end{minipage}\par\vspace{\belowdisplayskip}

The following lemma will be generalized in \cref{lem:crossings-over-caps-cups}.  However, we need the following partial result first to prove \cref{lem:left-zigzag}, which will then be used in the proof of the more general \cref{lem:crossings-over-caps-cups}.

\begin{lem} \label{lem:partial-pitchfork}
  The following relations hold:

  \noindent\begin{minipage}{0.5\linewidth}
    \begin{equation} \label{eq:pitchfork-up-leftcap-neg-l}
      %
      \quad \text{if } k > 0.
    \end{equation}
  \end{minipage}\par\vspace{\belowdisplayskip}
\end{lem}

\begin{proof}
  First suppose $k \le 0$.  We first claim that
  \begin{equation} \label{eq:right-cupslice}
    %
    \ .
  \end{equation}
  Composing on the bottom with the invertible map \cref{eq:inversion-relation-neg-l}, it suffices to prove that

  \noindent\begin{minipage}{0.4\linewidth}
    \begin{equation} \label{eq:right-cupslice-proof-1}
      %
    \ .
  \]
  This completes the proof of \cref{eq:right-cupslice}.

  Now we have
  \[
    %
    \ .
  \]
  This completes the proof of \cref{eq:pitchfork-up-leftcap-neg-l}.  The proof of \cref{eq:pitchfork-up-leftcup-pos-l} is similar and will be omitted.
\end{proof}

\begin{lem} \label{lem:left-zigzag}
  The left adjunction relations \cref{rel:zigzag-leftdown,rel:zigzag-leftup} hold.
\end{lem}

\begin{proof}
  We prove \cref{rel:zigzag-leftup}, since \cref{rel:zigzag-leftdown} then follows by applying $\omega$.  If $k \le 0$, we have
  \[
    %
    \ .
  \]
  On the other hand, if $k > 0$, we have
  \[
    %
    \ . \qedhere
  \]
\end{proof}

\begin{lem}[Pitchfork relations] \label{lem:crossings-over-caps-cups}
  The following relations hold:

  \noindent\begin{minipage}{0.5\linewidth}
    \begin{equation} \label{rel:pitchfork-up-leftcap}
      %
\ .
    \end{equation}
  \end{minipage}\par\vspace{\belowdisplayskip}
\end{lem}

\begin{proof}
  In light of \cref{lem:partial-pitchfork}, it suffices to prove \cref{rel:pitchfork-up-leftcap} for $k > 0$ and \cref{rel:pitchfork-up-leftcup} for $k \le 0$.  Then \cref{rel:pitchfork-down-leftcap,rel:pitchfork-down-leftcup} following by applying $\omega$.  First suppose $k > 0$.  Attaching left caps to the top left and top right strands of \cref{eq:pitchfork-up-leftcup-pos-l} gives
  \[
    %
    \ .
  \]
  Then \cref{rel:pitchfork-up-leftcap} follows from \cref{rel:zigzag-leftdown,rel:zigzag-leftup}.  Similarly, when $k \le 0$, \cref{rel:pitchfork-up-leftcup} follows from attaching left cups to the bottom left and bottom right strands of \cref{eq:pitchfork-up-leftcap-neg-l} and using \cref{rel:zigzag-leftdown,rel:zigzag-leftup}.
\end{proof}

\begin{lem}
  The bubble slide relations \cref{rel:clockwise-bubble-slide,rel:counterclockwise-bubble-slide} hold.
\end{lem}

\begin{proof}
  It suffices to prove \cref{rel:clockwise-bubble-slide}, since then \cref{rel:counterclockwise-bubble-slide} follows by applying $\omega$, placing upward strands to the left and right of all diagrams, connecting the tops of the two leftmost straight strands with a right cap, connecting the bottoms of the rightmost strands with a right cup, and using \cref{rel:dot-right-cupcap,eq:f-right-cupcap-slide,rel:right-adjunction-up}.

  First suppose $k < 0$.  We have
  \begin{align*}
    \begin{tikzpicture}[anchorbase]
      \draw[<-] (0,0.3) arc(90:450:0.3);
      \draw[->] (0.8,-1) -- (0.8,1);
      \redcircle{(0.3,0)} node[anchor=west,color=black] {$r$};
      \bluedot{(-0.3,0)} node[anchor=east,color=black] {$f$};
    \end{tikzpicture}
    \ &\stackrel[\substack{\cref{rel:f-left-cup} \\ \cref{rel:right-adjunction-up} \\ \cref{rel:zigzag-leftup}}]{\substack{\cref{rel:dot-right-cupcap} \\ \cref{rel:down-up-doublecross-new}}}{=}\
    \begin{tikzpicture}[anchorbase]
      \draw[<-] (0,0.5) arc(90:450:0.5);
      \draw[->] (0.8,-1) .. controls (0,-0.3) and (0,0.3) .. (0.8,1);
      \redcircle{(-0.48,0.15)} node[anchor=east,color=black] {$r$};
      \bluedot{(-0.48,-0.15)} node[anchor=east,color=black] {$f$};
    \end{tikzpicture}
    \ - \sum_{u,s \ge 0} (-1)^{\bar a \bar b + \bar a \bar f + \bar a + \bar b}\
    \begin{tikzpicture}[anchorbase]
      \draw[->] (0,-1) -- (0,1);
      \draw[->] (2,0.3) arc(90:-270:0.3);
      \redcircle{(0,0.7)} node[anchor=west,color=black] {$u$};
      \bluedot{(0,0.4)} node[anchor=east,color=black] {$\chk{b}$};
      \bluedot{(0,-0.4)} node[anchor=east,color=black] {$f\psi^k(a)$};
      \redcircle{(0,-0.7)} node[anchor=west,color=black] {$s$};
      \redcircle{(1.7,0)} node[anchor=east,color=black] {\dotlabel{-u-s-2}};
      \bluedot{(2.3,0)} node[anchor=west,color=black] {$\chk{a}b$};
      \redcircle{(0,0)} node[anchor=east,color=black] {$r$};
    \end{tikzpicture}
    \\
    &\stackrel[\cref{rel:dot-token-up-slide}]{\substack{\cref{rel:pitchforks-right} \\ \cref{rel:pitchfork-up-leftcup}}}{=}\
    \begin{tikzpicture}[anchorbase]
      \draw[<-] (0,0.5) arc(90:450:0.5);
      \draw[->] (-0.8,-1) .. controls (0,-0.3) and (0,0.3) .. (-0.8,1);
      \redcircle{(-0.48,0.15)} node[anchor=east,color=black] {$r$};
      \bluedot{(-0.48,-0.15)} node[anchor=east,color=black] {$f$};
    \end{tikzpicture}
    \ - \sum_{u,s \ge 0} (-1)^{\bar b + \bar f \bar b}\
    \begin{tikzpicture}[anchorbase]
      \draw[->] (0,-1) -- (0,1);
      \draw[->] (1.8,-0.2) arc(90:-270:0.3);
      \redcircle{(0,0.1)} node[anchor=west,color=black] {$u+r+s$};
      \bluedot{(0,0.6)} node[anchor=west,color=black] {$\psi^{-u}(\chk{b}) \psi^{-u-r}(f) \psi^{k-u-r}(a)$};
      \redcircle{(1.5,-0.5)} node[anchor=east,color=black] {\dotlabel{-u-s-2}};
      \bluedot{(2.1,-0.5)} node[anchor=west,color=black] {$\chk{a}b$};
    \end{tikzpicture}
    \ .
  \end{align*}
  Introducing $t=u+r+s$ and replacing the sum over $u$ with a sum over $t$, the above double sum becomes
  \[
    \sum_{s \ge 0} \sum_{t \ge r+s} (-1)^{\bar b + \bar f \bar b + \bar a \bar b}\
    \begin{tikzpicture}[anchorbase]
      \draw[->] (0,-1) -- (0,1);
      \draw[->] (1.6,-0.2) arc(90:-270:0.3);
      \redcircle{(0,0.1)} node[anchor=west,color=black] {$t$};
      \bluedot{(0,0.6)} node[anchor=west,color=black] {$\psi^{r+s-t}(\chk{b}) \psi^{s-t}(f) \psi^{k+s-t}(a)$};
      \redcircle{(1.3,-0.5)} node[anchor=east,color=black] {\dotlabel{r-t-2}};
      \bluedot{(1.87,-0.35)} node[anchor=south west,color=black] {$b$};
      \bluedot{(1.87,-0.65)} node[anchor=north west,color=black] {$\chk{a}$};
    \end{tikzpicture}
    \ .
  \]
  Sliding the token labelled $b$ and the $r-t-2$ dots over the right cap (and past each other) and the token labelled $\chk{a}$ over the left cup, changing the sum over $a \in B$ to a sum over $\psi^k(a)$, and the sum over $b \in B$ to a sum over $\psi^{r+s-t}(\chk{b})$ (with left dual basis elements $(-1)^{\bar b} \psi^{r+s-t-1}(b)$ by \cref{eq:double-dual}) we obtain
  \begin{equation} \label{eq:bubble-slide-proof-step}
    \sum_{s \ge 0} \sum_{t \ge r+s} (-1)^{\bar f \bar b + \bar a \bar b}\
    \begin{tikzpicture}[anchorbase]
      \draw[->] (0,-1) -- (0,1);
      \draw[->] (3.1,-0.2) arc(90:-270:0.3);
      \redcircle{(0,0.1)} node[anchor=west,color=black] {$t$};
      \bluedot{(0,0.6)} node[anchor=west,color=black] {$b \psi^{s-t}(fa)$};
      \redcircle{(3.4,-0.5)} node[anchor=west,color=black] {\dotlabel{r-t-2}};
      \bluedot{(2.8,-0.5)} node[anchor=east,color=black] {$\psi^{-s-1}(\chk{b}) \chk{a}$};
    \end{tikzpicture}
    \ .
  \end{equation}
  Now
  \begin{multline*}
    (-1)^{\bar a \bar b} \psi^{-s-1}(\chk{b}) \chk{a} \otimes a
    \stackrel{\cref{eq:f-in-dual-basis}}{=} (-1)^{\bar b + \bar b \bar e} \tr \left( \psi^{-s-1}(\chk{b}) \chk{a} e \right) \chk{e} \otimes a
    \\
    \stackrel{\cref{eq:Nakayama-def}}{=} (-1)^{\bar b \bar e} \chk{e} \otimes \tr \left( \chk{a} e \psi^{-s}(\chk{b}) \right) a
    \stackrel{\cref{eq:f-in-basis}}{=} (-1)^{\bar b \bar e} \chk{e} \otimes e \psi^{-s}(\chk{b}).
  \end{multline*}
  Using this and \cref{eq:f-Euler-commute},
  \details{
    We have
    \begin{multline*}
      (-1)^{\bar f \bar b + \bar b \bar a} fa \otimes \chk{a}
      \stackrel{\cref{eq:f-in-basis}}{=} (-1)^{\bar f \bar b + \bar b \bar a} \tr(\chk{e}fa) e \otimes \chk{a}
      = (-1)^{\bar b \bar e} e \otimes \tr(\chk{e}fa) \chk{a}
      \\
      \stackrel{\cref{eq:f-in-dual-basis}}{=} (-1)^{\bar b \bar e} e \otimes \chk{e}f
      = (-1)^{\bar a \bar b} a \otimes \chk{a}f.
    \end{multline*}
  }
  the expression \cref{eq:bubble-slide-proof-step} becomes
  \[
    \sum_{s \ge 0} \sum_{t \ge r+s} (-1)^{\bar a \bar b}\

    \ .
  \]
  On the other hand, we have
  \begin{align*}
    %
    \ ,
  \end{align*}
  where, in the final equality, we introduced $t=r+s-1-u$ and changed the sum over $u$ to a sum over $t$.

  Combining the computations above, we have
  \[
    %
    \ ,
  \]
  which is equal to the right side of \cref{rel:clockwise-bubble-slide} after changing the order of summation.

  The case $k \ge 0$, which is similar but simpler, is left to the reader.
  \details{
    Suppose $k \ge 0$.  Then we have
    \begin{align*}
      %
      \ ,
    \end{align*}
    where, in last equality, we introduced $t=r+j-u-1$ and $s=r-1-u$ and changed the summations over $u$ and $j$ to summations over $t$ and $s$ (using \cref{eq:inf-grass1} in the process).
  }
\end{proof}

\begin{lem}
  The alternating braid relation \cref{rel:braid-alternating} holds.
\end{lem}

\begin{proof}
  On both sides of \cref{rel:braid-down-up-up}, attach crossings to the top left and bottom right pairs of strands to obtain
  \begin{equation} \label{eq:alternating-braid-proof-double-twist}
    %
    \ .
  \end{equation}
  First consider the case $-1 \le k \le 1$.  Using \cref{rel:down-up-doublecross-new}, the left side of \cref{eq:alternating-braid-proof-double-twist} becomes
  \[
    %
    \ .
  \]
  Similarly, using \cref{rel:up-down-doublecross-new}, the right side of \cref{eq:alternating-braid-proof-double-twist} becomes
  \[
    %
    \ .
  \]

  Now suppose $k \ge 2$.  Using \cref{rel:down-up-doublecross-new}, the left side of \cref{eq:alternating-braid-proof-double-twist} becomes
  \[
    %
    \ .
  \]
  Using \cref{rel:up-down-doublecross-new}, the right side of \cref{eq:alternating-braid-proof-double-twist} becomes
  \[
    %
    \ .
  \]
  Note that all terms above with $u \ge k$ are zero by \cref{eq:inf-grass2}.  For $0 \le u \le k-1$, we have
  \[
    %
    \ .
  \]
  Thus we have
  \[
    %
\ .
  \]
  We now introduce $s=u-1-t$ to replace the sum over $u$ with a sum over $s$ in order to obtain \cref{rel:braid-alternating} in the case $k \ge 2$.  The case $k \le -2$ is similar and will be omitted.

  \details{
    Suppose $k \le -2$.  Using \cref{rel:up-down-doublecross-new}, the right side of \cref{eq:alternating-braid-proof-double-twist} becomes
    \[
      %
    \]
    Using \cref{rel:down-up-doublecross-new}, the left side of \cref{eq:alternating-braid-proof-double-twist} becomes
    \[
      %
      \ .
    \]
    Note that all terms above with $u \ge -k$ are zero by \cref{eq:inf-grass1}.  For $0 \le u \le -k-1$, we have
    \[
      %
      \ .
    \]
    Thus we have
    \[
      %
    \]
    Now we introduce $r=u-1-t$ to replace the sum over $u$ with a sum over $r$, in order to obtain \cref{rel:braid-alternating} in the case $k \le -2$.
  }
\end{proof}

\begin{lem}
  The rotation relations \cref{rel:token-rotation,rel:dot-rotation,rel:crossing-rotation} hold.
\end{lem}

\begin{proof}
  The first equality of \cref{rel:token-rotation} is the definition of the token on a downward strand (see \cref{eq:right-mates}), while the second equality follows from \cref{rel:f-left-cap,rel:zigzag-leftdown}.  Similarly, the first of equality of \cref{rel:dot-rotation} is a definition, while the second equality follows from \cref{rel:dot-left-cup-slide,rel:zigzag-leftdown,eq:central-bubble-reverse}.  Finally, the first equation of \cref{rel:crossing-rotation} follows from the definitions \cref{eq:t-def,eq:right-mates} of $t$ and $s'$, while the second equation follows from \cref{rel:pitchfork-up-leftcap,rel:pitchfork-down-leftcap,rel:zigzag-leftdown}.
\end{proof}

It follows from \cref{rel:crossing-rotation,eq:right-mates} that
\[
  s' =
  \begin{tikzpicture}[anchorbase]
    \draw [<-](0,0) -- (0.6,0.6);
    \draw [<-](0.6,0) -- (0,0.6);
  \end{tikzpicture}
  \ :=\
  \begin{tikzpicture}[anchorbase]
    \draw[<-] (0.3,0) -- (-0.3,1);
    \draw[->] (-0.75,1) -- (-0.75,0.5) .. controls (-0.75,0.2) and (-0.5,0) .. (0,0.5) .. controls (0.5,1) and (0.75,0.8) .. (0.75,0.5) -- (0.75,0);
  \end{tikzpicture}
  \ = \
  \begin{tikzpicture}[anchorbase]
    \draw[<-] (-0.3,0) -- (0.3,1);
    \draw[->] (0.75,1) -- (0.75,0.5) .. controls (0.75,0.2) and (0.5,0) .. (0,0.5) .. controls (-0.5,1) and (-0.75,0.8) .. (-0.75,0.5) -- (-0.75,0);
  \end{tikzpicture}
  \ .
\]

\section{Proofs of theorems \label{sec:proofs}}

\subsection{Proof of \cref{theo:alternate-presentation}}

We first prove the existence of $c'$ and $d'$ satisfying \cref{theo-eq:doublecross-up-down,theo-eq:doublecross-down-up,theo-eq:right-curl,theo-eq:clockwise-circ,theo-eq:left-curl,theo-eq:counterclockwise-circ}. Define $\Heis_{F,k}$ as in \cref{def:H}.  Then define $t'$ and the decorated left cups and caps by \cref{eq:inversion-leftcup-def,eq:inversion-leftcap-def}, define $c'$ and $d'$ by \cref{eq:left-cup-def,eq:left-cap-def}, and define the negatively dotted bubbles by \cref{eq:neg-clockwise-bubble,eq:neg-cc-bubble}.  It follows from \cref{rel:pitchfork-up-leftcap,rel:zigzag-leftdown} that this definition of $t'$ agrees with \cref{eq:t-def-alt}.  Similarly, it follows from \cref{lem:bubble-determinant-identities} that this definition of negatively dotted bubbles agrees with \cref{theo-eq:neg-ccbubble,theo-eq:neg-cbubble}.  Then the relations \cref{theo-eq:doublecross-up-down,theo-eq:doublecross-down-up,theo-eq:right-curl,theo-eq:clockwise-circ,theo-eq:left-curl,theo-eq:counterclockwise-circ} follow from \cref{rel:up-down-doublecross-new,rel:down-up-doublecross-new,eq:inf-grass1,eq:inf-grass2,rel:left-dotted-curl,rel:right-dotted-curl}.

Now let $\mathcal{C}$ be the strict $\kk$-linear graded monoidal supercategory generated by the objects $\sQ_+$, $\sQ_-$, and morphisms $s,x,c,d,c',d'$, and $\beta_f$, $f \in F$, subject to the relations \cref{rel:token-homom,rel:braid-up,rel:doublecross-up,rel:dot-token-up-slide,rel:tokenslide-up-right,rel:dotslide1,rel:right-adjunction-up,rel:right-adjunction-down,theo-eq:doublecross-up-down,theo-eq:doublecross-down-up,theo-eq:right-curl,theo-eq:clockwise-circ,theo-eq:left-curl,theo-eq:counterclockwise-circ}.
Since all of these relations hold in $\Heis_{F,k}$, there is a strict $\kk$-linear monoidal functor $\mathbold{A} \colon \mathcal{C} \to \Heis_{F,k}$ taking the objects $\sQ_\pm$ and morphisms $x,s,c,d,c',d', \beta_f$, $f \in F$, in $\mathcal{C}$ to the elements with the same names in $\Heis_{F,k}$.

We claim that there is a strict $\kk$-linear monoidal functor $\mathbold{B} \colon \Heis_{F,k} \to \mathcal{C}$, sending the objects $\sQ_\pm$ and morphisms $x,s,c,d,c',d', \beta_f$, $f \in F$, in $\mathcal{C}$ to the elements with the same names in $\mathcal{C}$.  After showing the existence of $\mathbold{B}$, we will prove that it is a two-sided inverse of $\mathbold{A}$.  For the existence of $\mathbold{B}$, it suffices to verify that the defining relations of $\Heis_{F,k}$ hold in $\mathcal{C}$.  We will do this in the case $k \ge 0$, since the case $k < 0$ is similar.

In $\mathcal{C}$, we define new morphisms
\[
  \begin{tikzpicture}[>=To,baseline={([yshift=1ex]current bounding box.center)}]
    \draw[<-] (0,0.2) -- (0,0) arc (180:360:.3) -- (0.6,0.2);
    \greensquare{(0.3,-0.3)} node[anchor=north,color=black] {\squarelabel{(r,b)}};
  \end{tikzpicture}
  = - \sum_{s \ge 0} \
  \begin{tikzpicture}[anchorbase]
    \draw[<-] (0,0.3) -- (0,0) arc (180:360:.3) -- (0.6,0.3);
    \draw[->] (0.6,-0.4) arc (90:450:.3);
    \redcircle{(0.6,0)} node[anchor=east,color=black] {$s$};
    \bluedot{(0,0)} node[anchor=east,color=black] {$a$};
    \bluedot{(0.9,-0.7)} node[anchor=west,color=black] {$\chk{a}b$};
    \redcircle{(0.3,-0.7)} node[anchor=east,color=black] {\dotlabel{-r-s-2}};
  \end{tikzpicture}
  \quad, \quad 0 \le r < k,\ b \in B.
\]
We claim that the $1 \times (1+k \dim F)$ matrix
\[
  \left[
    \begin{tikzpicture}[anchorbase]
      \draw [<-](0,0) -- (0.6,0.6);
      \draw [->](0.6,0) -- (0,0.6);
    \end{tikzpicture}
    \quad
    \begin{tikzpicture}[>=To,baseline={([yshift=1ex]current bounding box.center)}]
      \draw[<-] (0,0.2) -- (0,0) arc (180:360:.3) -- (0.6,0.2);
      \greensquare{(0.3,-0.3)} node[anchor=north,color=black] {\squarelabel{(r,b)}};
    \end{tikzpicture},\
    0 \le r \le k-1,\ b \in B
  \right]
\]
is a two-sided inverse of \cref{eq:inversion-relation-pos-l}.  Composing in one order yields the morphism
\[
  \begin{tikzpicture}[anchorbase]
    \draw[->] (0,0) .. controls (0.5,0.5) .. (0,1);
    \draw[<-] (0.5,0) .. controls (0,0.5) .. (0.5,1);
  \end{tikzpicture}
  \ - \sum_{r,s \ge 0} \
  \begin{tikzpicture}[anchorbase]
    \draw[->] (0,0) -- (0,0.7) arc (180:0:0.3) -- (0.6,0);
    \draw[<-] (0,2.1) -- (0,1.8) arc (180:360:0.3) -- (0.6,2.1);
    \redcircle{(0,0.6)} node[anchor=west,color=black] {$r$};
    \bluedot{(0,0.2)} node[anchor=west,color=black] {$\chk{b}$};
    \redcircle{(0.6,1.8)} node[anchor=west,color=black] {$s$};
    \bluedot{(0,1.8)} node[anchor=east,color=black] {$a$};
    \draw[->] (1,1.55) arc(90:450:0.3);
    \bluedot{(1.3,1.25)} node[anchor=west,color=black] {$\chk{a} b$};
    \redcircle{(0.7,1.25)} node[anchor=east,color=black] {\dotlabel{-r-s-2}};
  \end{tikzpicture}
  \ ,
\]
which is the identity by \cref{theo-eq:doublecross-up-down}.  Composing in the other order, we obtain a $(1 + k \dim F) \times (1 + k \dim F)$ matrix.  Its $(1,1)$-entry is the identity by \eqref{theo-eq:doublecross-down-up}.  This is sufficient when $k=0$.  However, when $k > 0$, we also need to verify the following relations for $b \in B$, $0 \le r,s < k$:

\noindent\begin{minipage}{0.33\linewidth}
  \begin{equation} \label{eq:inverse-check1}
    \begin{tikzpicture}[anchorbase]
      \draw[->] (0.6,-0.5) .. controls (0.1,0) and (0,-0.1) .. (0,0.3) -- (0,0.7) arc (180:0:.3) -- (0.6,0.3) .. controls (0.6,-0.1) and (0.5,0) .. (0,-0.5);
      \redcircle{(0,0.6)} node[anchor=west,color=black] {$r$};
      \bluedot{(0,0.3)} node[anchor=west,color=black] {$\chk{b}$};
    \end{tikzpicture}
    = 0,
  \end{equation}
\end{minipage}%
\begin{minipage}{0.33\linewidth}
  \begin{equation} \label{eq:inverse-check2}
    \begin{tikzpicture}[anchorbase]
      \draw[->] (0,0) .. controls (0.3,-0.3) and (0.6,-0.3) .. (0.6,-0.6) arc(360:180:0.3) .. controls (0,-0.3) and (0.3,-0.3) .. (0.6,0);
      \greensquare{(0.3,-0.9)} node[anchor=north,color=black] {$(r,b)$};
    \end{tikzpicture}
    = 0,
  \end{equation}
\end{minipage}
\begin{minipage}{0.33\linewidth}
  \begin{equation} \label{eq:inverse-check3}
    \begin{tikzpicture}[anchorbase]
      \draw[<-] (0.3,0) arc(0:360:0.3);
      \greensquare{(0,-0.3)} node[anchor=north,color=black] {\squarelabel{(s,c)}};
      \bluedot{(-0.25,-0.15)} node[anchor=east,color=black] {$\chk{b}$};
      \redcircle{(-0.25,0.15)} node[anchor=east,color=black] {$r$};
    \end{tikzpicture}
    = \delta_{r,s} \delta_{b,c}.
  \end{equation}
\end{minipage}\par\vspace{\belowdisplayskip}

To prove \cref{eq:inverse-check1}, we compute
\begin{equation} \label{eq:inverse-check1-calc}
  \begin{tikzpicture}[anchorbase]
    \draw[->] (0.6,-0.5) .. controls (0.1,0) and (0,-0.1) .. (0,0.3) -- (0,0.7) arc (180:0:.3) -- (0.6,0.3) .. controls (0.6,-0.1) and (0.5,0) .. (0,-0.5);
    \redcircle{(0,0.6)} node[anchor=west,color=black] {$r$};
    \bluedot{(0,0.3)} node[anchor=west,color=black] {$\chk{b}$};
  \end{tikzpicture}
  \ \stackrel{\cref{eq:t-def-alt}}{=}\
  \begin{tikzpicture}[anchorbase]
    \draw[->] (0,-0.2) .. controls (0.25,0.5) and (0.5,1) .. (0.75,1) .. controls (1.25,1) and (1.25,0.1) .. (0.75,0.1) .. controls (0,0.1) and (0,1) .. (-0.5,1) .. controls (-1,1) and (-1,0.5) .. (-1,-0.2);
    \redcircle{(0.5,0.84)} node[anchor=east,color=black] {$r$};
    \bluedot{(0.3,0.5)} node[anchor=west,color=black] {$\chk{b}$};
  \end{tikzpicture}
  \ \stackrel[\cref{rel:dotslide2}]{\cref{rel:dot-token-up-slide}}{=}\
  \begin{tikzpicture}[anchorbase]
    \draw[->] (0,-0.2) .. controls (0.25,0.5) and (0.5,1) .. (0.75,1) .. controls (1.25,1) and (1.25,0.1) .. (0.75,0.1) .. controls (0,0.1) and (0,1) .. (-0.5,1) .. controls (-1,1) and (-1,0.5) .. (-1,-0.2);
    \redcircle{(0.08,0)} node[anchor=east,color=black] {$r$};
    \bluedot{(0.5,0.84)} node[anchor=south,color=black] {$\psi^{-r}(\chk{b})$};
  \end{tikzpicture}
  \ - \sum_{t=0}^{r-1}\
  \begin{tikzpicture}[anchorbase]
    \draw[<-] (0,-.6) -- (0,0.6) arc (180:0:.3) -- (.6,-.6);
    \draw[->] (2,1.4) -- (2,0.2) arc(360:180:0.3) -- (1.4,1.4) arc(180:0:0.3);
    \redcircle{(0.6,0.2)} node[anchor=east,color=black] {$t$};
    \bluedot{(0.6,0.6)} node[anchor=east,color=black] {$a$};
    \redcircle{(1.4,1.1)} node[anchor=east,color=black] {\dotlabel{r-1-t}};
    \bluedot{(1.4,0.2)} node[anchor=east,color=black] {$\chk{a}$};
    \bluedot{(1.4,1.4)} node[anchor=east,color=black] {\dotlabel{\psi^{-r}(\chk{b})}};
  \end{tikzpicture}
  \ \stackrel[\substack{\cref{rel:dot-token-up-slide} \\ \cref{theo-eq:clockwise-circ}}]{\substack{\cref{rel:tokenslide-up-right} \\ \cref{theo-eq:right-curl}}}{=} 0.
\end{equation}
To prove \cref{eq:inverse-check2}, note that
\[
  \begin{tikzpicture}[anchorbase]
    \draw[->] (0,0) .. controls (0.3,-0.3) and (0.6,-0.3) .. (0.6,-0.6) arc(360:180:0.3) .. controls (0,-0.3) and (0.3,-0.3) .. (0.6,0);
    \greensquare{(0.3,-0.9)} node[anchor=north,color=black] {$(r,b)$};
  \end{tikzpicture}
  = - \sum_{s \ge 0} \
  \begin{tikzpicture}[anchorbase]
    \draw[<-] (0.6,1) .. controls (0.3,0.7) and (0,0.5) .. (0,0.3) -- (0,0) arc (180:360:.3) -- (0.6,0.3) .. controls (0.6,0.5) and (0.3,0.7) .. (0,1);
    \draw[->] (0.6,-0.4) arc (90:450:.3);
    \redcircle{(0.6,0.15)} node[anchor=east,color=black] {$s$};
    \bluedot{(0,0.15)} node[anchor=east,color=black] {$a$};
    \bluedot{(0.9,-0.7)} node[anchor=west,color=black] {$\chk{a} b$};
    \redcircle{(0.3,-0.7)} node[anchor=east,color=black] {\dotlabel{-r-s-2}};
  \end{tikzpicture}\ .
\]
By \cref{theo-eq:neg-ccbubble}, the bubble above is zero if $s \ge k$, while for $0 \le s < k$, the curl is zero by an argument similar to \cref{eq:inverse-check1-calc}.
\details{
  \[
    \begin{tikzpicture}[anchorbase]
      \draw[<-] (0.6,0.5) .. controls (0.1,0) and (0,0.1) .. (0,-0.3) -- (0,-0.7) arc (180:360:.3) -- (0.6,-0.3) .. controls (0.6,0.1) and (0.5,0) .. (0,0.5);
      \redcircle{(0.6,-0.5)} node[anchor=west,color=black] {$s$};
      \bluedot{(0,-0.5)} node[anchor=west,color=black] {$c$};
    \end{tikzpicture}
    \ \stackrel{\cref{eq:t-def}}{=}\
    \begin{tikzpicture}[anchorbase]
      \draw[<-] (0,0.2) .. controls (0.25,-0.5) and (0.5,-1) .. (0.75,-1) .. controls (1.25,-1) and (1.25,-0.1) .. (0.75,-0.1) .. controls (0,-0.1) and (0,-1) .. (-0.5,-1) .. controls (-1,-1) and (-1,-0.5) .. (-1,0.2);
      \redcircle{(1.12,-0.5)} node[anchor=west,color=black] {$s$};
      \bluedot{(0.35,-0.6)} node[anchor=west,color=black] {$c$};
    \end{tikzpicture}
    \ \stackrel[\cref{rel:dotslide2}]{\cref{rel:dot-right-cupcap}}{=}\
    \begin{tikzpicture}[anchorbase]
      \draw[<-] (0,0.2) .. controls (0.25,-0.5) and (0.5,-1) .. (0.75,-1) .. controls (1.25,-1) and (1.25,-0.1) .. (0.75,-0.1) .. controls (0,-0.1) and (0,-1) .. (-0.5,-1) .. controls (-1,-1) and (-1,-0.5) .. (-1,0.2);
      \redcircle{(0,-0.6)} node[anchor=east,color=black] {$s$};
      \bluedot{(0.5,-0.84)} node[anchor=north,color=black] {$c$};
    \end{tikzpicture}
    \ - \sum_{t=0}^{s-1}\
    \begin{tikzpicture}[anchorbase]
      \draw[->] (0,0.6) -- (0,-0.6) arc (180:360:.3) -- (0.6,0.6);
      \draw[<-] (3,-1.4) -- (3,-0.2) arc(0:180:0.3) -- (2.4,-1.4) arc(180:360:0.3);
      \redcircle{(0.6,-0.6)} node[anchor=east,color=black] {$t$};
      \bluedot{(0.6,-0.2)} node[anchor=east,color=black] {$b$};
      \redcircle{(2.4,-0.2)} node[anchor=east,color=black] {\dotlabel{s-1-t}};
      \bluedot{(2.4,-0.8)} node[anchor=east,color=black] {$\chk{b}$};
      \bluedot{(2.4,-1.4)} node[anchor=east,color=black] {$c$};
    \end{tikzpicture}
    \ \stackrel[\substack{\cref{rel:f-left-cup} \\ \cref{theo-eq:clockwise-circ}}]{\substack{\cref{rel:tokenslide-up-left} \\ \cref{theo-eq:right-curl}}}{=} 0.
  \]
  Note that the proof of \cref{rel:f-left-cup} in the case $k > 0$ (obtained by applying $\omega$ to all diagrams in the proof of \cref{rel:f-left-cap} in the case $k < 0$), which we used above, is valid in $\mathcal{C}$.
}

Next we prove \cref{eq:inverse-check3}.  We have
\begin{equation} \label{eq:theo-inverse-check-left-cup}
  \begin{tikzpicture}[anchorbase]
    \draw[<-] (0.3,0) arc(0:360:0.3);
    \greensquare{(0,-0.3)} node[anchor=north,color=black] {\squarelabel{(s,a)}};
    \bluedot{(-0.25,-0.15)} node[anchor=east,color=black] {$\chk{b}$};
    \redcircle{(-0.25,0.15)} node[anchor=east,color=black] {$r$};
  \end{tikzpicture}
  = - \sum_{t \ge 0}\
  \begin{tikzpicture}[anchorbase]
    \draw[->] (0,0.3) arc (90:-270:0.3);
    \bluedot{(-0.3,0)} node[anchor=east,color=black] {$\chk{b}e$};
    \redcircle{(0.3,0)} node[anchor=west] {\dotlabel{r+t}};
    \draw[->] (0,-0.5) arc (90:450:0.3);
    \bluedot{(0.3,-0.8)} node[anchor=west,color=black] {$\chk{e}a$};
    \redcircle{(-0.3,-0.8)} node[anchor=east,color=black] {\dotlabel{-s-t-2}};
  \end{tikzpicture}\ .
\end{equation}
If $r < s$, then all the terms on the right side of \cref{eq:theo-inverse-check-left-cup} vanish by \cref{theo-eq:clockwise-circ,theo-eq:counterclockwise-circ} and we are done.  If $r=s$, then only the $t=0$ term survives, and \cref{eq:inverse-check3} follows from \cref{eq:f-in-basis,theo-eq:clockwise-circ,theo-eq:counterclockwise-circ}.  Now suppose $r > s$.  Note that the proof of \cref{rel:f-left-cup} in the case $k > 0$ (obtained by applying $\omega$ to all diagrams in the proof of \cref{rel:f-left-cap} in the case $k < 0$) is valid in $\mathcal{C}$.  Also, it follows from the computation \cref{eq:C-d'-uniqueness} below that \cref{rel:f-left-cap} holds in $\mathcal{C}$ when $k \ge 0$.  Thus, letting $f = \psi^{-k}(\chk{b}a)$, we have
\begin{multline*}
  \begin{tikzpicture}[anchorbase]
    \draw[<-] (0.3,0) arc(0:360:0.3);
    \greensquare{(0,-0.3)} node[anchor=north,color=black] {\squarelabel{(s,a)}};
    \bluedot{(-0.25,-0.15)} node[anchor=east,color=black] {$\chk{b}$};
    \redcircle{(-0.25,0.15)} node[anchor=east,color=black] {$r$};
  \end{tikzpicture}
  = - \sum_{t \ge 0}\
  \begin{tikzpicture}[anchorbase]
    \draw[->] (0,0.3) arc (90:-270:0.3);
    \bluedot{(-0.3,0)} node[anchor=east,color=black] {$\chk{b}e$};
    \redcircle{(0.3,0)} node[anchor=west] {\dotlabel{r+t}};
    \draw[->] (0,-0.5) arc (90:450:0.3);
    \bluedot{(0.3,-0.8)} node[anchor=west,color=black] {$\chk{e}a$};
    \redcircle{(-0.3,-0.8)} node[anchor=east,color=black] {\dotlabel{-s-t-2}};
  \end{tikzpicture}
  \stackrel[\cref{eq:f-Euler-commute}]{\substack{\cref{rel:f-left-cap} \\ \cref{rel:f-left-cup}}}{=} - \sum_{t \ge 0}\
  \begin{tikzpicture}[anchorbase]
    \draw[->] (0,0.3) arc (90:-270:0.3);
    \bluedot{(0.3,0)} node[anchor=west,color=black] {$fe$};
    \redcircle{(-0.3,0)} node[anchor=east] {\dotlabel{r+t}};
    \draw[->] (0,-0.5) arc (90:450:0.3);
    \bluedot{(-0.3,-0.8)} node[anchor=east,color=black] {$\chk{e}$};
    \redcircle{(0.3,-0.8)} node[anchor=west,color=black] {\dotlabel{-s-t-2}};
  \end{tikzpicture}
  \\
  = - \sum_{u=0}^{r-s}\
  \begin{tikzpicture}[anchorbase]
    \draw[->] (0,0.3) arc (90:-270:0.3);
    \bluedot{(0.3,0)} node[anchor=west,color=black] {$fe$};
    \redcircle{(-0.3,0)} node[anchor=east] {\dotlabel{u+k-1}};
    \draw[->] (0,-0.5) arc (90:450:0.3);
    \bluedot{(-0.3,-0.8)} node[anchor=east,color=black] {$\chk{e}$};
    \redcircle{(0.3,-0.8)} node[anchor=west,color=black] {\dotlabel{r-s-u-k-1}};
  \end{tikzpicture}
  =
  \ccbubble{$f$}{\dotlabel{r-s-k-1}}
  - \sum_{u=1}^{r-s}\
  \begin{tikzpicture}[anchorbase]
    \draw[->] (0,0.3) arc (90:-270:0.3);
    \bluedot{(0.3,0)} node[anchor=west,color=black] {$fe$};
    \redcircle{(-0.3,0)} node[anchor=east] {\dotlabel{u+k-1}};
    \draw[->] (0,-0.5) arc (90:450:0.3);
    \bluedot{(-0.3,-0.8)} node[anchor=east,color=black] {$\chk{e}$};
    \redcircle{(0.3,-0.8)} node[anchor=west,color=black] {\dotlabel{r-s-u-k-1}};
  \end{tikzpicture}
  = 0,
\end{multline*}
where, in the second equality, we changed from a sum over $e \in B$ to a sum over $\psi^{-k}(e)$, and the last equality follows by expansion of a determinant, as in the proof of \cref{lem:bubble-determinant-identities}.
\details{
  The left equality in \cref{eq:det-expansion}, with $r$ replaced by $r-s$, followed from an expansion of determinants; hence it holds in $\mathcal{C}$ in light of \cref{theo-eq:neg-ccbubble}.  Again, by \cref{theo-eq:neg-ccbubble}, the leftmost expression in \cref{eq:det-expansion} in equal to the rightmost expression.  Hence the right equality in \cref{eq:det-expansion} holds, which is exactly what we need.
}
This completes the proof of \cref{eq:inverse-check3}.

Since we have shown that $\mathcal{C}$ satisfies the defining relations of $\Heis_{F,k}$, in $\mathcal{C}$ we can now use all the relations that we deduced from the defining relations of $\Heis_{F,k}$.  We will do so in what follows.

We next show that $c'$ and $d'$ are the \emph{unique} morphisms in $\mathcal{C}$ satisfying \cref{theo-eq:doublecross-up-down,theo-eq:doublecross-down-up,theo-eq:right-curl,theo-eq:clockwise-circ,theo-eq:left-curl,theo-eq:counterclockwise-circ}.  To do this, we prove that these relations can be used to express $c'$ and $d'$ in terms of the other generators.  First note that $t'$ can be characterized as the first entry in the inverse of the morphism \cref{eq:inversion-relation-pos-l} when $k \ge 0$, or as the first entry in the inverse of the morphism \cref{eq:inversion-relation-neg-l} when $k < 0$.  Hence $t'$ does not depend on $c'$ and $d'$, even though it is defined in terms of them.  When $k \ge 0$, as in \cref{eq:inverse-check1-calc} we have
\begin{equation} \label{eq:C-d'-uniqueness}
  \begin{tikzpicture}[anchorbase]
    \draw[->] (0.6,-0.5) .. controls (0.1,0) and (0,-0.1) .. (0,0.3) -- (0,0.7) arc (180:0:.3) -- (0.6,0.3) .. controls (0.6,-0.1) and (0.5,0) .. (0,-0.5);
    \redcircle{(0,0.4)} node[anchor=west,color=black] {$k$};
  \end{tikzpicture}
  \ \stackrel{\cref{eq:t-def-alt}}{=}\
  \begin{tikzpicture}[anchorbase]
    \draw[->] (0,-0.2) .. controls (0.25,0.5) and (0.5,1) .. (0.75,1) .. controls (1.25,1) and (1.25,0.1) .. (0.75,0.1) .. controls (0,0.1) and (0,1) .. (-0.5,1) .. controls (-1,1) and (-1,0.5) .. (-1,-0.2);
    \redcircle{(0.3,0.5)} node[anchor=west,color=black] {$k$};
  \end{tikzpicture}
  \ \stackrel{\cref{rel:dotslide2}}{=}\
  \begin{tikzpicture}[anchorbase]
    \draw[->] (0,-0.2) .. controls (0.25,0.5) and (0.5,1) .. (0.75,1) .. controls (1.25,1) and (1.25,0.1) .. (0.75,0.1) .. controls (0,0.1) and (0,1) .. (-0.5,1) .. controls (-1,1) and (-1,0.5) .. (-1,-0.2);
    \redcircle{(0.08,0)} node[anchor=east,color=black] {$k$};
  \end{tikzpicture}
  \ - \sum_{r=0}^{k-1}\
  \begin{tikzpicture}[anchorbase]
    \draw[<-] (0,-.6) -- (0,0.6) arc (180:0:.3) -- (.6,-.6);
    \draw[->] (2,1.2) -- (2,0.2) arc(360:180:0.3) -- (1.4,1.2) arc(180:0:0.3);
    \redcircle{(0.6,0.2)} node[anchor=east,color=black] {$r$};
    \bluedot{(0.6,0.6)} node[anchor=east,color=black] {$a$};
    \redcircle{(1.4,1.2)} node[anchor=east,color=black] {\dotlabel{k-1-r}};
    \bluedot{(1.4,0.2)} node[anchor=east,color=black] {$\chk{a}$};
  \end{tikzpicture}
  \ \stackrel[\substack{\cref{theo-eq:clockwise-circ} \\ \cref{eq:f-in-basis}}]{\substack{\cref{theo-eq:right-curl} \\ \cref{rel:dot-token-up-slide} \\ \cref{eq:f-right-cupcap-slide}}}{=}\
  \begin{tikzpicture}[anchorbase]
    \draw[<-] (0,-.2) -- (0,0) arc (180:0:.3) -- (.6,-.2);
  \end{tikzpicture}
  \ .
\end{equation}
Hence $d'$ is unique when $k \ge 0$.  Similarly, when $k \le 0$, we have
\[
  \begin{tikzpicture}[anchorbase]
    \draw[->] (0.6,0.5) .. controls (0.1,0) and (0,0.1) .. (0,-0.3) -- (0,-0.7) arc (180:360:.3) -- (0.6,-0.3) .. controls (0.6,0.1) and (0.5,0) .. (0,0.5);
    \redcircle{(0.6,-0.4)} node[anchor=east,color=black] {$k$};
  \end{tikzpicture}
  \ \stackrel{\cref{eq:t-def-alt}}{=}\
  \begin{tikzpicture}[anchorbase]
    \draw[<-] (0,0.2) .. controls (-0.25,-0.5) and (-0.5,-1) .. (-0.75,-1) .. controls (-1.25,-1) and (-1.25,-0.1) .. (-0.75,-0.1) .. controls (0,-0.1) and (0,-1) .. (0.5,-1) .. controls (1,-1) and (1,-0.5) .. (1,0.2);
    \redcircle{(-0.3,-0.5)} node[anchor=east,color=black] {$k$};
  \end{tikzpicture}
  \ \stackrel{\cref{rel:dotslide2}}{=}\
  \begin{tikzpicture}[anchorbase]
    \draw[<-] (0,0.2) .. controls (-0.25,-0.5) and (-0.5,-1) .. (-0.75,-1) .. controls (-1.25,-1) and (-1.25,-0.1) .. (-0.75,-0.1) .. controls (0,-0.1) and (0,-1) .. (0.5,-1) .. controls (1,-1) and (1,-0.5) .. (1,0.2);
    \redcircle{(-0.12,-0.1)} node[anchor=west,color=black] {$k$};
  \end{tikzpicture}
  \ + \sum_{r=0}^{k-1}\
  \begin{tikzpicture}[anchorbase]
    \draw[->] (0,0.6) -- (0,-0.6) arc(360:180:.3) -- (-0.6,0.6);
    \draw[<-] (-2,-1.2) -- (-2,-0.2) arc(180:0:0.3) -- (-1.4,-1.2) arc(360:180:0.3);
    \redcircle{(-0.6,-0.2)} node[anchor=west,color=black] {$r$};
    \bluedot{(-0.6,-0.6)} node[anchor=west,color=black] {$\chk{a}$};
    \redcircle{(-1.4,-1.2)} node[anchor=west,color=black] {\dotlabel{k-1-r}};
    \bluedot{(-1.4,-0.2)} node[anchor=west,color=black] {$a$};
  \end{tikzpicture}
  \ \stackrel[\substack{\cref{theo-eq:counterclockwise-circ} \\ \cref{eq:f-in-dual-basis}}]{\substack{\cref{theo-eq:left-curl} \\ \cref{rel:f-left-cap}}}{=}\
  \begin{tikzpicture}[anchorbase]
    \draw[<-] (0,0.2) -- (0,0) arc (180:360:0.3) -- (0.6,0.2);
  \end{tikzpicture}
  \ .
\]
Hence $c'$ is unique when $k \le 0$.

It remains to prove that $d'$ is unique when $k < 0$ and that $c'$ is unique when $k > 0$.  When $k > 0$, it follows from the above that the $r=k-1$, $b=1$ entry of the inverse of \cref{eq:inversion-relation-pos-l} is
\[
  \begin{tikzpicture}[>=To,baseline={([yshift=1ex]current bounding box.center)}]
    \draw[<-] (0,0.2) -- (0,0) arc (180:360:.3) -- (0.6,0.2);
    \greensquare{(0.3,-0.3)} node[anchor=north,color=black] {\squarelabel{(k-1,1)}};
  \end{tikzpicture}
  = - \sum_{s \ge 0} \
  \begin{tikzpicture}[anchorbase]
    \draw[<-] (0,0.3) -- (0,0) arc (180:360:.3) -- (0.6,0.3);
    \draw[->] (0.6,-0.4) arc (90:450:.3);
    \redcircle{(0.6,0)} node[anchor=east,color=black] {$s$};
    \bluedot{(0,0)} node[anchor=east,color=black] {$b$};
    \bluedot{(0.9,-0.7)} node[anchor=west,color=black] {$\chk{b}$};
    \redcircle{(0.3,-0.7)} node[anchor=east,color=black] {\dotlabel{-r-s-2}};
  \end{tikzpicture}
  \stackrel[\cref{rel:f-left-cap}]{\cref{rel:dot-right-cupcap}}{=} - \sum_{s \ge 0} \
  \begin{tikzpicture}[anchorbase]
    \draw[<-] (0,0.3) -- (0,0) arc (180:360:.3) -- (0.6,0.3);
    \draw[->] (0.6,-0.4) arc (90:450:.3);
    \redcircle{(0.6,0)} node[anchor=east,color=black] {$s$};
    \bluedot{(0,0)} node[anchor=east,color=black] {$b$};
    \bluedot{(0.3,-0.7)} node[anchor=east,color=black] {$\psi^{-k}(\chk{b})$};
    \redcircle{(0.9,-0.7)} node[anchor=west,color=black] {\dotlabel{-r-s-2}};
  \end{tikzpicture}
  \ \stackrel[\cref{eq:f-in-basis}]{\cref{theo-eq:counterclockwise-circ}}{=} -\
  \begin{tikzpicture}[anchorbase]
    \draw[<-] (0,0.2) -- (0,0) arc (180:360:.3) -- (0.6,0.2);
  \end{tikzpicture}\ .
\]
Therefore $c'$ is unique when $k > 0$.  The proof that $d'$ is unique when $k < 0$ is analogous.

We can now complete the proof of theorem.  It is clear from the definitions that $\mathbold{A} \circ \mathbold{B} = \id_{\Heis_{F,k}}$.  It is also clear that $\mathbold{B} \circ \mathbold{A}$ is the identity on $x$, $s$, $c$, and $d$.  It then follows from the uniqueness established above that it is the identity on $c'$ and $d'$.  Hence $\mathbold{B} \circ \mathbold{A} = \id_\mathcal{C}$.  So $\Heis_{F,k}$ and $\mathcal{C}$ are isomorphic, which establishes the equivalent presentation from the statement of the theorem.  Finally, since $\Heis_{F,k}$ and $\mathcal{C}$ are isomorphic, the uniqueness of $c'$ and $d'$ in $\mathcal{C}$ established above demonstrates that they are also the unique isomorphisms in $\Heis_{F,k}$ satisfying \cref{theo-eq:doublecross-up-down,theo-eq:doublecross-down-up,theo-eq:right-curl,theo-eq:clockwise-circ,theo-eq:left-curl,theo-eq:counterclockwise-circ}.

\subsection{Proof of \cref{theo:specializations}}

When $F = \kk$, \cref{def:H} coincides with \cite[Def.~1.1]{Bru17}.  So part~\cref{theo-item:F=k-specialization} follows immediately from \cite[Th.~1.4]{Bru17}.

We now prove part~\ref{theo-item:xi=-1-specialization}.  The need to pass to the underlying category of the $\Pi$-envelope of $\Heis_{F,-1}(R)$ arises from the different notion of supercategory used in \cite{RS17}, where all morphisms are considered to be even of degree zero, but between shifted objects.  See the discussion in \cite[\S1]{BE17}.  We will suppress this difference in the argument to follow.

When $k = -1$, we have
\[
  \begin{tikzpicture}[anchorbase]
    \draw[->] (0,-0.75) .. controls (0,0.5) and (0.5,0.5) .. (0.5,0) .. controls (0.5,-0.5) and (0,-0.5) .. (0,0.75);
  \end{tikzpicture}
  \ \stackrel[\cref{eq:f-in-basis}]{\substack{\cref{rel:right-dotted-curl} \\ \cref{eq:inf-grass1}}}{=}\
  \begin{tikzpicture}[anchorbase]
    \draw[->] (0,-0.75) -- (0,0.75);
    \redcircle{(0,0.05)};
  \end{tikzpicture}
  \ + \
  \begin{tikzpicture}[anchorbase]
    \draw[->] (0,-0.75) -- (0,0.75);
    \draw[<-] (0.95,-0.05) arc(90:450:0.3);
    \bluedot{(0.65,-0.35)} node[anchor=east,color=black] {$\chk{b}$};
    \redcircle{(1.25,-0.35)} node[anchor=west,color=black] {$-1$};
    \bluedot{(0,0.4)} node[anchor=east,color=black] {$b$};
  \end{tikzpicture}
  \ \stackrel[\cref{eq:double-dual}]{\cref{rel:f-left-cup}}{=}\
  \begin{tikzpicture}[anchorbase]
    \draw[->] (0,-0.75) -- (0,0.75);
    \redcircle{(0,0.05)};
  \end{tikzpicture}
  \ + \
  \begin{tikzpicture}[anchorbase]
    \draw[->] (0,-0.75) -- (0,0.75);
    \bluedot{(0,-0.4)} node[anchor=east,color=black] {$\chk{b}$};
    \draw (0.6,0.2) node[shape=circle,draw,inner sep=2pt] (char) {$b$};
  \end{tikzpicture}
  \ .
\]
Thus, if $R = \left\{ \circled{$f$} : f \in F \right\}$, then the right curl is equal to the dot in $\Heis_{F,-1}(R)$.  \Cref{theo:alternate-presentation} gives a presentation of $\Heis_{F^\op,-1}$ with generating morphisms $x$, $s$, $c$, $d$, $c'$, $d'$, and $\beta_f$, $f \in F^\op$.  Comparing the relations \cref{rel:token-colide-up,rel:braid-up,rel:doublecross-up,rel:dot-token-up-slide,rel:tokenslide-up-right,rel:dotslide1,rel:right-adjunction-up,rel:right-adjunction-down,theo-eq:doublecross-up-down,theo-eq:doublecross-down-up,theo-eq:right-curl,theo-eq:clockwise-circ,theo-eq:left-curl,theo-eq:counterclockwise-circ} to the defining relations of $\cH_F'$ given in \cite[\S6]{RS17}, we see that there is a strict monoidal functor $\Heis_{F^\op,-1} \to \cH_F'$ sending $\sQ_+$ to $\mathsf{P}$, $\sQ_-$ to $\sQ$, the morphisms $s,c,d,c',d'$ and $\beta_f$, $f \in F$, to the morphisms in $\cH_F'$ represented by the same diagrams, and $x$ to the right curl.  (Note that the apparent sign difference between \cite[(6.13)]{RS17} and \cref{theo-eq:doublecross-down-up} arises from the fact that $\chk{b}$ denotes the \emph{right} dual in \cite{RS17} and that we consider the opposite superalgebra $F^\op$.)  This functor sends $\circled{$f$}$, $f \in F$, to a figure-eight diagram, which is zero since it contains a left curl.  Thus our functor induces a functor from the additive envelope of $\Heis_{F^\op,-1}(R)$ to $\cH_F'$.  This functor is an isomorphism since it has a two-sided inverse.  Precisely, the inverse sends any diagram representing a morphism in $\cH_F'$ to the morphism in the additive envelope of $\Heis_{F^\op,-1}(R)$ represented by the same diagram.  This functor is well-defined since all of the local relations of \cite[\S6]{RS17} hold in $\Heis_{F^\op,-1}(R)$.  The fact that it is a two-sided inverse is clear.  (Here we use the above fact that the right curl is equal to the dot in $\Heis_{F,-1}(R)$.)

\subsection{Proof of \cref{theo:categorification}}

In this subsection, we assume that $\kk$ is an algebraically closed field of characteristic zero.  Since the proofs of the results to follow are very similar to those of \cite{MS17,RS17}, we only provide sketches.  We let $\{s,\epsilon\}X$, $s \in \Z$, $\epsilon \in \Z_2$, denote the shift of an object $X$ in $\Kar \Heis_{F,k}$.

It follows from \cref{rel:braid-up,rel:doublecross-up,rel:dot-token-up-slide,rel:tokenslide-up-right,rel:dotslide1} and \cref{eq:F-to-sQ_-,rel:doublecross-down,rel:braid-down,rel:dot-token-down-slide,eq:tokenslide-downcross1,rel:dotslide-down1} that, for $n \ge 0$, we have graded superalgebra homomorphisms
\begin{equation} \label{eq:affine-wreath-to-end-alg}
  \cA_n(F) \to \END_{\Heis_{F,k}}(\sQ_+^n),\qquad
  \cA_n(F)^\op \to \END_{\Heis_{F,k}}(\sQ_-^n),
\end{equation}
where $x_i$ is mapped to a dot on the $i$-th strand, $f_i$ is mapped to a token labelled $f$ on the $i$-th strand, and $s_i$ is mapped to a crossing of the $i$-th and $(i+1)$-st strands.  Here we number strands \emph{from right to left}.

For $n \ge 1$ and an idempotent $f \in F$, let $e_{f,(n)} = f^{\otimes n} \frac{1}{n!} \sum_{\pi \in S_n} \pi$.  By an abuse of notation, we also use $e_{f,(n)}$ to denote the images of this element under the maps \cref{eq:affine-wreath-to-end-alg}, and we let $\sQ_\pm^{f,(n)} := (\sQ_\pm^n, e_{f,(n)})$ denote the corresponding object in $\Kar \Heis_{F,k}$.  We will denote the idempotents $e_{f,(n)}$ by boxes labelled $f,(n)$:
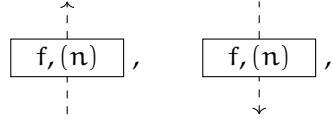
\[
  \begin{tikzpicture}[anchorbase]
    \draw (-0.75,-0.25) rectangle (0.75,0.25) node[midway] {$f,(n)$};
    \draw[dashed] (0,-0.75) -- (0,-0.25);
    \draw[dashed,->] (0,0.25) -- (0,0.75);
  \end{tikzpicture}
  \ ,\qquad
  \begin{tikzpicture}[anchorbase]
    \draw (-0.75,-0.25) rectangle (0.75,0.25) node[midway] {$f,(n)$};
    \draw[dashed,<-] (0,-0.75) -- (0,-0.25);
    \draw[dashed] (0,0.25) -- (0,0.75);
  \end{tikzpicture}
  \ ,
\]
where here, and in what follows, we use undecorated dashed strands to represent multiple parallel strands when the number of strands is clear from the context.  We will sometimes omit the strands emanating from the tops of boxes appearing at the top of a diagram, and similarly for boxes appearing at the bottom of a diagram.

Fix idempotents $f,g \in F$ and an enumeration $b_1,\dotsc,b_M$ of a subset of $B$ such that $\{g \chk{b}_i f : 1 \le i \le M\}$ is a basis for $gFf$.  For $r \ge 1$, let
\[
  B_r^{f,g} = \left\{ \big( (s_1,i_1), \dotsc, (s_r,i_r) \big) \in (\{0,1,\dotsc,k-1\} \times \{1,2,\dotsc,M\})^r : (s_1,i_1) \le (s_2,i_2) \le \dotsb (s_r,i_r) \right\},
\]
where we have used the lexicographic order on pairs of integers.  For $z = \big( (s_1,i_1), \dotsc, (s_r,i_r) \big) \in B_r^{f,g}$, we define
\[
  \begin{tikzpicture}[anchorbase]
    \draw[->,dashed] (0,-0.5) -- (0,0.5);
    \opensquare{(0,0)};
    \draw (0,0) node[anchor=east,color=black] {$z$};
  \end{tikzpicture}
  =
  \begin{tikzpicture}[anchorbase]
    \draw[->] (0,-0.5) -- (0,0.5);
    \redcircle{(0,0.25)} node[anchor=east,color=black] {$s_1$};
    \bluedot{(0,-0.25)} node[anchor=east,color=black] {$\chk{b}_{i_1}$};
    \draw[->] (1,-0.5) -- (1,0.5);
    \redcircle{(1,0.25)} node[anchor=east,color=black] {$s_2$};
    \bluedot{(1,-0.25)} node[anchor=east,color=black] {$\chk{b}_{i_2}$};
    \draw (1.5,0) node {$\cdots$};
    \draw[->] (2.5,-0.5) -- (2.5,0.5);
    \redcircle{(2.5,0.25)} node[anchor=east,color=black] {$s_r$};
    \bluedot{(2.5,-0.25)} node[anchor=east,color=black] {$\chk{b}_{i_r}$};
  \end{tikzpicture}
  ,\qquad
  |z| = \sum_{j=1}^r (|b_{i_j}| + s_j \Delta),\quad
  \bar z = \sum_{j=1}^r \bar{b}_{i_j}.
\]
\details{
  Note that the diagrammatic conventions for monoidal supercategories imply that tokens appearing at the same height yield the same morphism as that obtained by shifting the leftmost token slightly up.  See \cite[(1.3)]{BE17}.
}

For $m,n \ge 1$, $0 \le r \le \min\{m,n\}$, and $z \in B_r^{f,g}$, define
\begin{equation}
  \tau_z =\
  \begin{tikzpicture}[anchorbase]
    \draw (-0.4,-1) rectangle (-2.1,-1.5) node[midway] {$g,(m-r)$};
    \draw (0.4,-1) rectangle (2.1,-1.5) node[midway] {$f,(n-r)$};
    \draw (-0.5,1) rectangle (-2,1.5) node[midway] {$f,(n)$};
    \draw (0.5,1) rectangle (2,1.5) node[midway] {$g,(m)$};
    \draw[->,dashed] (-0.8,1) arc (180:360:0.8);
    \opensquare{(0.72,0.6)};
    \draw{(0.72,0.6)} node[color=black,anchor=east] {$z$};
    \draw[->,dashed] (-1.25,-1) .. controls (-1.25,-0.3) and (1.7,0.3) .. (1.7,1);
    \draw[<-,dashed] (1.25,-1) .. controls (1.25,-0.3) and (-1.7,0.3) .. (-1.7,1);
  \end{tikzpicture}
  \ \colon \{|z|, \bar{z}\} \sQ_+^{g,(m-r)} \otimes \sQ_-^{f,(n-r)} \to \sQ_-^{f,(n)} \otimes \sQ_+^{g,(m)}.
\end{equation}

\begin{lem} \label{lem:commutation-identity-decomp}
  In $\Kar \Heis_{F,k}$, the map
  \[
    \sum_{r=0}^{\min\{m,n\}} \sum_{z \in B_r^{f,g}} \tau_z \colon\
    \bigoplus_{r=0}^{\min\{m,n\}} \bigoplus_{z \in B_r^{f,g}} \{|z|,\bar z\} \sQ_+^{g,(m-r)} \sQ_-^{f,(n-r)} \to \sQ_-^{f,(n)} \sQ_+^{g,(m)}
  \]
  is an isomorphism.
\end{lem}

\begin{proof}
  First, note that it follows from \cref{rel:dotslide1,rel:dotslide-right1,rel:dotslide-down1,eq:dotslide-left1,rel:dotslide2,rel:dotslide-right2,rel:dotslide-down2,eq:dotslide-left2}, that dots slide through crossings modulo diagrams with fewer dots.  Therefore, the arguments of \cite[Prop.~4.2]{MS17} and \cite[Th.~9.2]{RS17} show that there exist morphisms
  \[
    \tau_z' \colon \sQ_-^{f,(n)} \sQ_+^{g,(m)} \to \{|z|, \bar{z}\} \sQ_+^{g,(m-r)} \sQ_-^{f,(n-r)},\qquad
    0 \le r \le \min\{m,n\},\ z \in B_r^{f,g},
  \]
  such that
  \[
    \id_{\sQ_-^{f,(n)} \sQ_+^{g,(m)}} = \sum_{r=0}^{\min\{m,n\}} \sum_{z \in B_r^{f,g}} \tau_z \circ \tau_z'
  \]
  is a decomposition into orthogonal idempotents.  Namely, one uses the defining relations to pull the strands in $\id_{\sQ_-^{f,(n)} \sQ_+^{g,(m)}}$ across each other as in \cite[Prop.~4.2]{MS17} (see also \cite[Lem.~8.1, Cor.~8.2]{BSW1}) to get a decomposition of the identity of the above form.  This shows that the morphism defined by the row vector $[\tau_z]_{0 \le r \le \min\{m,n\},\, z \in B_r^{f,g}}^T$ has a right inverse.  Then one argues as in \cite[Lem.~4.1]{MS17} that it also has a left inverse.
\end{proof}

Let $V_1,\dotsc,V_N$ be a complete list of simple finite-dimensional left $F$-modules, up to grading shift, parity shift, and isomorphism.  Shifting with respect to the $\Z$-grading if necessary, we may assume that the $V_i$ are non-negatively graded, with nonzero degree zero piece.  Recall that a simple module is said to be of type $Q$ if it is evenly isomorphic to its parity shift, and type $M$ otherwise.  After possibly reordering, we assume that
\[
  V_1,\dotsc,V_R \text{ are of type $M$ and } V_{R+1},\dotsc,V_N \text{ are of type $Q$}.
\]
For $1 \le i \le N$, fix an idempotent $e_i \in F$ such that $F e_i$ is the projective cover of $V_i$.

We will view $\kk[x]/(x^k)$ as a graded super vector space by declaring $x$ to be even of degree $\Delta$.  If $V$ is a graded super vector space and $X$ is an object of $\Kar \Heis_{F,k}$, we define
\[
  V \otimes X := \sum_{v \in V} \{|v|,\bar{v}\} X,
\]
where the sum is over a homogenous basis of $V$.  The graded dimension of $V$ is defined to be
\[
  \grdim V = \sum_{n \in \Z} \left( q^n \dim V_{n,0} + q^n \pi \dim V_{n,1} \right)
  \in \Z_{q,\pi},
\]
where $V_{n,\epsilon}$ is the piece of $V$ of degree $n$ and parity $\epsilon$.  For $r \ge 0$, we define
\[
  S^r(V) := V^{\otimes r}/\left\langle v - \pi v : \pi \in S_r,\ v \in V^{\otimes r} \right\rangle.
\]
(Recall that the action of $S_r$ on $V^{\otimes r}$ is via superpermutations.)

\begin{prop}
  Suppose $f,g \in F$ are idempotents, $n,m \ge 1$, and $i \in \{R+1,\dotsc,N\}$.  In $\Kar \Heis_{F,k}$, we have the following isomorphisms:
  \begin{gather}
    \sQ_+^{f,(n)} \sQ_+^{g,(m)} \cong \sQ_+^{g,(m)} \sQ_+^{f,(n)},\qquad
    \sQ_-^{(f,n)} \sQ_-^{g,(m)} \cong \sQ_-^{g,(m)} \sQ_-^{f,(n)}, \label{eq:Qpm-easy-commute}
    \\
    \bigoplus_{r=0}^n \sQ_+^{e_i,(2r)} \sQ_+^{e_i,(2n-2r)} \cong \bigoplus_{r=0}^{n-1} \sQ_+^{e_i,(2r+1)} \sQ_+^{e_i,(2n-2r-1)}, \label{eq:Q-function-isom+}
    \\
    \bigoplus_{r=0}^n \sQ_-^{e_i,(2r)} \sQ_-^{e_i,(2n-2r)} \cong \bigoplus_{r=0}^{n-1} \sQ_-^{e_i,(2r+1)} \sQ_-^{e_i,(2n-2r-1)}, \label{eq:Q-function-isom-}
    \\
    \sQ_-^{f,(n)} \sQ_+^{g,(m)} \cong \bigoplus_{r \ge 0} S^r \left( gFf \otimes \kk[x]/(x^k) \right) \otimes \left( \sQ_+^{g,(m-r)} \sQ_-^{f,(n-r)} \right). \label{eq:Qpm-hard-commute}
  \end{gather}
\end{prop}

\begin{proof}
  It follows from \cref{rel:braid-up,rel:doublecross-up,rel:doublecross-down,rel:braid-down} that symmetrizers slide through crossings when all strands are oriented up or all strands are oriented down.  Thus the relations \cref{eq:Qpm-easy-commute} follow as in \cite[Th.~9.2]{RS17}.
  \details{
    The morphisms
    \[
      \begin{tikzpicture}[anchorbase]
        \draw (-0.25,1) rectangle (-1.35,.5) node[midway] {$f,(n)$};
        \draw (0.25,1) rectangle (1.35,.5) node[midway] {$g,(m)$};
        \draw (-0.25,-1) rectangle (-1.35,-.5) node[midway] {$g,(m)$};
        \draw (0.25,-1) rectangle (1.35,-.5) node[midway] {$f,(n)$};
        \draw[->,dashed] (-0.8,-0.5) .. controls (-0.8,0) and (0.8,0) .. (0.8,0.5);
        \draw[->,dashed] (0.8,-0.5) .. controls (0.8,0) and (-0.8,0) .. (-0.8,0.5);
      \end{tikzpicture}
      \qquad \text{and} \qquad
      \begin{tikzpicture}[anchorbase]
        \draw (-0.25,1) rectangle (-1.35,.5) node[midway] {$g,(m)$};
        \draw (0.25,1) rectangle (1.35,.5) node[midway] {$f,(n)$};
        \draw (-0.25,-1) rectangle (-1.35,-.5) node[midway] {$f,(n)$};
        \draw (0.25,-1) rectangle (1.35,-.5) node[midway] {$g,(m)$};
        \draw[->,dashed] (-0.8,-0.5) .. controls (-0.8,0) and (0.8,0) .. (0.8,0.5);
        \draw[->,dashed] (0.8,-0.5) .. controls (0.8,0) and (-0.8,0) .. (-0.8,0.5);
      \end{tikzpicture}
    \]
    are mutually inverse, giving the first isomorphism in \cref{eq:Qpm-easy-commute}.  The second follows by applying $\omega$.
  }
  Relations \cref{eq:Q-function-isom+,eq:Q-function-isom-} also follow as in \cite[Th.~9.2]{RS17}.  Finally, \cref{eq:Qpm-hard-commute} follows from \cref{lem:commutation-identity-decomp}.
\end{proof}

We can now complete the proof of \cref{theo:categorification}.  Using the local relations, an inductive argument implies that all closed diagrams in $\Heis_{F,k}$ can be written as linear combinations of products of bubbles (clockwise or counterclockwise circles with dots and tokens).  By \cref{eq:inf-grass1,eq:inf-grass2,eq:inf-grass3}, all clockwise bubbles can be written in terms of counterclockwise bubbles.  Then it follows from \cref{eq:inf-grass2} that all closed diagrams in $\Heis_{F,k}$ have nonnegative degree.  Furthermore, the condition on $R$ in the final sentence of \cref{theo:categorification} implies that all closed diagrams of degree zero are scalar multiples of the identity diagram.  Then the theorem follows from the arguments of \cite[\S10]{RS17}.

\subsection*{Acknowledgements}

The author would like to thank J.\ Brundan, J.\ Kujawa, and A.\ Licata for helpful conversations.


\bibliographystyle{alphaurl}
\bibliography{Savage-Frobenius-Heisenberg-categorification}

\end{document}